\documentclass[letterpaper,11pt,oneside,reqno]{amsart}
\usepackage[T1]{fontenc}

\usepackage{amsfonts,amsmath, amssymb,amsthm,amscd}
\usepackage{bm}
\usepackage{mathrsfs}
\usepackage{yfonts}
\usepackage{verbatim}
\usepackage{hyperref}
\usepackage{graphicx}
\usepackage[utf8]{inputenc}
\usepackage{latexsym}
\usepackage{upgreek}
\usepackage{relsize}
\usepackage{xcolor}
\usepackage{mathdots}
\usepackage{mathtools}
\usepackage{dsfont}
\usepackage[left=3cm, right=3cm, top=3cm, bottom=3cm]{geometry}
\usepackage{tikz}
\usepackage{lmodern}
\usepackage{stmaryrd}

\numberwithin{equation}{section}
\renewcommand{\emph}[1]{\textsf{\textit{#1}}}
\usepackage[width=.9\textwidth]{caption}

\makeatletter
\renewcommand\section{\@startsection{section}{1}%
	\z@{1.5\linespacing\@plus.5\linespacing}{1\linespacing}%
	{\large\normalfont\scshape\centering}}
\makeatother

\newcommand{\PP}{\ensuremath{\mathbb{P}}}
\newcommand{\R}{\ensuremath{\mathbb{R}}}
\newcommand{\C}{\ensuremath{\mathbb{C}}}
\newcommand{\Z}{\ensuremath{\mathbb{Z}}}

\newcommand{\I}{\ensuremath{\mathbf{i}}}
\newcommand{\Real}{\ensuremath{\mathfrak{Re}}}

\newcommand{\la}{\lambda}

\renewcommand{\rho}{\varrho}

\newcommand{\eps}{\varepsilon}
\renewcommand{\leq}{\leqslant}
\renewcommand{\geq}{\geqslant}

\newcommand{\Sign}{\mathsf{Sign}}
\newcommand{\Signplus}{\Sign^+}
\newcommand{\Zstrip}{\mathbb{Z}^{\rm strip}}
\newcommand{\Geom}{\mathrm{Geom}}
\newcommand{\Gammainv}{\mathrm{Gamma}^{-1}}

\usetikzlibrary{patterns}
\usetikzlibrary{shapes.multipart}
\usetikzlibrary{arrows}

\newtheorem{theorem}{Theorem}[section]

\newtheorem{lemma}[theorem]{Lemma}
\newtheorem{proposition}[theorem]{Proposition}
\newtheorem{corollary}[theorem]{Corollary}

\theoremstyle{definition}
\newtheorem{remark}[theorem]{Remark}

\theoremstyle{definition}

\theoremstyle{definition}
\newtheorem{definition}[theorem]{Definition}

\theoremstyle{definition}

\title[Integral formulas for two-layer Schur and Whittaker processes]{Integral formulas for two-layer Schur and Whittaker processes}

\author[G.~Barraquand]{Guillaume Barraquand}
\address{Laboratoire de Physique de l'Ecole Normale Supérieure, Ecole Normale Supérieure, PSL University, CNRS, Sorbonne Université, Université Paris-Cité, 24 rue Lhomond, 75005 PARIS}
\email{guillaume.barraquand@math.cnrs.fr} 

\begin{document}
		\begin{abstract}
Stationary measures of last passage percolation with geometric weights and the log-gamma polymer in a strip of the $\mathbb Z^2$ lattice are characterized in \cite{barraquand2023stationary} using variants of Schur and Whittaker processes, called two-layer Gibbs measures. In this article, we prove contour integral formulas characterizing the multipoint  joint distribution of two-layer Schur and Whittaker processes.  We also express them  as Doob transformed Markov processes with explicit transition kernels. As an example of application of our formulas, we compute the growth rate of the KPZ equation on $[0,L]$ with arbitrary boundary parameters.
			\end{abstract}

	\maketitle
	\setcounter{tocdepth}{1}
	\tableofcontents


\section{Introduction}
\label{sec:introduction}
\subsection{Preface}
Integrable random systems defined on unbounded domains often admit particularly simple stationary measures, described  by families of independent and identically distributed  random variables. In finite volume with open boundary conditions, however, stationary measures are typically more complicated,  with spatial correlations and a non-trivial dependence on boundary conditions.
The archetype  example is the asymmetric simple exclusion process on $\ell$ sites, connected to boundary reservoirs (open ASEP). Computing its stationary measure boils down to finding the principal eigenvector of a rather simple Markov matrix, indexed by elements of the state space $\lbrace 0,1\rbrace^\ell$. But this is not easy. In 1993, \cite{derrida1993exact} introduced an elegant method, called the matrix product ansatz, to determine the stationary measure of open ASEP. This method is remarkably powerful and has been extended in various directions over the years. Yet, the matrix product ansatz  becomes increasingly difficult to apply on models with a larger state space, for instance  a state space such as $\mathbb Z^\ell$ or $\R^\ell$. 

\medskip 
In order to characterize stationary measures of geometric last passage percolation and the log-gamma polymer in a strip of the $\mathbb Z^2$ lattice (i.e. paths are constrained to stay between two walls), another method was introduced in \cite{barraquand2023stationary}. It relies on connections between these models and families of symmetric functions, namely the Schur functions for last passage percolation, and Whittaker functions for the log-gamma polymer model. 

\medskip
Stationary measures of out of equilibrium systems in finite domains are typically not expected to be Gibbs measures. Yet the method of  \cite{barraquand2023stationary} expresses them as marginals of a Gibbs measure on a larger graph. The structure of this graph and the expression of the Boltzmann  weights are closely related  to the branching rule satisfied by the corresponding family of symmetric functions. These measures, called two-layer Gibbs measures in \cite{barraquand2023stationary},  can also be seen as variants of Schur or Whittaker processes \cite{okounkov2003correlation, borodin2005eynard,  corwin2014tropical, borodin2014macdonald, betea2018free}. While the matrix product ansatz characterizes stationary measures through formulas, two-layer Gibbs measures yield directly a probabilistic description, which is suitable to taking scaling limits, notably to the Kardar-Parisi-Zhang (KPZ) equation. The goal of the present paper is to complement this qualitative and probabilistic description  with more quantitative information.  

\medskip 
Our results are very similar for the two models: geometric last passage percolation and the log-gamma polymer in a strip (defined in Section \ref{sec:defmodels}). 
\begin{enumerate}
	\item We prove explicit contour integral formulas for the multipoint Laplace transform of the stationary process. Our results have a similar form as formulas for the stationary measures   of open ASEP \cite{bryc2017asymmetric} and  the open KPZ equation \cite{corwin2021stationary}.
	\item We provide alternative descriptions of the two-layer Gibbs measures introduced in \cite{barraquand2023stationary} in terms of explicit Markov processes with  well-chosen initial condition. These results may be seen as discrete analogues of \cite{bryc2021markov} (they also have similarities with \cite{das2024convergence}).
\end{enumerate}

\medskip 

 We describe in  Section \ref{sec:KPZ} an application of our exact formulas to the KPZ equation. Assuming the convergence of the log-gamma polymer free energy on a strip to the open KPZ equation (a precise statement is conjectured in \cite{barraquand2023stationary}), we obtain that the stationary solution $h(t,x)$ of the KPZ equation on $[0,L]$ with boundary parameters $u,v\in \mathbb R$ is such that 
$$ \mathbb E\left[h(t,0)\right] = t\left( \frac{-1}{24}+ \frac{1}{2} \partial_L \log \mathcal Z_{u,v}(L) \right) ,$$
where
$$
\mathcal Z_{u,v}(L) = \int_{\I\R}\frac{dz}{2\I\pi} \left\vert \frac{\Gamma(u+ z)\Gamma(v+ z) }{\Gamma(2z)}\right\vert^2 \frac{e^{z^2 L}}{2}.
$$
The formulas above hold only when $u,v>0$ but they may be  analytically continued to study other values of  $u$ or $v$. In the large $L$ limit, we obtain a phase transition with three phases (Corollary \ref{cor:phasetransition}). The phase diagram is of course the same as that of open ASEP \cite{derrida1993exact}, as predicted by universality. 

\bigskip

As an  application of the Markovian description, we discuss limits of stationary measures on a finite domain as the size of the domain goes to infinity.  As shown in \cite{bryc2021markov2} in the context of the open KPZ equation, a Markovian description is particularly convenient for such task. We prove that two-layer Schur and Whittaker processes can be seen as Markov processes, more precisely Doob transforms of killed random walks.   We establish  exact formulas for the transition kernels, and explain how they can be used to study the convergence of stationary measures on a strip as the width of the strip goes to infinity.

\bigskip 
The method of \cite{barraquand2023stationary} is expected to apply as well to other models related to other families of symmetric functions. While we restrict here to the cases studied in \cite{barraquand2023stationary}, we believe that the methods used in the present paper could be adapted to obtain integral formulas for two-layer Gibbs measures related to other families of symmetric functions such as (spin)-$q$ Whittaker and (spin)-Hall-Littlewood functions. 

\subsection{Definitions of the models}
\label{sec:defmodels}
We consider stochastic growth models defined on a strip of the $\Z^2$ lattice denoted  $\Zstrip=\left\lbrace (i,j)\in \Z^2; j\leq i\leq j+N\right\rbrace$. 
\begin{figure}[h]
	\begin{center}
	\begin{tikzpicture}[scale=0.7]
		\begin{scope} 
			\clip (0,0) -- (6,0) -- (11.5,5.5) -- (5.5,5.5) -- cycle;
			\draw[gray] (0,0) grid (12,12);
		\end{scope}
		\draw[gray, ultra thick] (0,0) node[anchor=north east] {$(0,0)$}   -- (6,0);
		\draw (4,0) node[below] {$(i,0)$};
		\draw[ultra thick] (4,0) -- (4,2) -- (6,2) -- (6,4) -- (7,4);
		\draw (7,4) node[above right] {$(n,m)$}; 
		\draw (6,2) node[below right] {$\pi$}; 
	\end{tikzpicture}
\end{center}
\caption{The subset $\Zstrip$ of the $\mathbb Z^2$ lattice. In black is shown an example of up-right path as used in the definitions of last passage percolation and the log-gamma polymer in the strip.}
\label{fig:strip}
\end{figure}

\begin{definition}[Last passage percolation on the strip] Let $a\in (0,1)$ and $c_1,c_2\geq 0$ such that $ac_1, ac_2<1$. We associate a family of independent random variables $w_{i,j}$ to the vertices of $\Zstrip$. We assume that in the bulk of the strip, i.e.  for $j<i<j+N$, $w_{i,j} \sim \Geom(a^2)$, while on the left boundary $w_{i,i}\sim \Geom(ac_1)$ and on the right boundary, $w_{j+N,j}\sim \Geom(ac_2)$. 
For a given initial condition $G_0: \llbracket 0,N\rrbracket \to \R$ with $G(0)=0$, we set $G(i,0)=G_0(i)$ and define 
\begin{equation}
G(n,m) = \max_{1\leq i\leq N} \left\lbrace G_0(i) + \max_{\pi:(i,1)\to (n,m)} \sum_{(i,j)\in \pi} w_{i,j} \right\rbrace,
\end{equation}
where the inner maximum runs over up-right paths $\pi$ from $(i,1)$ to $(n,m)$, as shown in Fig. \ref{fig:strip}. 
We may then define the Markov process $G_m=\left(G_m(i)\right)_{1\leq i\leq N}$ on $\mathbb Z^{N}$  by setting  
$$G_m(i)=G(m+i,m)-G(m,m)\text{ for }1\leq i\leq N.$$
We say that a probability measure $\PP^{\rm stat}$ on $\Z^N$ is a stationary measure for LPP on the strip if $G_0\sim \PP^{\rm stat}$ implies that $G_m\sim\PP^{\rm stat}$ for all $m\geq 0$.   
\label{def:LPP}
\end{definition}

\begin{definition}[Log-gamma polymer model] Let $\alpha>0$ and $u,v\in \R$ such that $\alpha+u, \alpha+v>0$. We consider another family of independent random variables $\omega_{i,j}$ such that for $j<i<j+N$, $\omega_{i,j} \sim \Gammainv(2\alpha)$, on the left boundary $\omega_{i,i}\sim \Gammainv(\alpha+u)$
	 and on the right boundary, $\omega_{j+N,j}\sim \Gammainv(\alpha+v)$ (here, $\Gammainv(\theta)$ denotes the inverse Gamma distribution with shape parameter $\theta$). 
	For a given initial condition $H_0: \llbracket 0,N\rrbracket \to \R$ with $H(0)=0$, we set $H(i,0)=H_0(i)$ and define 
	\begin{equation}
		H(n,m) = \log\left( \sum_{1\leq i\leq N} e^{ H_0(i)}  \sum_{\pi:(i,1)\to (n,m)} \prod_{(i,j)\in \pi} w_{i,j} \right).
	\end{equation}
	We may then define the Markov process $H_m=\left(H_m(i)\right)_{1\leq i\leq N}$ on $\mathbb R^{N}$ as before by setting  
	$$H_m(i)=H(m+i,m)-H(m,m)\text{ for }1\leq i\leq N$$
	 and consider its stationary measures. 
	\label{def:loggamma}
\end{definition}

\begin{remark}
One can define both models in a more general way, by imposing an initial condition on an arbitrary down-right path joining both sides of the strip, instead of the horizontal path in Fig. \ref{fig:strip}. One may then consider a more general notion of stationary measure for each shape of down right path. We refer to \cite{barraquand2023stationary} for definitions and results in this more general setting. See also Definition \ref{def:generaltwolayerSchur} and Definition \ref{def:generaltwolayerWhittaker} below. We believe that all the methods used in the present paper are  applicable to this more general setting, though notations and formulas would become more involved. This is why we decided to restrict to horizontal paths for most results (with the exception of Proposition \ref{prop:Markoviangeneral} and Proposition \ref{prop:MarkoviangeneralWhittaker}).  
\end{remark}
\begin{remark}
One can also define more general models depending on a family of inhomogeneity parameters $a_1, \dots a_N$ in the LPP case, and $\alpha_1, \dots, \alpha_N$ in the log-gamma case. Our results in Sections \ref{sec:Schur} and Section \ref{sec:Whittaker} are actually stated in this more general setting, though in this introduction, we present only results for the homogeneous case where $a_i\equiv a$ and $\alpha_i\equiv \alpha$. 
\end{remark}

\subsection{Probabilistic description of stationary measures}
For both Markov processes $G_m $ and $H_m$, and for any choice of parameters, an explicit description of their unique ergodic stationary measure is given in  \cite{barraquand2023stationary}. We recall these results in the next two propositions.  
\begin{definition} Under the same assumptions on parameters $a, c_1, c_2$ as in Definition \ref{def:LPP}, 
	let us define a probability distribution on couples of integer-valued random walks $\mathbf L_1 = \big(L_1(i)\big)_{0\leq i\leq N}$ and $\mathbf L_2 = \big(L_2(i)\big)_{0\leq i\leq N}$ by 
	\begin{equation}
		\PP^{a,c_1,c_2}_{\rm Geo}(\mathbf L_1,\mathbf L_2)= \frac{1}{\mathcal K^{a,c_1,c_2}_{\rm Geo}}\,\,(c_1c_2)^{-\mathbf L_1 \otimes \mathbf L_2(N)}   \,\, \PP_{\rm Geo}^{ac_2}(\mathbf L_1)\,  \PP_{\rm Geo}^{ac_1}(\mathbf L_2),
		\label{eq:defstationarymeasureLPPintro}
	\end{equation}
	where the reference measure $\PP_{\rm Geo}^{q}$ is the law of geometric random walks with parameter $q$, starting from $L_1(0)=L_2(0)=0$, i.e. 
	\begin{equation*}
		\PP_{\rm Geo}^{q}(\mathbf L) = \prod_{i=1}^N q^{\mathbf L(i)-\mathbf L(i-1)}(1-q),
	\end{equation*}
	and the composition $\mathbf L_1\otimes \mathbf L_2$ is the discrete Pitman transform  
	$$ \mathbf L_1\otimes \mathbf L_2 (k) =  \min_{1\leq j\leq k} \left\lbrace L_1(j-1)+ L_2(k)-L_2(j)\right\rbrace.$$ 
	\label{def:stationaryGeo}
\end{definition}
\begin{proposition}[{\cite[Theorem 1.3]{barraquand2023stationary}}]
The marginal distribution of $\mathbf L_1$ under the probability measure $\PP^{a,c_1,c_2}_{\rm Geo}$ is the unique stationary measure of the process $G_m$ from Definition \ref{def:LPP}. 
\end{proposition}
\begin{remark}
Definition \ref{def:stationaryGeo} is a reformulation of \cite[Definition 1.2]{barraquand2023stationary}. We chose this rewriting to make the Pitman transform appear. Following notations from \cite{o2002representation}, we could have defined the composition 
$$ \mathbf L_1 \odot \mathbf L_2(k) = \max_{1\leq j\leq k} \lbrace L_1(j)+L_2(k)-L_2(j-1)\rbrace.$$
The discrete  Pitman transform is then generally defined as the map $(\mathbf L_1, \mathbf L_2) \mapsto (\mathbf L_1\otimes \mathbf L_2,  \mathbf L_2 \odot \mathbf L_1)$. Using  that $\mathbf L_1+\mathbf L_2=\mathbf L_1\otimes \mathbf L_2 + \mathbf L_2 \odot \mathbf L_1$, we could have defined $\PP^{a,c_1,c_2}_{\rm Geo}$ as 
\begin{equation}
	\PP^{a,c_1,c_2}_{\rm Geo}(\mathbf L_1,\mathbf L_2)= \frac{1}{\widetilde{\mathcal K}^{a,c_1,c_2}_{\rm Geo}}\,\,(c_1c_2)^{\mathbf L_2 \odot \mathbf L_1(N)}   \,\, \PP_{\rm Geo}^{a/c_1}(\mathbf L_1)\,  \PP_{\rm Geo}^{a/c_2}(\mathbf L_2),
	\label{eq:defstationarymeasureLPPintro-alternative}
\end{equation}
where $\widetilde{\mathcal K}^{a,c_1,c_2}_{\rm Geo}$ is another normalization constant. Since we do not assume that $a/c_1$ and $a/c_2$ are in $[0,1)$, we prefer using \eqref{eq:defstationarymeasureLPPintro} than  \eqref{eq:defstationarymeasureLPPintro-alternative}.
\end{remark}

We now turn to the log-gamma polymer. 
\begin{definition} Under the same assumptions on parameters as in Definition \ref{def:loggamma}, let us define a probability density for real-valued random walks $\mathbf L_1, \mathbf L_2$ by 
	\begin{equation*}
		\PP^{\alpha,u,v}_{\rm LG}(\mathbf L_1,\mathbf L_2)= \frac{1}{\mathcal K^{\alpha,u,v}_{\rm LG}}\,\,e^{(u+v) \mathbf L_1 \otimes \mathbf L_2(N)}   \PP_{\rm LG}^{\alpha+v}(\mathbf L_1) \PP_{\rm LG}^{\alpha+u}(\mathbf L_2),
	\end{equation*}
	where now the reference measure $\PP_{\rm LG}^{\alpha}$ is the law of log-gamma increment random walk with parameter $\alpha$, starting from $L_1(0)=L_2(0)=0$, i.e. with density 
	\begin{equation}
		\PP_{\rm LG}^{\alpha} (\mathbf L) = \prod_{i=1}^N \frac{1}{\Gamma(\alpha)} \exp\left( -\alpha(\mathbf L(i)-\mathbf L(i-1)) -e^{-(\mathbf L(i)-\mathbf L(i-1))}    \right),
	\end{equation}
	and, with a slight abuse of notations,  the composition $\otimes$ is now the discrete geometric\footnote{Here, geometric refers to the  geometric Brownian motion  appearing in the study of exponential functionals of the Brownian motion, where positive temperature analogues of the Pitman transform where first studied \cite{matsumoto2000analogue}, or geometric random walks appearing in geometric liftings of the RSK correspondence \cite{noumi2004tropical, corwin2014tropical}.} Pitman transform  defined as 
	$$ \mathbf L_1\otimes \mathbf L_2 (k) =  -\log \left( \sum_{j=1}^k e^{-(L_1(j-1)+ L_2(k)-L_2(j))}\right).$$ 
	\label{def:stationaryloggamma}
\end{definition}
\begin{proposition}[{\cite[Theorem 1.6]{barraquand2023stationary}}]
	The marginal distribution of $\mathbf L_1$ under the probability density  $\PP^{\alpha,u,v}_{\rm LG}$ is the unique stationary measure of the process $H_m$ from Definition \ref{def:loggamma}. 
\end{proposition}

\subsection{Main results} We start with the case of last passage percolation. 
For integers $0= x_0< x_1< \dots< x_k= N$  the joint distribution of the increments $\mathbf L_1(x_i)-\mathbf L_1(x_{i-1})$ for $1\leq i\leq N$ is characterized by the generating function 
\begin{equation*}
\mathcal L_{\rm Geo}^{a,c_1,c_2}(\mathbf t, \mathbf x):= \mathbb E_{\rm Geo}^{a, c_1, c_2}\left[\prod_{i=1}^k t_i^{2(\mathbf L_1(x_i)-\mathbf L_1(x_{i-1}))} \right].
\end{equation*} 
Let us introduce the notations 
\begin{equation}
	\mathfrak{m}_{\rm Geo}(dz) = \frac{(1-z^2)(1-1/z^2)}{2z} , \;\; \phi(z,a) = \phi_{\rm Geo}(z;a) = \frac{1}{1-z a},
	\label{eq:defmphigeo}
\end{equation}
and define the kernel 
\begin{equation} 
	K_{\rm Geo}(t_1,t_2; z_1,z_2) = \sqrt{1-\left(\frac{t_1}{t_2}\right)^2}\phi_{\rm Geo}(z_2z_1, t_1/t_2)\phi_{\rm Geo}(z_2/z_1, t_1/t_2).
	\label{eq:defKgeo}
\end{equation}
\begin{theorem} Let $a, c_1, c_2$ be parameters as in Definition \ref{def:LPP}. Let $t_1,\dots, t_k$ be real numbers such that  $ c_1<t_1<\dots <t_k<  \frac{1}{c_2}$ and $at_k<1$. For  $c_1,c_2<1$, we have the formula 
		\begin{multline}  	\mathcal L_{\rm Geo}^{a,c_1,c_2}(\mathbf t, \mathbf x) =    \frac{1}{\mathcal Z_{\rm Geo}^{a, c_1,c_2}(N)}  \oint\frac{\mathfrak m_{\rm Geo}(dz_1)}{2\I\pi}\dots  \oint\frac{\mathfrak m_{\rm Geo}(dz_k)}{2\I\pi}
		 \left\vert \prod_{i=1}^k  \phi(z_i, at_i)^{x_i-x_{i-1}} \right\vert^2
		\\ \times 
		\left\vert  \phi(z_1, c_1/t_1)\phi(z_k, c_2 t_k) \prod_{i=1}^{k-1} K_{\rm Geo}(t_i, t_{i+1}; z_i, z_{i+1}) \right\vert^2, 
		\label{eq:Laplacetransformintro}
	\end{multline}
where the integration contours  are circles around $0$ with radius $1$ (for a variable $z$ on such contour, and a rational function $f$ with real coefficients, $\vert f(z)\vert^2=f(z)f(1/z)$, so that the integral in \eqref{eq:Laplacetransformintro} is a complex contour integral in the usual sense, and it can be evaluated by residues). Throughout the paper, all circular contours will be positively oriented.
 The normalization constant can be computed explicitly as 
	\begin{equation}
	\mathcal Z_{\rm Geo}^{a, c_1,c_2}(N) = 	 \oint \mathfrak m_{\rm Geo}(dz) \left\vert  \phi(z, c_1)\phi(z,c_2)(\phi(z,a))^N\right\vert^2.
	\label{eq:noramlizationLPPintro}
\end{equation}

Finally, for any $c_1, c_2$ as in Definition \ref{def:LPP} but not satisfying the assumptions above, if the $t_i$ are such that $ t_1<\dots <t_k$ with $at_k<1$ and $at_k^2 c_2<1$,   the formula \eqref{eq:Laplacetransformintro} still holds after replacing the right-hand side of \eqref{eq:Laplacetransformintro} by its analytic continuation in parameters $c_1$ and/or $c_2$. 
	\label{theo:formulaLPPintro}
\end{theorem}	
Theorem \ref{theo:formulaLPPintro} is proved in Section \ref{sec:proofLPP}. In its statement,  and in several times below, we will deal with analytic functions in two  variables. We say that $f(x,y)$ is (real)-analytic in variables  $x,y$, for $(x,y)\in \mathcal D$ for some domain $\mathcal D\subset \R^2$ if for any $(x_0,y_0)\in \mathcal D$, the function $x\mapsto f(x,y_0)$ is analytic at $x_0$ and the function $y\mapsto f(x_0,y)$ is analytic at $y_0$. When we consider the analytic  continuation of a function $f(x,y)$ to a larger domain $\mathcal D'\supset \mathcal D$, we will consider it to be the analytic continuation with respect to one variable first, and the other then. In the cases discussed below, the resulting expression will not depend on the order in which these analytic continuations are performed.

In the simplest case, that is for $k=1$, Theorem \ref{theo:formulaLPPintro} implies that for $c_1,c_2<1$ 	with  $ c_1 <t< \min \lbrace  \frac{1}{c_2}, \frac{1}{a} \rbrace $, 
	\begin{equation*}
		\mathbb E_{\rm Geo}^{a, c_1, c_2}\left[ t^{2\mathbf L_1(N)} \right]=   \frac{1}{\mathcal Z_{\rm Geo}^{a, c_1,c_2}(N)}   \oint \frac{dz}{ 2\I\pi z} \frac{(1-z^2)(1-1/z^2)}{2\left( (1-t a z)(1-t a/ z) \right)^N}  \frac{1}{1-\frac{zc_1}{t}}\frac{1}{1-\frac{c_1}{zt}}\frac{1}{1-\frac{c_2t}{z}}\frac{1}{1-c_2tz}. 
	\end{equation*}

We now state an analogous result for the log-gamma polymer stationary measure, more precisely for the joint Laplace transform of its increments 
\begin{equation*}
	\mathcal L_{\rm LG}^{\alpha, u,v}(\mathbf t, \mathbf x) = \mathbb E_{\rm LG}^{\alpha, u, v}\left[\prod_{i=1}^k e^{-2 t_i(\mathbf L_1(x_i)-\mathbf L_1(x_{i-1}))} \right].
\end{equation*}
Let us redefine the function $\phi$ as 
$$\phi(z, \alpha) = \phi_{\rm LG}(z, \alpha) = \Gamma(z+\alpha), $$ and use the notations 
$$ \mathfrak m_{\rm LG}(dz) =  \frac{1}{2\Gamma(2z)\Gamma(-2z)}, \,\,\, K_{\rm LG}(t_1, t_{2}; z_1, z_{2})  = \frac{\phi(z_2+z_1; t_1-t_2)\phi(z_2+z_1; t_1-t_2)}{\sqrt{\Gamma(2t_1-2t_2)}}.$$   
\begin{theorem} 
	Let $\alpha, u,v$ be parameters as in Definition \ref{def:loggamma}. 
	   Let $t_1, \dots, t_k$ be real numbers such that $u>t_1>\dots >t_k>-v$ and $\alpha+t_k>0$. For $u,v>0$, we have
	   \begin{multline}  	\mathcal L_{\rm LG}^{\alpha, u,v}(\mathbf t, \mathbf x) =   \frac{1}{\mathcal Z_{\rm LG}^{\alpha, u,v}(N)} \int_{\I\R}\frac{\mathfrak m_{\rm LG}(dz_1)}{2\I\pi} \dots  \int_{\I\R}\frac{\mathfrak m_{\rm LG}(dz_k)}{2\I\pi}  
	   	\left\vert \prod_{i=1}^k  \phi(z_i, \alpha + t_i)^{x_i-x_{i-1}} \right\vert^2
	   	\\ \times 
	   	\left\vert  \phi(z_1, u-t_1)\phi(z_k, v+t_k) \prod_{i=1}^{k-1} K_{\rm LG}(t_i, t_{i+1}; z_i, z_{i+1}) \right\vert^2,
	   	\label{eq:LaplacetransformWhittakerintro}
	   \end{multline}
	   where we recall that for $z\in \I\R$, $\vert f(z)\vert^2 = f(z)f(-z)$. 
	   
  The normalization constant is given by 
	   \begin{align}
	   	\mathcal Z_{\rm LG}^{\alpha, u,v}(N) &= \int_{\I\R} \frac{\mathfrak m_{\rm LG}(dz)}{2\I\pi} \left\vert  \phi(z, c_1)\phi(z,c_2)(\phi(z,a))^N\right\vert^2 \nonumber\\ 
	   	&=\int_{\I\R}\frac{dz}{2\I\pi}\frac{\Gamma(u\pm z)\Gamma(v\pm z) \left(\Gamma(\alpha\pm z)\right)^N}{2\Gamma(\pm 2z)}.
	   	\label{eq:normalizationloggammaintro}
	   \end{align}
	   In the second line, we use the notational convention that each time that $\pm$ appears in the formula, we take the product of the Gamma function with a plus sign and the Gamma function with a minus sign.

Moreover, for any  $u,v$ such that $u+\alpha, v+\alpha>0$, if we assume that $t_1>\dots >t_k$ with $2t_k>-\alpha+ \max\lbrace \vert u\vert , \vert v\vert\rbrace$, then  the formula \eqref{eq:Laplacetransformintro} still holds after replacing the right-hand side of \eqref{eq:LaplacetransformWhittakerintro} by its analytic continuation in parameters $u$ and/or $v$. 
	\label{theo:formulaloggammaintro}
\end{theorem}
Theorem \ref{theo:formulaloggammaintro} is proved in Section \ref{sec:proofWhittaker}. To see why \eqref{eq:normalizationloggammaintro} is a convergent integral, one can use the fact that Gamma factors decay exponentially as  $\vert \Gamma(a+\I x)\vert \leq  Ce^{-\pi \vert x\vert /2} $ for fixed $a>0$,  so that the integral is convergent as long as $N\geq 1$.  In the simplest case, that is when  $k=1$, Theorem \ref{theo:formulaloggammaintro} implies that  for $u>t>\max\lbrace -v, -\alpha\rbrace$, 
\begin{equation}  	\mathbb E_{\rm LG}^{\alpha, u,v}\left[ e^{-2 t \mathbf L_1(N)} \right] =   \frac{1}{\mathcal Z_{\rm LG}^{\alpha, u,v}(N)} \int_{\I\R}\frac{dz}{2\I\pi} \frac{\Gamma(u-t\pm z)\Gamma(v+t\pm z)\left(\Gamma(\alpha+t\pm z) \right)^N}{2\Gamma(2z)\Gamma(-2z)}.
	\label{eq:Laplacetransformspecialcaseintro} 
\end{equation}

\begin{remark} Using the change of variables $z_i=e^{\I \pi\theta_i}$ with $ \theta_i\in (-1,1]$ in \eqref{eq:Laplacetransformintro}  and interpreting the $\theta_i$ as successive values of a Markov process, the Laplace transform in \eqref{eq:Laplacetransformintro} can be rewritten in terms of the free\footnote{here, the adjective free refers to the fact that it is the $q=0$ case of the more general Askey-Wilson process, as the free bi-Poisson process from \cite{bryc2007bi} is a $q=0$ case of more general bi-Poisson processes \cite{bryc2008bi}.} Askey-Wilson process. Askey-Wilson processes are stochastic processes taking values in $[-1,1]$,  introduced in \cite{bryc2010askey}  in the context of quadratic harnesses \cite{bryc2007quadratic}\footnote{Strictly speaking, Askey-Wilson processes take values in $[-1,1]$ only for range of times and parameters such  their transition density has no atoms. In general, other values  may arise.}. Their transition density is expressed in terms of the orthogonality measure for Askey-Wilson orthogonal polynomials, hence the name. Askey-Wilson processes also arises in Laplace transform formulas characterizing the stationary measure of open ASEP \cite{bryc2017asymmetric}. The limit of these formulas when ASEP scales to the KPZ equation was studied in \cite{corwin2021stationary}, which led to study a limit of Askey-Wilson processes called continuous dual Hahn processes, also considered in \cite{bryc2012quadratic} (see also \cite{bryc2022continuous}). 
Letting $z_i=\pm \I \sqrt{x_i}/2$ in \eqref{eq:LaplacetransformWhittakerintro}, we also notice that the formula \eqref{eq:LaplacetransformWhittakerintro} can be rewritten as a functional of the continuous dual Hahn process. We believe that the relations between two-layer Gibbs measures and Askey-Wilson processes deserve to be studied more systematically. 
\label{rem:AskeyWilson}
\end{remark}
\begin{remark}
	Besides the matrix product ansatz and the method of \cite{barraquand2023stationary},  there exist other methods to compute stationary measures  of finite volume models on various state spaces. There exists  notably a body of work using Markov duality to compute stationary measures for a variety of symmetric models  \cite{frassek2022exact, carinci2024solvable, giardina2024intertwining}.
	It would be interesting to study whether such Markov dualities can be related to the type of dual processes discussed in Remark \ref{rem:AskeyWilson}. Let us also mention that for integrable models in a  half-space or on a ring, there exist yet other methods to compute the stationary measures, based on symmetries with respect to permuting inhomogeneity parameters \cite{barraquand2023stationaryloggamma, aggarwal2023colored} or other ideas \cite{he2023periodic, corwin2024periodic}. 
\end{remark}
\begin{remark}
	In the course of proving Theorem \ref{theo:formulaLPPintro} and Theorem \ref{theo:formulaloggammaintro}, we encountered seemingly new summation (resp. integral) identities for Schur (resp. Whittaker) functions: Lemma \ref{lem:specialcomputation} and Lemma \ref{lem:specialcomputationWhittaker}. Presently, these identities are valid only for signatures of length $n=2$,  and it is not clear if they admit interesting generalizations to arbitrary $n$. 
\end{remark}

\subsection{Two-layer Schur and Whittaker processes}
\label{sec:two-layer-intro}
Our results are based on another description of stationary processes from Definition \ref{def:stationaryGeo} and Definition \ref{def:stationaryloggamma}: the two-layer Schur and Whittaker processes. These probability measures are introduced in \cite{barraquand2023stationary} as Gibbs measures defined on certain graphs. For the purposes of the present paper, it is more convenient to define them in terms of Schur and Whittaker functions, and we will not use the Gibbsian definition. The definitions given in the next paragraphs are special cases of Definition \ref{def:generaltwolayerSchur} and Definition \ref{def:generaltwolayerWhittaker}.

\subsubsection{Two-layer Schur process}
Consider an integer $N\geq 1$, and parameters $a, c_1,c_2$ as in Definition \ref{def:LPP} with the additional assumption that $c_1c_2<1$. Let $\Sign_2$ denote the set of $(\la_1,\la_2)\in\mathbb Z^2$ with $\la_1\geq \la_2$. Define a probability measure on sequences  $\boldsymbol\lambda=(\lambda^{0}, \dots, \lambda^{N})$ of elements $\lambda^i\in  \Sign_2$ such that
\begin{equation}
	\mathbb P(\boldsymbol{\lambda}) = \frac{1}{\mathcal Z_{\rm Geo}^{a,c_1,c_2}(N)} c_1^{\lambda_1^0-\lambda_2^0}c_2^{\lambda_1^N-\lambda_2^N} \prod_{i=1}^N s_{\lambda^i/\lambda^{i-1} }(a)\;\; \mathds{1}_{\lambda_2^0=0},
	\label{eq:introtwolyaerSchur}
\end{equation} 
where $s_{\lambda/\mu}$ is a skew Schur function defined below in \eqref{eq:defskewSchur}. We provide some background on Schur functions in  Section \ref{sec:Schurbackground}. Then, as proved in \cite[Prop. 2.17]{barraquand2023stationary}, we have the equality in law 
\begin{equation}
	\left( L_1(x), L_2(x)\right)_{0\leq x\leq N} \overset{(d)}{=} \left( \la_1^x-\la_1^0, \la_2^x-\la_2^0 \right)_{0\leq x\leq N}, 
	\label{eq:twodescriptions}
\end{equation}
where $\mathbf L_1, \mathbf L_2$ are distributed as in Definition \ref{def:stationaryGeo} (the  equality is shown by explicitly averaging over $\la_1^0$ in \eqref{eq:introtwolyaerSchur}).

\subsubsection{Two-layer Whittaker process} 
Consider an integer $N\geq 1$, and parameters $\alpha,u,v$ as in Definition \ref{def:loggamma} with the additional assumption that $u+v>0$. We consider the probability measure on sequences  $\boldsymbol\lambda=(\lambda^{0}, \dots, \lambda^{N}) \in (\R^2)^{N+1}$ with density 
\begin{equation}
	\mathbb P(\boldsymbol{\lambda}) = \frac{1}{\mathcal Z^{\alpha, u,v}_{\rm LG}(N)} e^{-u(\lambda_1^0-\lambda_2^0)}e^{-v(\lambda_1^N-\lambda_2^N)} \prod_{i=1}^N \Psi_{\alpha}(\lambda^i/\lambda^{i-1}) \delta_0(\lambda_2^0).
	\label{eq:introtwolayerWhittaker}
\end{equation} 
where for $\lambda, \mu \in \mathbb R^2$ and $(\alpha_1, \dots, \alpha_k)\in \mathbb C^k$, $\Psi_{\alpha_1,\dots, \alpha_k}(\lambda/\mu)$ is a skew Whittaker function defined below in \eqref{eq:defskewWhittaker}. In the equation above, $\delta_0$ denotes a Dirac mass at zero, so that thoughout the paper, for both models, we always have $\lambda_2^0=0$.
Then, similarly as for two-layer Schur processes,  we have again an equality in law of the form  \eqref{eq:twodescriptions} where now, $\mathbf L_1, \mathbf L_2$ are distributed as in Definition \ref{def:stationaryloggamma}  (this is proved in \cite[Prop. 3.14]{barraquand2023stationary}).

\subsection{Markovian description of stationary measures}
The stationary measure for the open KPZ equation can  be described in terms of the increments of a Markov process with a specific initial condition \cite{bryc2021markov}. The transition probabilities of the Markov process involved is somewhat complicated, but this point of view is quite useful when studying various type of limits (large scale limit, limit from the interval to $\mathbb R_+$, limit from discrete processes to continuous ones, etc, see for instance \cite{bryc2021markov2, bryc2024limit, bryc2024limits, bryc2023pitman}). 

In this section, we explain how the probability measures  from Section \ref{sec:two-layer-intro} can be alternatively viewed as Markov processes. Our argument is inspired by earlier results: on the one hand  the results of \cite{bryc2021markov} expressing the open KPZ stationary measure in terms of a Doob transform of a Brownian motion killed at a certain rate, and on the other hand the results  of \cite{das2024convergence} which relate certain two-layer Gibbs measures (very similar with our two-layer Whittaker process) and Doob-transformed processes. We will show that the technical tools developed in this article allow to provide fully explicit expressions for the transition probabilities of the Doob-transformed Markov processes involved. 

These explicit expressions can in turn be studied as $N$ goes to infinity, or in other asymptotic regimes. This may be used to show that as $N$ goes to infinity, the stationary measures of last passage percolation or the log-gamma polymer in a strip converge to stationary measures of half-space last passage percolation or log-gamma polymer. We will not prove this convergence in general but only sketch  how this can be done for some range of parameters. We believe that this question should be addressed in a more general setting, i.e. for the full range of boundary parameters, and not only for stationary measures on a horizontal path (see the more general Definitions \ref{def:generaltwolayerSchur} and \ref{def:generaltwolayerWhittaker} below). 
For the open KPZ equation, the corresponding limit, i.e. the convergence of stationary measures on $[0,L]$ to stationary measures on $\R_+$,  is equivalent to results of Hariya-Yor \cite{hariya2004limiting}, as shown in \cite{barraquand2021steady}.

\subsubsection{Two-layer Schur process: Markovian description} We provide more details about the results of this section in Section \ref{sec:MarkovianSchur}, dealing with more general two-layer Schur processes.

The marginal distribution of $\la^{0}$ under the probability measure \eqref{eq:introtwolyaerSchur} can be written as  $\mathsf p_0(\la^0)$ where 
\begin{equation*}
	\mathsf p_0(\la) = c_1^{\la_1-\la_2} h_{0,N}(\la_1-\la_2)\mathds{1}_{\la_2=0},
\end{equation*}
and  the functions $h_{x,N}:\Z_{\geq 0}\to \mathbb R_+$ are defined by 
\begin{equation}
	h_{x,N}(\ell)  = \frac{q_{N-x}(\ell)}{\sum_{\ell\geq 0} c_1^\ell q_N(\ell)}
	\label{eq:defhintro}
\end{equation} 
with  
\begin{equation*}
	q_M(\ell) = (1-a)^{2M} \sum_{\la\in \Sign_2} s_{\la/(\ell,0)}\left((a)^M\right) c_2^{\la_1-\la_2}, 
\end{equation*}
where $(a)^M=(a, \dots, a)\in [0,1)^M$. 
We note that the denominator $\sum_{\ell\geq 0} c_1^\ell q_N(\ell)$ in \eqref{eq:defhintro} is equal to the partition function $\mathcal Z_{\rm Geo}^{a,c_1,c_2}(N)$ in \eqref{eq:introtwolyaerSchur}, but we chose to write it in this way so that it becomes clear that $\mathsf p_0$ defines a probability distribution. 

For integers $0=x_0<x_1<\dots <x_k$, the function $h_{x,N}$ is defined precisely so that the marginal distribution of $\la^{x_0}, \la^{x_1}, \dots, \la^{x_k}$ is given by 
\begin{equation*}
	\mathsf p_0(\la^{x_0})\prod_{i=1}^k \mathsf p_{x_{i-1}, x_{i}}(\la^{x_{i-1}}, \la^{x_{i}}),
	\label{eq:p0}
\end{equation*}
where 
\begin{equation*}
	\mathsf p_{x,y}(\la, \mu) = (1-a)^{2(y-x)} s_{\mu/\la}\left((a)^{y-x}\right)\frac{h_{y,N}(\mu_1-\mu_2)}{h_{x,N}(\la_1-\la_2)}. 
\end{equation*}

Thus, the process $(\la^i)_{0\leq i\leq N}$ defined by \eqref{eq:introtwolyaerSchur} can also be described as a time-inhomogeneous Markov process with initial distribution $\mathsf p_0$ and transition kernel $	\mathsf p_{x,y}(\lambda, \mu)$ (from time $x$ to time $y$). Furthermore, all quantities appearing in the definition of the kernel admit explicit integral formulas.  Skew Schur functions can be written as a contour integral by Lemma \ref{lem:skewSchur}. The normalization constant $\mathcal Z_{\rm Geo}^{a,c_1,c_2}(N)$ is computed in \eqref{eq:noramlizationLPPintro}. Finally, the function $q_{M}$ admits the following expression (proved as Proposition \ref{prop:formulaqgeneral} in Section \ref{sec:MarkovianSchur} below).

 For $c_1, c_2<1$, and any $\ell\in \Z_{\geq 0}$, we have
	\begin{equation}
		q_M(\ell) = \frac{1}{2} \oint \frac{dz}{2\I\pi} \left(z^{\ell+1} - \frac{1}{z^{\ell+1}} \right)\left(\frac{1}{z^2} - 1 \right) \frac{1}{1-c_2z}\frac{1}{1-c_2/z}\left(\frac{1-a}{1-az}\frac{1-a}{1-a/z} \right)^M,
		\label{prop:expressionQ}
	\end{equation}
where the contour is a circle of radius $1$ around the origin. 
\begin{remark} 
When $N\to\infty$ one can show   (see Proposition \ref{prop:limitq}) that for any $x$ and $c_2<1$, 
\begin{equation}
	h_{x,N}(\ell)\xrightarrow[N\to\infty]{}h(\ell) := (\ell+1)(1-c_1)^2.
	\label{eq:defhlimitintro}
\end{equation}
In particular the limit does not depend on $x$.  Thus, for $c_1,c_2<1$, the large $N$ limit of the Markov process with transition probabilities $\mathsf p_{x,y}(\la, \mu)$ becomes a very simple Doob transformed process, which can be interpreted as a conditioned random walk studied in  \cite{o2003conditioned}. As one should  expect, it  matches with  the stationary model for half-space geometric LPP denoted $G_{r,s}^{\rm stat}$   in  \cite[Sec. 3]{barraquand2023stationaryloggamma}, with $r=c_1, s=1$ (see more details in Section \ref{sec:MarkovianSchur}).  
\end{remark}
\begin{remark}
	The limit \eqref{eq:defhlimitintro} has another application. It allows to define a consistent  family of probability measures 
	$$ \mathbb P_n\left( \lambda^0,\dots,\lambda^n\right) = c_1^{\lambda_1^0-\lambda_2^0}\mathds{1}_{\lambda_2^0=0}  \times \left( \prod_{i=1}^n (1-a)^2s_{\lambda^i/\lambda^{i-1} }(a)\right) \times h(\lambda_1^n-\lambda_2^n) .$$
	They are consistent in the sense that for all $k\geq 0$, the marginal distribution of $\lambda^0,\dots,\lambda^k$ in $\mathbb P_n$ is the same for all $n\geq k$. We believe that these measures are the appropriate object to adapt the method of \cite{barraquand2023stationary} to half-space models. In general, one should consider a family of possible functions $h$, parametrized by $c_2$ with  $$h(\ell)\propto  s_{(\ell,0)}(c_2,1/c_2) = \frac{c_2^{\ell+1}-c_2^{-\ell-1}}{c_2-1/c_2}.$$
\end{remark}

\subsubsection{Two-layer Whittaker process: Markovian description} The Whittaker case is similar. We provide more details in Section \ref{sec:MarkovWhittaker}, dealing with more general two-layer Whittaker processes. 

Using the branching rule for Whittaker functions, the marginal distribution of $\la^{0}$ under the probability measure \eqref{eq:introtwolyaerSchur} has density $\mathsf P_0(\la^0)$ where 
\begin{equation*}
	\mathsf P_0(\la) = e^{-u(\la_1-\la_2)} H_{0,N}(\la_1-\la_2)\delta_0(\la_2),
\end{equation*}
and  the functions $H_{x,N}:\R\to \mathbb R_+$ are defined by 
\begin{equation*}
	H_{x,N}(\ell)  = \frac{Q_{N-x}(\ell)}{\sum_{\ell\geq 0} e^{-u\ell} Q_N(\ell)}
	\label{eq:defHintro}
\end{equation*} 
with  
\begin{equation*}
	Q_M(\ell) =\frac{1}{\Gamma(\alpha)^{2M}} \int_{\R^2}d\lambda  \Psi_{(\alpha)^M} (\la/(\ell,0)) e^{-v(\la_1-\la_2)}. 
\end{equation*} 
For integers $0=x_0<x_1<\dots <x_k$, the marginal distribution of $\la^{x_0}, \la^{x_1}, \dots, \la^{x_k}$ has density 
\begin{equation*}
	\mathsf P_0(\la^{x_0})\prod_{i=1}^k \mathsf P_{x_{i-1}, x_{i}}(\la^{x_{i-1}}, \la^{x_{i}}),
\end{equation*}
where 
\begin{equation*}
	\mathsf P_{x,y}(\la, \mu) = \frac{1}{\Gamma(\alpha)^{2(y-x)}}\Psi_{(a)^{y-x}}(\mu/\la) \frac{H_{y,N}(\mu_1-\mu_2)}{H_{x,N}(\la_1-\la_2)}. 
\end{equation*}

Thus, the process $(\la^i)_{0\leq i\leq N}$ defined by \eqref{eq:introtwolayerWhittaker} is a Markov process with initial distribution $\mathsf P_0$ and transition kernel $\mathsf P_{x,y}$. Again, all quantities appearing in the definition of the kernel admit explicit integral formulas.  Skew Whittaker functions can be written as contour integrals  by Proposition \ref{lem:integralWhittaker}. The normalization constant $\mathcal Z_{\rm LG}^{\alpha, u,v}(N)=\sum_{\ell\geq 0} e^{-u\ell} Q_N(\ell)$ is computed in \eqref{eq:normalizationloggammaintro}. Finally, the function $Q_{M}$ admits the following expression (see Proposition \ref{prop:expressionQWhittaker} below). For $u,v>0$, and any $\ell\in \R$,  we have
	\begin{equation*}
		Q_M(\ell) =\int_{\I\R} \frac{dz}{2\I\pi} \Psi_z(\ell,0)  \frac{\Gamma(v\pm z)}{2\Gamma(\pm 2 z)}\left(\frac{\Gamma(\alpha\pm z)}{\Gamma(\alpha)^2}\right)^M.
		\label{prop:expressionQWhittaker}
	\end{equation*} 

\begin{remark} 
	When $N\to\infty$ an asymptotic analysis of the above integral (Proposition \ref{prop:limitQ}) shows that for any $x$ and $v>0$, 
	\begin{equation*}
		H_{x,N}(\ell)\xrightarrow[N\to\infty]{}H(\ell) :=  \frac{2K_0(2e^{-\ell/2}) }{\Gamma(u)^2},
	\end{equation*}
where $K_0$ is a Bessel K function. 
Again, the limit does not depend on $x$, so that  for $u,v>0$, the large $N$ limit of the Markov process with transition probabilities $\mathsf P_{x,y}(\la, \mu)$ becomes a very simple Doob transformed process, with transition probability (from time $x$ to $x+1$)
\begin{equation}
\overline{	\mathsf P}(\lambda, \mu) = \frac{H(\mu_1-\mu_2)}{H(\la_1-\la_2)} \frac{1}{\Gamma(\alpha)^2}\Psi_{\alpha}(\mu/\la).
	\label{eq:Doobtransformedprocessloggamma}
\end{equation}
 One remarks that the function $H$ is the same as in the Markovian description of limits of the open KPZ equation stationary measure \cite[Theorem 2.3]{bryc2021markov2} (up to some factors of 2 which result from different conventions used). The same $H$-transform appeared  earlier in several works, for instance in the Matsumoto-Yor theorem \cite{matsumoto2000analogue}. 

The Markov process with initial distribution $\mathsf P_0(\lambda) = e^{-u(\la_1-\la_2)}H(\la_1-\la_2)\delta_0(\la_2)$ and kernel \eqref{eq:Doobtransformedprocessloggamma} should match with  the half-space log-gamma stationary measure denoted $h_{u,0}$ in \cite[Definition 1.6]{barraquand2023stationaryloggamma}. It is outside the scope of the present paper to fully prove this matching, but we expect that it could be checked using the fact that  $h_{u,0}$ arises as the free energy of a log-gamma polymer on two rows (see \cite{barraquand2023stationaryloggamma}), which can be studied through half-space Whittaker processes \cite{barraquand2018half}. In particular,  results of  \cite{corwin2014tropical}  relate certain marginals of Whittaker processes to explicit Doob transformed processes (see \cite[Eq. (3.11)]{corwin2014tropical}). 
\label{rem:Doobloggamma}
\end{remark}

\subsection{Outline of the rest of the paper} Although  results are very  similar in the Schur and Whittaker case, arguments are often different, and  we will need to establish more preliminary results about Whittaker functions. Thus, we chose to present all results about two-layer Schur processes in Section \ref{sec:Schur}, and all results about two-layer Whittaker processes in Section \ref{sec:Whittaker}. The two sections are independent, although some arguments common to both sections are presented with slightly more details in Section \ref{sec:Schur}. Finally, in Section \ref{sec:KPZ}, we discuss the application to the open KPZ equation. 

\subsection*{Acknowledgments} This work was partially supported by Agence Nationale de la Recherche through grants ANR-21-CE40-0019 and ANR-23-ERCB-0007-01, as well as by the Swedish Research Council under grant no. 2016-06596, while the author was in residence at Institut Mittag-Leffler in Djursholm, Sweden, during the Fall 2024. The author thanks Ivan Corwin and Pierre Le Doussal for useful discussions about some of the results, and Wlodek Bryc for valuable comments on an early draft of this article.

\section{Two-layer Schur process}
\label{sec:Schur}

\subsection{Background on Schur functions}
\label{sec:Schurbackground}
We recall the main properties of Schur functions. We refer to \cite[Chap. 1]{macdonald1995symmetric} for a more comprehensive presentation. 
\subsubsection{Definitions}
Let us denote by $\Sign_n=\lbrace (\lambda_1,\dots,\lambda_n)\in \mathbb Z^n; \lambda_1\geq \dots \geq \lambda_n\rbrace$ the set of integer signatures of length $n$. 
 Let  $\Signplus_n=\lbrace (\lambda_1,\dots,\lambda_n)\in \mathbb Z^n; \lambda_1\geq \dots \geq \lambda_n\geqslant 0\rbrace$. We also denote by $\mathbb Y$ the set of all integer partitions, i.e. the set of non-increasing sequences of non-negative integers of the form $(\lambda_1, \lambda_2, \dots, )$ with finitely many non-zero coordinates. We will denote by $\ell(\lambda)$ the number of non-zero elements in a partition or a signature $\lambda$ and we will sometimes write a partition $\lambda$ as $\lambda=(\lambda_1, \dots, \lambda_{\ell(\lambda)}, 0, \dots, 0)$, thus viewing $\lambda$ as an element of $\Signplus_n$ for any  $n\geq \ell(\la)$. 

For $\lambda\in \Signplus_n$ and a set of variables $x=\lbrace x_1, \dots, x_n\rbrace$, Schur functions $s_{\lambda}(x)$ are symmetric polynomials 
\begin{equation} s_{\la}(x) := \frac{\det\left(x_i^{\la_j+n-j} \right)_{i,j=1}^n}{\det\left(x_i^{n-j} \right)_{i,j=1}^n}.
	\label{eq:defSchur}
\end{equation}
One extends this definition when $x=\lbrace x_1, \dots, x_m\rbrace$ with $m> n$ by letting $s_{\lambda}(x):=s_{\tilde\lambda}(x)$ where $\tilde\lambda=(\lambda_1, \dots, \lambda_n, 0,\dots, 0)\in \Signplus_m$. It is well-known that we have a stability property, i.e. for any $m$, $s_{\lambda}(x_1, \dots, x_m,0)=  s_{\lambda}(x_1, \dots, x_m)$. These considerations  imply  that 
Schur functions are more naturally indexed by partitions rather than integer signatures.\footnote{Moreover, Schur polynomials can be further extended to elements of the ring of symmetric functions, hence the name Schur functions. We refer to \cite{macdonald1995symmetric} for details.  }
However, in the present work,  it will be more convenient for us to  rather view Schur functions as polynomials in finitely many variables,  indexed by signatures. 

Two other properties are obvious from \eqref{eq:defSchur}: the scaling property, that is for $t\in \mathbb C$,  
\begin{equation}
s_{\la}(tx)=t^{\vert \lambda\vert} s_{\lambda}(x) \text{ with }\vert \la\vert =\sum_{i} \la_i, 
\label{eq:scalingproperty}
\end{equation}
and the shift property, that is for $x=\lbrace x_1, \dots, x_m\rbrace$,  $\lambda\in \Signplus_n$ and $c\in \mathbb Z_{\geq 0}$, 
\begin{equation}
	s_{\lambda+c^n}(x) = (x_1, \dots, x_m)^c s_{\la}(x). 
	\label{eq:shiftproperty}
\end{equation}
The property \eqref{eq:shiftproperty} can finally be imposed to hold for any $c\in \mathbb Z$ so as to extend the definition of Schur functions $s_{\la}$ to any $\la\in\Sign_n$ (for any $n$).

We could have defined Schur functions in an alternative way. For $\lambda, \mu\in \Sign_n$ and a single variable $x$, define skew Schur functions
\begin{equation}
	s_{\mu/\lambda}(x) = \mathds{1}_{\lambda\prec\mu} x^{\vert \mu \vert-\vert\lambda \vert}, 
	\label{eq:defskewSchur}
\end{equation}
where $\lambda\prec \mu$ means that $\lambda$ interlaces with $\mu$, that is $\mu_1\geq \la_1\geq \mu_2\geq\la_2\geq \dots\geq\mu_n\geq \la_n$.  
We then extend the definition to any tuple of variables $x=\lbrace x_1, \dots, x_m \rbrace$ by imposing the branching rule 
\begin{equation}
	s_{\mu/\lambda}(x_1, \dots, x_m) = \sum_{\nu\in \Sign_n } s_{\mu/\nu}(x_1, \dots, x_{m-1})s_{\nu/\lambda}(x_m),
	\label{eq:branchingrule}
\end{equation}
and it turns out that $s_{\mu/\lambda}(x_1, \dots, x_m)$ is a symmetric polynomial. 
Skew Schur functions also satisfy a scaling property, for $t\in \mathbb C$, 
\begin{equation}
	s_{\mu/\la}(tx)=t^{\vert \mu\vert-\vert \la\vert} s_{\mu/\la}(x).
	\label{eq:scalingpropertyskew}
\end{equation}
It also turns out that for any $\lambda\in \Signplus_n$ and $x=\lbrace x_1, \dots, x_m\rbrace$,   $s_{\lambda/0^n}(x)=s_{\lambda}(x)$ as defined in \eqref{eq:defSchur}.

\subsubsection{Summation formulas and orthogonality}

Schur functions satisfy the so-called Cauchy identity: for sets of variables $x=\lbrace x_1, \dots, x_n\rbrace$ and $y=\lbrace y_1,\dots, y_m\rbrace$, 
\begin{equation}
	\sum_{\la\in \Signplus_n} s_{\la}(x)s_{\la}(y) = \Pi(x;y) 
	\label{eq:Cauchy}
\end{equation}
where
\begin{equation}
	\Pi(x;y) = \prod_{i=1}^n\prod_{j=1}^m \frac{1}{1-x_i y_j}.
	\label{eq:defPi}
\end{equation}
For more general specializations of the Schur functions, the summation in \eqref{eq:Cauchy}  should be taken  over all integer partitions $\la\in\mathbb Y$. In our case, $s_{\la}(x_1, \dots, x_n)$ being zero for partitions $\la$ with length higher than $n$, we can restrict the summation over $\la \in\Signplus_n$.

Next, we will use the fact \cite[Ch. VI, Section 9]{macdonald1995symmetric} that $n$-variable Schur polynomials are orthonormal with respect to the scalar product defined as 
\begin{equation}
	\langle f,g\rangle_n = \oint\frac{dz_1}{2\I\pi z_1}\dots \oint\frac{dz_n}{2\I\pi z_n} f(z)g(z^{-1})\Delta(z)
	\label{eq:defscalarproductSchur}	
\end{equation} 
for polynomials $f,g$ of $n$ variables, where 
$$ \Delta(z)= \frac{1}{n!}\prod_{i\neq j}(1-z_i/z_j).$$
 
Using the Cauchy identity \eqref{eq:Cauchy} and the orthogonality of Schur polynomials, one can deduce the following skew Cauchy identity \cite[Ch I, Section 5 Ex. 26]{macdonald1995symmetric}: 
For any $\lambda\in \Signplus_n$, 
\begin{equation}
	\sum_{\mu\in\Signplus_n} s_{\mu/\lambda}(a_1, \dots, a_k)s_{\mu}(z_1, \dots, z_n) = s_{\lambda}(z_1,\dots, z_n)\Pi(a;z), 
	\label{eq:specialCauchy}
\end{equation}

Finally, we will use another summation formula called  Littlewood identity \cite[Ch I, Section 5 Ex. 7]{macdonald1995symmetric}: 
\begin{equation}
	\sum_{\la\in \Signplus_n} s_{\lambda}(a_1, \dots, a_n)c^{\sum_{i=1}^{n}(-1)^{i-1}\lambda_i} = \prod_{i=1}^n \frac{1}{1-ca_i} \prod_{1\leq i<j\leq n}\frac{1}{1-a_ia_j}.
	\label{eq:Littlewood}
\end{equation}

\subsection{General definition of the two-layer open Schur process}

We start with a more general definition of two-layer open Schur processes from \cite{barraquand2023stationary}. This is a probability measure on a sequence of signatures, depending on some down-right path $\mathcal P$ as depicted in Figure \ref{fig:downrightpaths}. It also depends on a sequence $a_1, \dots, a_N$ of inhomogeneity parameters. This probability measure describes the stationary measure of Last Passage Percolation on a strip along any down-right path, and  we refer to \cite{barraquand2023stationary} for more precise statements.

\begin{figure}[h]
	\begin{center}
		\begin{tikzpicture}[scale=0.7]
			\begin{scope} 
				\clip (0.5,0.5) -- (6.5,0.5) -- (11.5,5.5) -- (5.5,5.5) -- cycle;
				\draw[gray] (0,0) grid (12,12);
			\end{scope}
			\draw[ultra thick] (4,4) --(6,4) -- (6,2) -- (7,2) -- (7,1);
					\draw[] (4,4) node[below] {$\mathbf p_0$};
					\draw[] (5,4) node[below]{$\mathbf p_1$};
			\draw (6,2) node[below right] {$\mathcal P$}; 
		\end{tikzpicture}
	\end{center}
	\caption{A down-right path $\mathcal P$ joining left and right boundaries of the strip.}
	\label{fig:downrightpaths}
\end{figure}
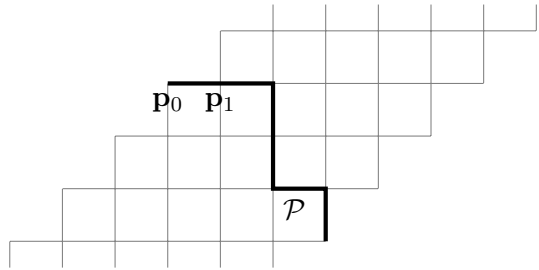
We will consider paths $\mathcal P$ joining points $\mathbf p_0, \mathbf p_1, \dots, \mathbf p_N$ where $\mathbf p_0=(x,x)$ for some $x\in \Z$, and $\mathbf p_N=(y+N,y)$ for some $y\in \Z$, such as $\mathbf p_{i+1}-\mathbf p_i$ can be either $(1,0) = ``\rightarrow"$ or $(0,-1)= ``\downarrow"$.  Up to translations in the diagonal direction, such paths $\mathcal P$ are determined by the sequence of increments $\mathbf p_{i+1}-\mathbf p_i$, which can be seen as a word $w\in \lbrace \rightarrow, \downarrow \rbrace^N$. 
 
\begin{definition}
Consider an integer $N\geq 1$, parameters $a_1, \dots, a_N\in (0,1)$ and boundary parameters $c_1,c_2$ such that $c_1c_2<1$ and $c_ia_j<1$ for all $i,j$. Let $w\in \lbrace \rightarrow, \downarrow \rbrace^N$. The two-layer open Schur process is a probability measure on sequences  $\boldsymbol\lambda=(\lambda^{0}, \dots, \lambda^{N}) \in (\Sign_2)^{N+1}$ such that
\begin{equation}
	\mathbb P(\boldsymbol{\lambda}) = \frac{1}{\mathcal Z^{\vec a, c_1,c_2}_{\rm Geo}(N)} c_1^{\lambda_1^0-\lambda_2^0}c_2^{\lambda_1^N-\lambda_2^N} \prod_{\footnotesize \begin{matrix}
			1\leq i\leq N \\ w_i=\rightarrow
	\end{matrix}} s_{\lambda^i/\lambda^{i-1} }(a_i)\prod_{\footnotesize \begin{matrix}
	1\leq i\leq N \\ w_i=\downarrow
\end{matrix}} s_{\lambda^{i-1}/\lambda^{i} }(a_i)\;\; \mathds{1}_{\lambda_2^0=0},
\end{equation} 
where the normalization constant $\mathcal Z^{\vec a, c_1,c_2}_{\rm Geo}$ does not depend on the shape of the path $w$ (\cite[Proposition 2.15]{barraquand2023stationary}). 
\label{def:generaltwolayerSchur}
\end{definition} 
 
In particular, when $w=(\rightarrow, \dots, \rightarrow)$, that is when $\mathcal P$ is a horizontal path, so that 
\begin{equation}
	\mathbb P(\boldsymbol{\lambda}) = \frac{1}{\mathcal Z^{\vec a, c_1,c_2}_{\rm Geo}(N)} c_1^{\lambda_1^0-\lambda_2^0}c_2^{\lambda_1^N-\lambda_2^N} \prod_{i=1}^N s_{\lambda^i/\lambda^{i-1} }(a_i)\;\; \mathds{1}_{\lambda_2^0=0},
	\label{eq:deftwolayerSchur}
\end{equation} 
and we recover the probability measure defined in Section \ref{sec:two-layer-intro}. 
Given some integers $0= x_0< x_1< \dots< x_k= N$, one can apply the branching rule \eqref{eq:branchingrule} to check that the marginal distribution of $\lambda^{x_0}, \dots, \lambda^{x_k}$ is 
\begin{equation}
	\mathbb P\left( \la^{x_0}=\mu^0, \dots, \la^{x_k}=\mu^k\right) = 
	\frac{1}{\mathcal Z^{\vec a, c_1,c_2}_{\rm Geo}(N)} c_1^{\mu_1^{0}-\mu_2^{0}}c_2^{\mu_1^{k}-\mu_2^{k}} \prod_{i=1}^k s_{\mu^{i}/\mu^{i-1} }(a_{x_{i-1}+1},\dots,a_{x_i})\;\; \mathds{1}_{\mu_2^{0}=0}.
	\label{eq:marginaldistributionSchur}
\end{equation}

\subsection{Exact integral formulas}
One of the main results of this article  is Theorem \ref{theo:formulaSchur} below. We explain in Section \ref{sec:proofLPP} how it directly implies Theorem \ref{theo:formulaLPPintro} stated in the introduction. For simplicity of the formulas and notations, we restrict in this Section to the case where $w=(\rightarrow)^N$, i.e. the path $\mathcal P$ is horizontal. We believe that  our method could be adapted for arbitrary $w$.

\begin{theorem}Assume that $c_1, c_2<1$. Let $t_1,\dots, t_k$ be real numbers such that  $ c_1^2<t_1<\dots <t_k< \left(\frac{1}{c_2}\right)^2$ and for all $1\leq i\leq k$ and $x_{i-1}<r\leq x_i$, we impose $\sqrt{t_i}<1/a_r$. Then, 
	
	\begin{multline}  	\mathbb E\left[\prod_{i=1}^k t_i^{\lambda_1^{x_i}-\lambda_1^{x_{i-1}}} \right] =    \frac{1}{\mathcal Z^{\vec a, c_1,c_2}_{\rm Geo}(N)}   \oint\frac{\mathfrak m_{\rm Geo}(dz_1)}{2\I\pi}\dots  \oint\frac{\mathfrak m_{\rm Geo}(dz_k)}{2\I\pi}
		\prod_{i=1}^k  \prod_{r=x_{i-1}+1}^{x_i} \phi(z_i^{\pm 1}, a_r \sqrt{t_i})\\
		\times \phi(z_1^{\pm 1}, c_1/\sqrt{t_1})  \phi(z_k^{\pm 1}, c_2\sqrt{t_k})  \prod_{i=1}^{k-1} \left(1-\frac{t_i  }{t_{i+1}} \right)
			\phi\left(z_i^{\pm 1}z_{i+1}^{\pm 1}, \sqrt{\frac{t_i}{t_{i+1}}}\right), 
		\label{eq:Laplacetransform2}
	\end{multline}
	where the integration contours  are circles around $0$ with radius $1$,  the superscript $\pm 1$ means that we take the product of the function $\phi$ with the choice  $+1$ and the function $\phi$ with the choice $-1$, and the notations $\mathfrak m_{\rm geo}(dz)$ and $\phi(z, a)$ are defined in \eqref{eq:defmphigeo} (the formula \eqref{eq:Laplacetransform2} is rewritten below in \eqref{eq:formulaschurdetails} without using any shorthand notations). 
\label{theo:formulaSchur}
\end{theorem} 

 As a corollary, the normalization constant $\mathcal Z^{\vec a, c_1,c_2}_{\rm Geo}$ can also be expressed as an integral.
\begin{corollary}
Assume that $c_1,c_2<1$. Then, 
 \begin{equation} \mathcal Z^{\vec a, c_1,c_2}_{\rm Geo}(N) =  \frac{1}{2}  \oint \frac{dz}{ 2\I\pi z} \frac{(1-z^2)(1-1/z^2)}{(1-zc_1)(1-\frac{c_1}{z})(1-\frac{c_2}{z})(1-c_2 z)} \prod_{i=1}^N\left(\frac{1}{(1- a_i z)(1-a_i/ z)} \right), 
	 	\label{eq:normalisation integral}
\end{equation}
where the contour is a circle around $0$ with radius $1$. 
Evaluating by residues for small values of $N$, we find, 
\begin{equation}
	\mathcal Z^{\vec a, c_1,c_2}_{\rm Geo}(N=0) = \frac{1}{1-c_1c_2}, \quad \quad \mathcal Z^{\vec a, c_1,c_2}_{\rm Geo}(N=1) = \frac{1}{(1-c_1c_2)(1-a_1c_1)(1-a_1c_2)}, 
	\label{eq:normalisationexplicit1}
\end{equation}
and 
\begin{equation}
\mathcal Z^{\vec a, c_1,c_2}_{\rm Geo}(N=2) = \frac{1-a_1a_2c_1c_2}{(1-c_1c_2)(1-a_1c_1)(1-a_1c_2)(1-a_2c_1)(1-a_2c_2)(1-a_1a_2)}.
	\label{eq:normalisationexplicit2}
\end{equation}
For larger values of $N$, the partition function does not seem to factorize anymore. 
\label{cor:normalization}
\end{corollary}
\begin{proof}
	The normalization constant does not depend on $k$, and is necessarily equal to the same integral as in \eqref{eq:Laplacetransform2} where $k=1$ and  $t_1=1$, which immediately yields the formula \eqref{eq:normalisation integral}. 
	To evaluate this integral by residues, we compute residues at the poles for $z=c_1$, $z=c_2$ and $z=a_j$,  which evaluate to \eqref{eq:normalisationexplicit1} and \eqref{eq:normalisationexplicit2}.
\end{proof}
\begin{remark}
When $c_1<1$ but $c_2>1$, or vice versa, the value of the normalization constant can be computed from \eqref{eq:normalisation integral} by analytically continuing the integral. In other words, \eqref{eq:normalisation integral} is true when $c_1c_2<1$. Otherwise,  we should  define the integral as being the sum of residues of the integrand  at $z=c_1$, $z=c_2$ and $z =a_j$ for all $1\leq j\leq N$. 
\end{remark}

\subsection{Proof of Theorem \ref{theo:formulaSchur}}
\label{sec:proofLaplacetransform}
In order to compute the multidimensional Laplace transform appearing in \eqref{eq:Laplacetransform1}, we use an integral formula for skew Schur functions arising in the marginal distribution  \eqref{eq:marginaldistributionSchur}.  The following well-known result  holds for signatures of any length $n$, though we will eventually use it only for $n=2$. 
\begin{lemma}
For $\lambda, \mu\in \Sign_n$, 
\begin{equation}
	s_{\mu/\lambda}(a_1, \dots, a_k) = \oint\frac{dz_1}{2\I\pi z_1}\dots \oint\frac{dz_n}{2\I\pi z_n} \Delta(z)  s_{\mu}(z^{-1})s_{\lambda}(z) \prod_{i=1}^n\prod_{j=1}^k \frac{1}{1-z_i a_j} 
	\label{eq:formulaskewSchur}
\end{equation}	
where $z=(z_1,\dots, z_n)$, $z^{-1}=(z_1^{-1}, \dots, z_n^{-1})$, 
and the variables are integrated along a circle around $0$ with radius smaller than all the $a_j^{-1}$. 
\label{lem:skewSchur}
\end{lemma}
\begin{proof}
	Let us start from the following special case of the skew Cauchy identity \eqref{eq:specialCauchy}
\begin{equation}
	 \sum_{\mu\in\Signplus_n} s_{\mu/\lambda}(a_1, \dots, a_k)s_{\mu}(z_1, \dots, z_n) = s_{\lambda}(z_1,\dots, z_n)\Pi(a;z).
	 \label{eq:specialCauchy2}
\end{equation}
Taking the scalar product \eqref{eq:defscalarproductSchur}, in both sides of \eqref{eq:specialCauchy}, against $s_{\mu}$, we obtain
$$s_{\mu/\lambda}(a_1, \dots, a_k) = \left\langle s_{\lambda}(\cdot)\Pi(a;\cdot ) \; , \; s_{\mu}(\cdot) \right\rangle_n$$
which establishes the formula \eqref{eq:formulaskewSchur} when $\lambda, \mu\in \Signplus_n$. The formula extends to any $\lambda, \mu\in \Sign_n$ thanks to the shift property of Schur functions \eqref{eq:shiftproperty}. 
\end{proof}
Before we prove Theorem \ref{theo:formulaSchur}, let us state another useful identity. This one holds for signatures of length  $n=2$ and does not seem to admit natural generalizations. 
\begin{lemma}
	For $x=(x_1,x_2)\in \mathbb C^2$ and $y=(y_1,y_2)\in \mathbb C^2$ with $\vert x_iy_j\vert <1$, and $c\in \mathbb C$ such that $\vert c x_iy_j\vert <1$,  we have 
	\begin{equation}
		\sum_{\lambda\in \Signplus_2} s_{\lambda}(x)s_{\lambda}(y)c^{\lambda_1-\lambda_2} = \Pi(x,cy)\frac{1-c^2 x_1x_2y_1y_2}{1- x_1x_2y_1y_2}.
	\end{equation}
\label{lem:specialcomputation}
\end{lemma}
\begin{proof}
The formula can be proved by a direct computation, writing Schur functions as alternant determinant ratios, expanding the determinants, and computing all sums explicitly. We prefer to give another proof, which is slightly less elementary, but does not involve any tedious computation.

The Schur measure on integer partitions
\begin{equation}
	\mathbb P(\lambda) = \frac{1}{\Pi(x,cy)}s_{\lambda}(x)s_{\lambda}(cy),
	\label{eq:Schurmeasure}
\end{equation} 
where $x=(x_1,x_2)$ and $y=(y_1,y_2)$ and $c>0$, is supported on signatures $\lambda\in \Sign_2$ and it is well-known that  under this probability measure, 
the coordinates $\lambda_1, \lambda_2$ are distributed as  
\begin{align}
	\lambda_1 &= w_{11}+ \max\lbrace w_{12},w_{21}\rbrace+ w_{22}\label{eq:RSK1}\\ 
	\lambda_1+\lambda_2 &= w_{11}+  w_{12}+w_{21}+ w_{22},\label{eq:RSK2}
\end{align}
where $w_{ij}$ are independent random variables  such that $w_{ij} \sim\mathrm{Geom}(cx_iy_j)$. This follows from the Robinson-Schensted-Knuth correspondence, see e.g.  \cite{johansson2000shape}. 
Ley us now write 
\begin{equation}
	\frac{1}{\Pi(x, cy)}\sum_{\lambda\in \Signplus_2} s_{\lambda}(x)s_{\lambda}(y)c^{\lambda_1-\lambda_2}=\frac{1}{\Pi(x, cy)} \sum_{\lambda\in \Signplus_2} s_{\lambda}(x)s_{\lambda}(cy)c^{-2\lambda_2}= \mathbb E\left[c^{-2 \min\lbrace w_{12}, w_{21}  \rbrace} \right],
	\label{eq:sumtoexpectation}
\end{equation}
where in the first equality we have used the scaling property \eqref{eq:scalingproperty}, and in the second equality, we recognize an expectation value with respect to the measure \eqref{eq:Schurmeasure}. 
Under this measure, \eqref{eq:RSK1} and \eqref{eq:RSK2} imply that  $\min\lbrace w_{12}, w_{21} \rbrace \sim \mathrm{Geom}(c^2 x_1x_2y_1y_2)$, so that 
$$\mathbb E\left[c^{-2 \min\lbrace w_{12}, w_{21}  \rbrace} \right] = \frac{1-c^2 x_1x_2y_1y_2}{1- x_1x_2y_1y_2}.$$ 
Hence, multiplying both sides of \eqref{eq:sumtoexpectation} by $\Pi(x,cy)$ we obtain the desired result. 
\end{proof}

Now we are ready to prove Laplace transform formulas. We first establish the following (which requires  $c_1c_2<1$  but not necessarily $c_1<1$ or $c_2<1$).
\begin{proposition}
	Let $t_1,\dots, t_k$ be real numbers such that  $c_1^2<t_1<\dots <t_k<\left( \frac{1}{c_2}\right)^2$. We have 
	\begin{multline}  	\mathbb E\left[\prod_{i=1}^k t_i^{\lambda_1^{x_i}-\lambda_1^{x_{i-1}}} \right] =   \frac{1}{\mathcal Z^{\vec a, c_1,c_2}_{\rm Geo}(N)}   \oint\!\!\!\oint \nu(dw^1) \dots \oint\!\!\!\oint \nu(dw^k) 
		\prod_{i=1}^k   \prod_{j=1}^2 \prod_{r=x_{i-1}+1}^{x_i} \frac{1}{1-a_r \sqrt{t_i}w_j^i}
		\\
		\times \prod_{i=1}^{k-1} \left(\frac{1-\frac{t_i w_1^{i+1}w_2^{i+1} }{t_{i+1} w_1^{i}w_2^{i} }   }{1- \frac{ w_1^{i+1}w_2^{i+1} }{w_1^{i}w_2^{i} } } 
		\prod_{r=1}^2\prod_{s=1}^2 \frac{1}{1-\sqrt{\frac{t_i}{t_{i+1}}} \frac{w^{i+1}_r}{w^{i}_s} }\right)
		\frac{1}{1-\frac{1}{w_1^kw_2^k} }\prod_{j=1}^2 	\frac{1}{1-\frac{w_j^1 c_1}{\sqrt{t_1}}}  \frac{1}{1-\frac{c_2\sqrt{t_k }}{w_j^k}}, 
		\label{eq:Laplacetransform1} 
	\end{multline}
	where we have used the short-hand notations $w^i=(w^i_1, w^i_2)$  and   $\nu(dw) = \Delta(w)\frac{dw_1}{2\I\pi w_1}\frac{dw_2}{2\I\pi w_2}$. For each $1\leq i\leq k$, the variables $w_1^i $ and $ w_2^i$ are integrated along a circle around $0$ with radius $r_i$ where the radii satisfy 
	\begin{equation}
		\frac{\sqrt{t_1}}{c_1}>r_1>r_2>\dots>r_k>1 \text{ and } r_i<\frac{1}{a_r \sqrt{t_i}} \text{ for }x_{i-1}<r\leq x_i.
		\label{eq:conditionsoncontours}
	\end{equation}
	\label{prop:formulaSchur}
\end{proposition}
\begin{remark}
	In Proposition \ref{prop:formulaSchur}, we could assume  slightly weaker assumptions on $t_i$. It suffices that  $a_jr_i \sqrt{t_i}<1$, $r_{i+1}<r_{i}$, $\sqrt{t_i}r_{i+1}<\sqrt{t_{i+1}}r_{i}$, $r_k>1$, $c_1r_1<\sqrt{t_1}$ and $c_2\sqrt{t_k}<r_k$. This ensures  that all summations performed in the proof of Proposition \ref{prop:formulaSchur} below are absolutely convergent so that the formula \eqref{eq:Laplacetransform1} still holds.
	\label{rem:weakerassumptions} 
\end{remark}
\begin{proof}[Proof of Proposition \ref{prop:formulaSchur}]
By definition, using \eqref{eq:marginaldistributionSchur}, 
\begin{multline}  	\mathbb E\left[\prod_{i=1}^k t_i^{\lambda_1^{x_i}-\lambda_1^{x_{i-1}}} \right] = \\  \frac{1}{\mathcal Z^{\vec a, c_1,c_2}_{\rm Geo}(N)} \sum_{\mu^{0}, \dots, \mu^{k}\in \Signplus_2}\mathds{1}_{\mu_2^{0}=0}\; c_1^{\mu_1^{0}-\mu_2^{0}}c_2^{\mu_1^{k}-\mu_2^{k}} \prod_{i=1}^k s_{\mu^{i}/\mu^{{i-1}} }(a_{x_{i-1}+1},\dots,a_{x_i}) t_i^{\mu_1^{i}-\mu_1^{{i-1}}}.
	\label{eq:laplacetransformstarting}
	\end{multline}
 Writing skew Schur functions as integrals, we obtain 
\begin{multline}  	\mathbb E\left[\prod_{i=1}^k t_i^{\lambda_1^{x_i}-\lambda_1^{x_{i-1}}} \right] =   \frac{1}{\mathcal Z^{\vec a, c_1,c_2}_{\rm Geo}(N)}  \sum_{\boldsymbol\mu\in (\Signplus_2)^k}\mathds{1}_{\mu_2^{0}=0} \oint\!\!\!\oint \nu(dz^1) \dots \oint\!\!\!\oint \nu(dz^k) 
 c_1^{\mu_1^{0}-\mu_2^{0}} c_2^{\mu_1^{k}-\mu_2^{k}} \\ \times  
 \prod_{i=1}^k \left(  \prod_{j=1}^2 \prod_{r=x_{i-1}+1}^{x_i} \frac{1}{1-a_r z_j^i}  \times s_{\mu^{i}}(\sqrt{t_i}  (z^i)^{-1}) s_{\mu^{{i-1}} }(z^i (\sqrt{t_i})^{-1})  \frac{\sqrt{t_i}^{\mu_1^{i}-\mu_2^{i}}}{\sqrt{t_i}^{\mu_1^{{i-1}}-\mu_2^{{i-1} } } }  \right), 
 \label{eq:applicationskewformula}
\end{multline}
where we have used the short hand notations $z^i=(z^i_1, z^i_2)$  and   $\nu(dz) = \Delta(z)\frac{dz_1}{2\I\pi z_1}\frac{dz_2}{2\I\pi z_2}$. To obtain \eqref{eq:applicationskewformula}, we have simply applied Lemma \ref{lem:skewSchur} to replace all skew Schur functions $s_{\mu^{i}/\mu^{{i-1}} }(a_{x_{i-1}+1},\dots,a_{x_i})$ in \eqref{eq:laplacetransformstarting} by integrals, and we have written 
 $$t_i^{\mu_1^{i}}=\sqrt{t_i}^{(\mu_1^{i}-\mu_2^{i})+ (\mu_1^{i}+\mu_2^{i})}$$
and likewise for $t_i^{\mu_1^{{i-1}}}$, so that 
$$ s_{\mu^{i}}(\sqrt{t_i} (z^i)^{-1}) s_{\mu^{{i-1}} }((\sqrt{t_i})^{-1}z^i) t_i^{\mu_1^{i}-\mu_1^{{i-1}}} = s_{\mu^{i}}(\sqrt{t_i} (z^i)^{-1}) s_{\mu^{{i-1}} }((\sqrt{t_i})^{-1}z^i)  \frac{\sqrt{t_i}^{\mu_1^{i}-\mu_2^{i}}}{\sqrt{t_i}^{\mu_1^{{i-1}}-\mu_2^{{i-1} } } }.$$
Initially, the variables in \eqref{eq:applicationskewformula} are integrated along  circles around $0$ with radius $1$, but we may deform the contours  so that  $\vert z_1^i\vert =\vert z_2^i\vert=R_i$ with $R_i<\min_{x_{i-1}<r\leq x_i}\lbrace a_r^{-1}\rbrace$.  The summations performed in the following steps will impose further  conditions that the $R_i$ must satisfy.  
  
Now, we may sum over the $\mu^{i}$ explicitly to greatly simplify the formula \eqref{eq:applicationskewformula}. The sum over $\mu^{0}$ is computed using  
\begin{multline}  \sum_{\mu^{0}\in \Signplus_2} \mathds{1}_{\mu_2^{0}=0}\;\;  c_1^{\mu_1^{0}-\mu_2^{0}}s_{\mu^{{0}} }\left((\sqrt{t_1})^{-1}z^1 \right)  \frac{1}{\sqrt{t_1}^{\mu_1^{{0}}-\mu_2^{{0} }}}  = \\ 
	\sum_{\mu_1^{0}\in \mathbb Z_{\geq 0}}  c_1^{\mu_1^{0}}s_{(\mu_1^{{0}}, 0) }\left((\sqrt{t_1})^{-1}z^1\right)  \frac{1}{\sqrt{t_1}^{\mu_1^{0}}}  =  \prod_{j=1}^2 \frac{1}{1-\frac{z_j^1c_1}{t_1}}, 
\label{eq:summu0}
\end{multline}
where in the second equality, we have used that for any integer $m\geq 0$, the Schur functions  $s_{(m,0)}(z_1,z_2) = h_m(z_1,z_2)$, where $h_m$ is the $m$th complete homogeneous symmetric polynomial.  
The sum  \eqref{eq:summu0} is convergent as long as $\vert z_j^1 \vert =R_1$ with  $c_1 R_1<t_1$. The sum over $\mu^{i}$ when $0< i<k$ is computed, using Lemma \ref{lem:specialcomputation}, as
$$\sum_{\mu^{i}\in \Signplus_2} s_{\mu^{i}}(\sqrt{t_i} (z^i)^{-1})s_{\mu^{{i}} }\left((\sqrt{t_{i+1}})^{-1}z^{i+1}\right) \sqrt{\frac{t_i}{t_{i+1}}}^{\mu_1^{i}-\mu_2^{i}}= \frac{1-\frac{t_i^2 z_1^{i+1}z_2^{i+1} }{t_{i+1}^2 z_1^{i}z_2^{i} }   }{1- \frac{t_i}{t_{i+1}}\frac{ z_1^{i+1}z_2^{i+1} }{z_1^{i}z_2^{i} } } 
\prod_{r=1}^2\prod_{s=1}^2 \frac{1}{1-\frac{t_i}{t_{i+1}} \frac{z^{i+1}_r}{z^{i}_s} }.$$
This sum is convergent as long as $R_{i+1} \sqrt{t_i}<R_{i}\sqrt{t_{i+1}}$ and $R_{i+1} t_i<R_{i}t_{i+1}$. These conditions can be satisfied for many choices of $t_i$ and $R_i$.  For instance, we note that conditions are satisfied  if  we assume that  $R_{i+1} t_i<R_{i}t_{i+1}$ and $t_i\geq t_{i+1}$, so that,  strictly speaking,  it is not necessary to impose that the $t_i$ are ordered as in the statement of the Proposition, though we will see below that it is more convenient to make this assumption.

The sum over $\mu^{k}$ is computed using \eqref{eq:Littlewood} as
$$  \sum_{\mu^{k}\in \Signplus_2} s_{\mu^{k}}(\sqrt{t_k} (z^k)^{-1}) (c_2\sqrt{t_k})^{\mu_1^{{k}}-\mu_2^{{k} }} = \frac{1}{1-\frac{t_k}{z_1^kz_2^k} }\prod_{j=1}^2 \frac{1}{1-\frac{c_2t_k }{z_j^k}} .$$
This sum is convergent as long as $t_k c_2<R_k$ and $t_k<R_k^2$. 
To summarize, after summing over $\mu^i$ for $1\leq i\leq k$ in \eqref{eq:applicationskewformula}, we obtain that 
\begin{multline}  	\mathbb E\left[\prod_{i=1}^k t_i^{\lambda_1^{x_i}-\lambda_1^{x_{i-1}}} \right] =   \frac{1}{\mathcal Z^{\vec a, c_1,c_2}_{\rm Geo}}   \oint\!\!\!\oint \nu(dz^1) \dots \oint\!\!\!\oint \nu(dz^k) 
	\prod_{i=1}^k   \prod_{j=1}^2 \prod_{r=x_{i-1}+1}^{x_i} \frac{1}{1-a_r z_j^i}
 \\
\times \prod_{i=1}^{k-1} \left(\frac{1-\frac{t_i^2 z_1^{i+1}z_2^{i+1} }{t_{i+1}^2 z_1^{i}z_2^{i} }   }{1- \frac{t_i}{t_{i+1}}\frac{ z_1^{i+1}z_2^{i+1} }{z_1^{i}z_2^{i} } } 
\prod_{r=1}^2\prod_{s=1}^2 \frac{1}{1-\frac{t_i}{t_{i+1}} \frac{z^{i+1}_r}{z^{i}_s} }\right)
	\frac{1}{1-\frac{t_k}{z_1^kz_2^k} }\prod_{j=1}^2 	\frac{1}{1-\frac{z_j^1 c_1}{t_1}} \frac{1}{1-\frac{c_2t_k }{z_j^k}}.  
\end{multline}
At this point, it is convenient to perform the change of variables $z^i_j=\sqrt{t_i}w^i_j$. We obtain
\begin{multline}  	\mathbb E\left[\prod_{i=1}^k t_i^{\lambda_1^{x_i}-\lambda_1^{x_{i-1}}} \right] =   \frac{1}{\mathcal Z^{\vec a, c_1,c_2}_{\rm Geo}}   \oint\!\!\!\oint \nu(dw^1) \dots \oint\!\!\!\oint \nu(dw^k) 
	\prod_{i=1}^k   \prod_{j=1}^2 \prod_{r=x_{i-1}+1}^{x_i} \frac{1}{1-a_r \sqrt{t_i}w_j^i}
	\\
	\times \prod_{i=1}^{k-1} \left(\frac{1-\frac{t_i w_1^{i+1}w_2^{i+1} }{t_{i+1} w_1^{i}w_2^{i} }   }{1- \frac{ w_1^{i+1}w_2^{i+1} }{w_1^{i}w_2^{i} } } 
	\prod_{r=1}^2\prod_{s=1}^2 \frac{1}{1-\sqrt{\frac{t_i}{t_{i+1}}} \frac{w^{i+1}_r}{w^{i}_s} }\right)
	\frac{1}{1-\frac{1}{w_1^kw_2^k} }\prod_{j=1}^2 	\frac{1}{1-\frac{w_j^1 c_1}{\sqrt{t_1}}}  \frac{1}{1-\frac{c_2\sqrt{t_k }}{w_j^k}},
	\label{eq:Laplacetransform3}
\end{multline}
where the contours are now circles with radii $r_i$, with $r_i=R_i/\sqrt{t_i}$. If the $t_i$ are ordered as in the statement of Proposition \ref{prop:formulaSchur}, then the rather simple conditions on $r_i$ given in \eqref{eq:conditionsoncontours} imply that all the inequalities that the $R_i$ must satisfy above are indeed satisfied. This concludes the proof of Proposition \ref{prop:formulaSchur}. 

As mentioned in Remark \ref{rem:weakerassumptions}, it is not necessary to impose that the $t_i$ are ordered. It suffices that the $R_i$ satisfy all conditions given above, and using $R_i=\sqrt{t_i}r_i$, this leads to the conditions given in Remark \ref{rem:weakerassumptions}. 
\end{proof}

\begin{proof}[Proof of Theorem \ref{theo:formulaSchur}]
We start from  the Laplace transform formula from Proposition \ref{prop:formulaSchur}, that is the formula \eqref{eq:Laplacetransform3} just above. Consider the integration over the variable $w_2^k$. There is a pole inside the contour at $w_2^k=c_2\sqrt{t_k}$ and at $w_2^k=1/w_1^k$.  One notices that the  first one yields no contribution, since after computing the residue at $w_2=c_2\sqrt{t_k}$, there are no longer poles inside the contour for the variable $w_1^k$. Here we have used the additional condition $c_2<1$ which we did not assume in Proposition \ref{prop:formulaSchur} but is necessary in Theorem \ref{theo:formulaSchur}. Hence, we can simply replace the integration by the residue at $w_2^k=1/w_1^k$.  We will now consider the integration over variables $w_2^{k-1}, w_2^{k-2}, \dots$ sequentially. At first sight, for some $1\leq i\leq k-1$, the integral over $w_i^2$ does not seem to have a pole at $1/w_i^1$.  However, the factor 
$$ \frac{1 }{1- \frac{ w_1^{i+1}w_2^{i+1} }{w_1^{i}w_2^{i} } }  $$
in \eqref{eq:Laplacetransform3}, after computing the residue at $w_2^{i+1}=1/w_1^{i+1}$, becomes 
$$ \frac{1 }{1- \frac{ 1 }{w_1^{i}w_2^{i} } }  $$
so that there is actually a simple pole at $w_2^i=1/w_1^i$ when we perform the integrations sequentially. There are two other poles in the variable $w_2^i$ at $w_2^i = \sqrt{\frac{t_i}{t_{i+1}}}w_{i+1}^1$ and at $w_2^i = \sqrt{\frac{t_i}{t_{i+1}}}\frac{1}{w_{i+1}^1}$. One notices that after computing these last two residues in the variable $w_2^i$, the factors in $\Delta(w^{i})$ (present in the measure $\nu(dw^i)$) exactly cancel with some denominators in the formula, so that eventually, there are no poles in the variable $w_1^i$ inside the contours. Thus, the integral over $w_2^i$ can be replaced by a residue at $w_2^i=1/w_1^i$, and one can then consider the integral over $w_2^{i-1}$.

Repeating the procedure  until computing the residue at $w_2^1=1/w_1^1$, and renaming variables as  $z_i=w_1^i$ for $1\leq i\leq k$, we obtain 
\begin{multline}  	\mathbb E\left[\prod_{i=1}^k t_i^{\lambda_1^{x_i}-\lambda_1^{x_{i-1}}} \right] =  \\  \frac{1}{\mathcal Z^{\vec a, c_1,c_2}_{\rm Geo}}   \oint\frac{dz_1}{2\I\pi z_1}\dots  \oint\frac{dz_k}{2\I\pi z_k}
	\prod_{i=1}^k \frac{ (1-z_i^2)\left( 1-1/z_i^2\right) }{2} \prod_{r=x_{i-1}+1}^{x_i} \frac{1}{(1-a_r \sqrt{t_i}z_i)(1-a_r \sqrt{t_i}/z_i)}
	\\
	\times \prod_{i=1}^{k-1} \left(\frac{1-\frac{t_i  }{t_{i+1}  }   }{     \left(1-\sqrt{\frac{t_i}{t_{i+1}} }z_iz_{i+1} \right)\left(1-\sqrt{\frac{t_i}{t_{i+1}} }\frac{z_i}{z_{i+1}} \right) \left(1-\sqrt{\frac{t_i}{t_{i+1}} }\frac{z_{i+1}}{z_{i}} \right)\left(1-\sqrt{\frac{t_i}{t_{i+1}} }\frac{1}{z_{i}z_{i+1} }     \right)  } \right)\\ \times 
	\frac{1}{\left(1-\frac{z_1 c_1}{\sqrt{t_1}}\right)\left(1-\frac{ c_1}{z_1\sqrt{t_1}}\right)}  \frac{1}{\left(1-\frac{c_2\sqrt{t_k }}{z_k}\right)\left(1- z_k c_2\sqrt{t_k }\right)},  
	\label{eq:formulaschurdetails}
\end{multline}
which is exactly \eqref{eq:Laplacetransform2}. 
This concludes the proof of Theorem  \ref{theo:formulaSchur}.
\end{proof}

\subsection{Proof of Theorem \ref{theo:formulaLPPintro}} 
\label{sec:proofLPP}
For $c_1,c_2<1$, in view of the identity in distribution \eqref{eq:twodescriptions}, the formula stated in Theorem \ref{theo:formulaLPPintro} is simply that of Theorem \ref{theo:formulaSchur} in the special case where all parameters $a_i$ are equal. It remains to justify the analytic continuation for other values of $c_1,c_2$. 

Given the definition of the probability measure $\mathbb P_{\rm Geo}^{a, c_1,c_2}$ in \eqref{eq:defstationarymeasureLPPintro}, the multipoint probability distribution of $\mathbf L_1, \mathbf L_2$ is a power series in $c_1$ and $c_2$ (this was  already remarked and used in \cite{barraquand2023stationary}). Thus, the expectation 
\begin{equation}
	\mathbb E_{\rm Geo}^{a, c_1,c_2}\left[\prod_{i=1}^k t_i^{2(L_1(x_i)-L_1(x_{i-1}))} \right]
	\label{eq:expextationpowerseries}
\end{equation} 
is likewise a power series. One needs to impose conditions on the $t_i$ such that the expectation is an absolutely convergent series. For that, it is convenient to observe that from the probabilistic definition \eqref{eq:defstationarymeasureLPPintro} 
$$ t^{2L_1(N)} \PP^{a,c_1,c_2}_{\rm Geo}(\mathbf L_1,\mathbf L_2) \leq C \PP^{at,c_1/t,c_2t}_{\rm Geo}(\mathbf L_1,\mathbf L_2) $$
where the constant $C$ does not depend on $\mathbf L_1,\mathbf L_2$. Thus, $\mathbb E_{\rm Geo}^{a, c_1,c_2}\left[t^{2L_2(N)}\right]$ is summable as long as $\PP^{at,c_1/t,c_2t}_{\rm Geo}$ is a valid probability measure. We know that this is the case when $at<1$, $at^2c_2<1$ and $ac_1<1$. This is why we assume that $at_k<1$ and $at_k^2c_2<1$. 
In such case, the absolutely convergent power series is an analytic function in $c_1, c_2$. Hence, it is equal to the analytic continuation of the integral formula in the right-hand side of \eqref{eq:formulaschurdetails}. One can perform the analytic continuation first with respect to $c_1$ and then with respect to $c_2$, or vice-versa. The result must be the same whatever the order, since it is equal to \eqref{eq:expextationpowerseries}. This concludes the proof of Theorem \ref{theo:formulaLPPintro}

\begin{remark} One could express the analytic continuation of \eqref{eq:formulaschurdetails}  as the sum of an integral term and some residues, depending on the values of $c_1,c_2$. We do not make this analytic continuation more explicit since we do not use it for any application. However, in the log-gamma case, we will  make some analytic continuations explicit, see in particular \eqref{eq:analyticcontinuation}, as we need them to prove results about the KPZ equation without restrictions on boundary parameters. 
\end{remark}

\subsection{Markovian description of the two-layer open Schur process}
\label{sec:MarkovianSchur}
In this Section we describe the two-layer Schur process \eqref{eq:deftwolayerSchur} as a Markov process with explicit, though complicated, transition probabilities and initial distribution. This alternative description is easy to obtain.  It follows directly from the definitions and the branching property of Schur polynomials. We then show that, using the same techniques as above, one can write explicit integral formulas for the transition probabilities, and in the large $N$ limit, those simplify, leading to a simple description of the stationary measure as a Doob transformed process.

Define the probability measure 
\begin{equation*}
 	\mathsf p_0(\la) = c_1^{\la_1-\la_2} h_{0,N}(\la_1-\la_2)\mathds{1}_{\la_2=0},
\end{equation*}
 where the functions $h_{x,N}:\Z_{\geq 0}\to \mathbb R_+$ are defined  by
 \begin{equation}
 	h_{x,N}(\ell)  = \frac{q_{N-x}(\ell)}{\sum_{\ell\geq 0} c_1^\ell q_N(\ell)},
 	\label{eq:defh}
 \end{equation} 
 with, for $M\leq N$, 
 \begin{equation}
 	q_M(\ell) = \prod_{i=N-M+1}^N(1-a_i)^2 \sum_{\la\in \Sign_2} s_{\la/(\ell,0)}\left(a_{N-M+1}, \dots, a_N\right) c_2^{\la_1-\la_2}, 
 	\label{eq:defq}
 \end{equation}
We also define transition probabilities
\begin{equation*}
	\mathsf p_{x,y}(\la, \mu) = \prod_{i=x+1}^y(1-a_i)^2 s_{\mu/\la}\left(a_{x+1}, \dots, a_y\right)\frac{h_{y,N}(\mu_1-\mu_2)}{h_{x,N}(\la_1-\la_2)}. 
	\label{eq:transitionkernel}
\end{equation*}
Using the branching rule \eqref{eq:branchingrule} and the fact that $s_{\mu/\la}$ remains invariant when adding the same integer to $\la_1,\la_2, \mu_2, \mu_2$, one may  check that $\mathsf p_{x,y}(\la, \mu)$ defines a probability measure on $\mu\in \Sign_2$ for every given  $\la$ and $x<y$. The branching rule also implies that $\mathsf p_{x,y}$ satisfy a semi-group like property 
\begin{equation}
 \sum_{\kappa\in \Sign_2}\mathsf p_{x,y}(\la, \kappa) \mathsf p_{y,z}( \kappa, \mu) =  \mathsf p_{x,z}(\la, \mu).
\end{equation}

\begin{proposition}
The process $x\mapsto \la^x$ defined by \eqref{eq:deftwolayerSchur} is a time-inhomogeneous Markov chain with transition kernel $\mathsf p_{x,y}(\la^x, \la^y)$ and initial distribution $\mathsf p_0(\la^0)$.
\label{prop:markoviandescription}
\end{proposition}
\begin{proof}
It suffices to notice that for any $0=x_0<x_1<\dots<x_k$,  the product 
$$\mathsf p_0(\mu^0) \prod_{i=1}^k\mathsf p_{x_{i-1},x_i}(\mu^{i-1}, \mu^i),$$
simplifies and matches with the marginal distribution of the two-layer Schur process \eqref{eq:marginaldistributionSchur}. 
\end{proof}
The Markovian description can even be generalized to the more general two-layer Schur processes in Definition \ref{def:generaltwolayerSchur}. Let 
$$ \mathsf p_x^{\rightarrow}(\lambda, \mu) = (1-a_x)^2s_{\mu/\lambda}(a_{x})\frac{h_{x,N}(\mu_1-\mu_2)}{h_{x-1,N}(\lambda_1-\lambda_2)} $$
 and 
 $$ \mathsf p_x^{\downarrow}(\lambda, \mu) = (1-a_x)^2s_{\lambda/\mu}(a_{x})\frac{h_{x,N}(\mu_1-\mu_2)}{h_{x-1,N}(\lambda_1-\lambda_2)}. $$
\begin{proposition}
The process $x\mapsto \lambda^x$ from Definition \ref{def:generaltwolayerSchur}, associated with a word $w\in \lbrace \rightarrow, \downarrow\rbrace^N$, is a (time-inhomogeneous) Markov process with kernel $\mathsf p_x^{w_x}(\lambda^{x-1},\lambda^x)$ and initial distribution $\mathsf p_0(\la^0)$.
\label{prop:Markoviangeneral}
\end{proposition}
\begin{proof}
Again,  this relies on the same simplifications of products of ratios of functions $h_{x,N}$ as in the proof of Proposition \ref{prop:markoviandescription}, as well as the following property. For any $0\leq x\leq N$, the sum 
\begin{equation} \frac{1}{\mathcal Z^{\vec a, c_1,c_2}_{\rm Geo}(N)}  \sum_{\la^{x+1}, \dots, \la^N\in\Sign_2} c_2^{\lambda_1^N-\lambda_2^N} \prod_{\footnotesize \begin{matrix}
		x+1\leq i\leq N \\ w_i=\rightarrow
\end{matrix}} s_{\lambda^i/\lambda^{i-1} }(a_i)\prod_{\footnotesize \begin{matrix}
		x+1\leq i\leq N \\ w_i=\downarrow
\end{matrix}} s_{\lambda^{i-1}/\lambda^{i} }(a_i)
\label{eq:sumarbitrarypath}
\end{equation}
does not depend on the $w_i$ and equals $h_{x,N}(\la_1^x-\la^x_2)$, which,  by definition \eqref{eq:defh},   corresponds to the sum \eqref{eq:sumarbitrarypath} in the case where $w_x=\rightarrow$ for all $x$. As in the proof of \cite[Proposition 2.15]{barraquand2023stationary}, one uses the Cauchy and Littlewood identities \cite[Propositions 2.6 and 2.7]{barraquand2023stationary}: For any $\mu,\nu\in\Sign_2$ and $a,b$ such that $\vert ab\vert<1$, 
\begin{equation}
	\sum_{\la\in \Sign_2} s_{\la/\mu}(a)s_{\la/\nu}(a) = \sum_{\kappa\in \Sign_2}s_{\mu/\kappa}(b)s_{\nu/\kappa}(a)
	\label{eq:skewCauchysignatures}
\end{equation} 
and for $\mu\in \Sign_2$ and $a,c$ such that $\vert ac\vert<1$, 
\begin{equation}
	\sum_{\la\in \Sign_2} c^{\lambda_1-\lambda_2} s_{\la/\mu}(a) = \sum_{\kappa\in \Sign_2} c^{\kappa_1-\kappa_2}s_{\mu/\kappa}(a).
	\label{eq:skewLittlewoodsignatures}
\end{equation}
Successive applications of identities \eqref{eq:skewCauchysignatures} and \eqref{eq:skewLittlewoodsignatures} allow to transform the sum \eqref{eq:sumarbitrarypath} into 
\begin{multline} \frac{1}{\mathcal Z^{\vec a, c_1,c_2}_{\rm Geo}(N)}   \sum_{\la^{x+1}, \dots, \la^N} c_2^{\lambda_1^N-\lambda_2^N} \prod_{
		x+1\leq i\leq N } s_{\lambda^i/\lambda^{i-1} }(a_i) = \\  \sum_{\la^N\in \Sign_2} c_2^{\lambda_1^N-\lambda_2^N} s_{\lambda^N/\lambda^x}(a_{x+1},\dots, a_N) = \frac{q_{N-x}(\lambda^x_1-\lambda^x_2)}{\mathcal Z^{\vec a, c_1,c_2}_{\rm Geo}(N)} =   h_{x,N}(\lambda^x_1-\lambda^x_2).
\end{multline}
Hence, we obtain that for all $0\leq k\leq N$,  the product 
$$\mathsf p_0(\mu^0) \prod_{x=1}^k\mathsf p_{x}(\mu^{i-1}, \mu^i)$$
coincides with with the marginal distribution of $\la^0,\dots,\la^k$ in \eqref{eq:deftwolayerSchur}, which concludes the proof.
\end{proof}

 \subsubsection{Application}
Before giving an application of the Markovian description above, let us first  establish an integral formula for the function $q_M$ arising in the definition of the function $h_{x,N}$. 
\begin{proposition} For $c_1, c_2<1$, and any $\ell\in \Z_{\geq 0}$, we have
	\begin{equation}
		q_M(\ell) = \frac{1}{2} \oint \frac{dz}{2\I\pi z} \left(z^{\ell+1} - \frac{1}{z^{\ell+1}} \right)\left(\frac{1}{z} - z \right) \frac{1}{1-c_2z}\frac{1}{1-c_2/z}\prod_{i=N-M+1}^N\left(\frac{1-a_i}{1-a_iz}\frac{1-a_i}{1-a_i/z} \right),
		\label{eq:expressionqgeneral}
	\end{equation}
	where the contour is a circle of radius $1$ around the origin. 
	If $c_1>1$ or $c_2>1$, $q_M(\ell)$ is given by the analytic continuation of the formula above. 
	\label{prop:formulaqgeneral}
\end{proposition}
\begin{proof}
In \eqref{eq:defq}, we may express the Schur function $s_{\la/(\ell,0)}$ using Proposition \ref{lem:skewSchur}. Summing over $\la$ using the Littlewood identity \eqref{eq:Littlewood}, we obtain 
$$ q_M(\ell) =  \oint\!\!\!\oint \nu(dz)s_{(\ell,0)}(z_1,z_2) \frac{1}{1-1/z_1z_2}\frac{1}{(1-c_2/z_1)(1-c_2/z_2)}\prod_{i=N-M+1}^N\frac{(1-a_i)^2}{(1-a_iz_1)(1-a_iz_2)},$$
where the contour is a circle around zero with radius slightly larger than $1$. Then, for the same reason as in the proof of Theorem \ref{theo:formulaSchur}, the integral over $z_2$ is equal to the residue at $z_2=1/z_1$. Computing this residue and using that 
$$ s_{(\ell,0)}(z_1,z_2)  = \frac{z_1^{\ell+1}-z_2^{\ell+1}}{z_1-z_2}, $$ 
we obtain \eqref{eq:expressionqgeneral}. 
\end{proof}

Using the explicit expression above, one can now study limits as $N$ goes to infinity. 
\begin{proposition}
Assume $a_i\equiv a\in (0,1)$ and $c_2<1$. As $M\to \infty$, 
\begin{equation*}
q_M(\ell) \sim_{M\to\infty} \frac{1}{(1-c_2)^2}\frac{1}{2\sqrt{\pi}}\left( \frac{(1-a)^2}{a M}\right)^{3/2} (\ell+1)
\end{equation*}
so that  as $N\to\infty$,
\begin{equation}
	h_{x,N}(\ell)\xrightarrow[N\to\infty]{}(\ell+1)(1-c_1)^2.
\end{equation}	
\label{prop:limitq}
\end{proposition} 
\begin{proof}
In \eqref{eq:expressionqgeneral}, one may perform the change of variables $z=e^{i x/\sqrt{M}}$, so that as $M$ going to infinity, 
$$ q_M(\ell) \sim_{M\to\infty} \frac{1}{(1-c_2)^2}\frac{\ell+1}{\pi M^{3/2}} \int_{\mathbb R} x^2 e^{\frac{-a x^2}{(1-a)^2}}dx.$$
The integral can be calculated explicitly and leads to the  two limits above. 
\end{proof}
The above proposition shows that as $N$ goes to infinity, the process $x\mapsto \la^x$ (defined by \eqref{eq:deftwolayerSchur}, in the case $a_i\equiv a$, $c_2<1$) weakly converges to the Markov process with transition probability 
\begin{equation}
	 \overline{\mathsf p}(\la^{x},\la^{x+1})  = (1-a)^2 s_{\la^{x+1}/\la^x}(a) \frac{\la_1^{x+1}-\la_2^{x+1}+1}{\la_1^x-\la_2^x+1} =  (1-a)^2 s_{\la^{x+1}/\la^x}(a) \frac{s_{\lambda^{x+1}}(1,1)}{s_{\lambda^{x}}(1,1)}
	 \label{eq:limitDoobtransform}
\end{equation}
 and initial distribution 
 $$ \overline{\mathsf p}_0(\la^0) = (1-c_1)^2 c_1^{\la_1^0-\la_2^0} (\la_1^0-\la_2^0+1)\mathds{1}_{\la_2=0},$$
 that is, $\la_2^0=0$ and  $\la_1^0$ is distributed as the sum of two independent $\mathrm{Geom}(c_1)$ random variables. Furthermore, the factor $ (1-a)^2 s_{\la^{x+1}/\la^x}(a)$ in  \eqref{eq:limitDoobtransform} can be interpreted as the transition kernel for a couple of independent random walks $(\la_1^x, \la_2^x) $ with $\mathrm{Geom}(a)$ increments, killed when $\la_2^x>\la_1^{x-1}$. Then, $\overline{\mathsf p}(\la^{x},\la^{x+1})$ can be interpreted as the transition  probability of the same killed random walk conditioned to survive forever \cite{o2003conditioned}. 
 
 In \cite{barraquand2023stationaryloggamma}, stationary models of last passage percolation and the log-gamma polymer in a half-space were studied. In the last passage percolation case, \cite[Sec. 3]{barraquand2023stationaryloggamma} defines a stationary last passage percolation model $G_{r,s}^{\rm stat}(n,m)$, stationary  in the sense that  the distribution of the process $k\mapsto G_{r,s}^{\rm stat}(m+k,m)-G_{r,s}^{\rm stat}(m,m)$ is the same for each $m\geq 2$. In the case $m=2$, $G_{r,s}^{\rm stat}(m+k,m)$ is the last passage time along  the second row  of a  model of last passage percolation in a half-quadrant. Such last passage percolation times can be seen as (a certain limit of) a  marginal of a Pfaffian Schur process \cite{borodin2005eynard} (see in particular \cite[Proposition 3.10]{baik2018pfaffian}).

 When $r=c_1$ and  $s=1$, the connection with the Pfaffian Schur process shows that  $G_{r,s}^{\rm stat}(2+k,2) - G_{r,s}^{\rm stat}(2,2)$ has the same distribution as a   process $k\mapsto \mu_1^{k}$ where  $(\mu_1^{k}, \mu_2^k)$ is a Markov process with kernel  \eqref{eq:limitDoobtransform} (the form of the transition kernel is implied by results of \cite{o2003conditioned}), and using the definition of $G_{r,s}$  in \cite{barraquand2023stationaryloggamma}, the initial condition is such that $\mu_1^{0}$ is the sum of two independent geometric random variables with parameter $c_1$,  and $\mu_2^0=0$. Thus, this shows that the stationary measure of geometric LPP in a strip converges to a stationary measure of LPP in a half-quadrant (here in the regime $c_1,c_2<1$, though we believe that other regimes could be studied with a similar approach).

\section{Two-layer Whittaker process} 
\label{sec:Whittaker}

\subsection{Background on (skew) Whittaker functions}  Unlike Schur functions, Whittaker functions cannot be defined as determinants, but they satisfy a branching rule that is very similar with that of Schur functions. Let us define for $\alpha\in \mathbb C$ and $x, y\in \mathbb R^n$, the kernels 
\begin{equation*}
	\Psi^{(n)}_{\alpha}(x/y):= \exp\left( -\sum_{i=1}^n \alpha(x_i-y_i) -\sum_{i=1}^n  e^{-(x_i-y_i)} - \sum_{i=1}^{n-1} e^{-(y_i-x_{i+1})}\right)
\end{equation*}
and for $x\in \mathbb R^n, y\in \mathbb R^{n-1}$, 
\begin{equation*}
	\Psi^{(n,n-1)}_{\alpha}(x/y):= \exp\left( -\alpha\left( \sum_{i=1}^nx_i - \sum_{i=1}^{n-1} y_i\right)   -\sum_{i=1}^{n-1}  e^{-(x_i-y_i)} - \sum_{i=1}^{n-1} e^{-(y_i-x_{i+1})}\right).
\end{equation*}
The usual conventions in the literature are such that when dealing with Whittaker functions, the roles of indices and variables are switched compared to the use of Schur functions. Hence, the argument $x/y$ above is analogous to the skew partitions $\lambda/\mu$ in the previous section.

The kernels $\Psi^{(n)}_{\alpha}(x/y)$ and $\Psi^{(n,n-1)}_{\alpha}(x/y)$ are sometimes referred to as Baxter operators \cite{gerasimov2008baxter}. Various conventions are used in the literature when defining Whittaker functions. In particular, compared to \cite{corwin2014tropical,o2014geometric}, our variables $x_i$ and $y_i$ are logarithms of the variables used there. 

We may then define (class-one, $\mathfrak gl_{n}(\R)$)-Whittaker functions, for $\boldsymbol{\alpha}=(\alpha_1, \dots, \alpha_n)\in \mathbb C^n$ and $x=x^n\in \mathbb R^n$, as  
\begin{equation}
	\Psi_{\boldsymbol{\alpha}}(x) = \int_{\mathbb R^{\frac{n(n-1)}{2}}} dx^{1}\dots dx^{n-1} \prod_{i=1}^n \Psi^{(i,i-1)}_{\alpha_i}(x^i / x^{i-1}).
	\label{eq:defWhittakerGivental}
\end{equation}
This is sometimes referred to as Givental's \cite{givental1997stationary} integral representation of Whittaker functions.  Although it is not obvious from this formula,  $\Psi_{\boldsymbol{\alpha}}(x)$ is symmetric in the $\alpha_i$. Viewed as functions of the variable $x$, for $\alpha_i$ such that $\Re[\alpha_i]>0$, Whittaker functions  have exponential decay at infinity in the Weyl chamber $x_1<\dots<x_n$  and doubly exponential decay at infinity away from the Weyl chamber (see \cite[Section 4]{borodin2014macdonald} for a precise statement and additional references).

 We also define, for $(\alpha_1, \dots, \alpha_k)\in \mathbb C^k$ and $x,y\in \mathbb R^n$,  \textit{skew} Whittaker functions via the branching rule
\begin{equation}
	\Psi_{\alpha_1, \dots, \alpha_k}^{(n)}(x/y) = \int_{\mathbb R^{n(k-1)}} dx^1 \dots dx^{k-1} \prod_{i=1}^k  \Psi^{(n)}_{\alpha_i}(x^{i}/x^{i-1}),
	\label{eq:defskewWhittaker}
\end{equation}
under the convention that $x^{k}=x$ and $x^0=y$.

We will also use dual Whittaker functions (the term dual is used by analogy with the theory of Macdonald symmetric functions; furthermore, they will be indeed a dual family with respect to the inner product below). For $\boldsymbol{\alpha}\in \C^n$ and $x\in \mathbb R^n$, they are simply defined as $\Psi^*_{\boldsymbol{\alpha}}(x)=e^{-e^{-x_n}}\Psi_{\boldsymbol{\alpha}}(x)$.  For $x\in \R^n$ and $\boldsymbol\beta=(\beta_1, \dots, \beta_m)\in \mathbb C^m$ with $m\geq n$, we extend the definition of (dual) Whittaker functions through the branching rule
\begin{equation}
	\Psi^*_{\boldsymbol{\beta}}(x) = \int_{\mathbb R^{n}}   \Psi^{(n)}_{\beta_1, \dots, \beta_{m-n}}(x / y) \Psi^*_{\beta_{m-n+1}, \dots, \beta_m}(y)dy.
	\label{eq:defWhittakerfunctiongeneral}
\end{equation}

Whittaker functions satisfy an analogue of the Cauchy identity: for $\Real[\alpha_i+\beta_j]>0$ and $m\geqslant n$, 
\begin{equation}
	\int_{\mathbb R^n} \Psi_{\boldsymbol{\alpha}}(x)\Psi^*_{\boldsymbol{\beta}}(x)dx  = \prod_{i=1}^n\prod_{j=1}^m \Gamma(\alpha_i+\beta_j).
	\label{eq:CauchyWhittaker}
\end{equation}
This identity was conjectured in \cite{bump1984fourier}  in the case $n=m$,  and proved in \cite{stade2001mellin, stade2002archimedean}. The case $m>n$ is proved in \cite[Corollary 3.5]{o2014geometric} using a probabilistic argument based on the geometric RSK correspondance. Whittaker functions also satisfy an analogue of the Littlewood identity \cite[Corollary 5.2]{o2014geometric}: for $s>0$,  $\Real[\alpha_i+\alpha_j]>0$ and $u\in \R$ such that $\Real[\alpha_i+u]>0$, 
\begin{equation}
	\int_{\mathbb R^n} \Psi_{\boldsymbol{\alpha}}(x) e^{-s e^{-x_n}} e^{-u(x_1-x_2+x_2-x_3+\dots)}dx = s^{-\sum_{i=1}^n \alpha_i}\prod_{i=1}^n \Gamma(\alpha_i+u)\prod_{1\leq i<j\leq n}\Gamma(\alpha_i+\alpha_j).
	\label{eq:LittlewoodWhittaker}
\end{equation}

Finally, let us mention another important property. Seen as functions of the variable  $z\in \I\mathbb R^n $, the functions $\Psi_z(x)$ form an orthonormal family with respect to the scalar product 
\begin{equation}
	\langle f,g\rangle_n = \int_{\I\mathbb R} \frac{dz_1}{2\I\pi} \dots  \int_{\I\mathbb R} \frac{dz_1}{2\I\pi} \Delta(z) f(z)\overline{g(z)}
\end{equation}
where now 
\begin{equation}
	\Delta(z) = \frac{1}{n!} \prod_{i\neq j} \frac{1}{\Gamma(z_i-z_j)}. 
	\label{eq:defDeltaWhittaker}
\end{equation}
\subsection{General definition of the two-layer open Whittaker process}
As for the two-layer Schur process, the two-layer Whittaker process in general depends on a down-right path $\mathcal P$ as in Figure \ref{fig:downrightpaths}, encoded by a word $w\in \lbrace \rightarrow, \downarrow \rbrace^N$. 
\begin{definition}
	Consider an integer $N\geq 1$, parameters $\alpha_1, \dots, \alpha_N\in (0,1)$ and boundary parameters $u,v\in \R$ such that $u+v>0$ and $u+\alpha_i>0$ for all $i,j$. Let $w\in \lbrace \rightarrow, \downarrow \rbrace^N$. The two-layer open Whittaker  process is a probability measure on sequences  $\boldsymbol\lambda=(\lambda^{0}, \dots, \lambda^{N}) \in \R^{2N+1}$ with density 
	\begin{equation*}
		\mathbb P(\boldsymbol{\lambda}) = \frac{1}{\mathcal Z_{\rm LG}^{\boldsymbol{\alpha}, u,v}(N)} e^{-u(\lambda_1^0-\lambda_2^0)}e^{-v(\lambda_1^N-\lambda_2^N)} \prod_{\footnotesize \begin{matrix}
				1\leq i\leq N \\ w_i=\rightarrow
		\end{matrix}} \Psi^{(2)}_{\alpha_i}(\lambda^i/\lambda^{i-1}) \prod_{\footnotesize \begin{matrix}
				1\leq i\leq N \\ w_i=\downarrow
		\end{matrix}} \Psi^{(2)}_{\alpha_i}(\lambda^{i-1}/\lambda^{i})\;\; \delta_0(\lambda_2^0),
	\end{equation*} 
	where the normalization constant $\mathcal Z_{\rm LG}^{\boldsymbol{\alpha}, u,v}(N)$ does not depend on the shape of the path (\cite[Proposition 3.12]{barraquand2023stationary}). 
	\label{def:generaltwolayerWhittaker}
\end{definition}

We will focus in particular on the case where $w=(\rightarrow, \dots, \rightarrow)$, where 
\begin{equation}
	\mathbb P^{\boldsymbol{\alpha}, u, v}(\boldsymbol{\lambda}) = \frac{1}{\mathcal Z_{\rm LG}^{\boldsymbol{\alpha}, u,v}(N)} e^{-u(\lambda_1^0-\lambda_2^0)}e^{-v(\lambda_1^N-\lambda_2^N)} \prod_{i=1}^N \Psi^{(2)}_{\alpha_i}(\lambda^i/\lambda^{i-1}) \delta_0(\lambda_2^0).
	\label{eq:deftwolayerWhittaker}
\end{equation} 
Given some integers $0= x_0< x_1< \dots< x_k= N$, one can apply the definition \eqref{eq:defskewWhittaker} to write the marginal density of $\lambda^{x_0}=\mu^0, \dots, \lambda^{x_k}=\mu^k$ as 
\begin{equation}
	\frac{1}{\mathcal Z_{\rm LG}^{\boldsymbol{\alpha}, u,v}(N)} e^{-u(\mu_1^0-\mu_2^0)}e^{-v(\mu_1^k-\mu_2^k)}  \prod_{i=1}^k \Psi^{(2)}_{\alpha_{x_{i-1}+1},\dots,\alpha_{x_i}}(\mu^{i}/\mu^{i-1}) \delta_0(\mu_2^0).
	\label{eq:marginaldistributionWhittaker}
\end{equation}
The main result of this Section is the following. We explain in Section \ref{sec:proofWhittaker} why it implies  Theorem \ref{theo:formulaloggammaintro} stated in the introduction. 
\begin{theorem} Assume that $u,v>0$.  Let $t_1, \dots, t_k$ be complex numbers such that $2u>\Real[t_1]>\dots >\Real[t_k]>-2v$ and for all $1\leq i\leq k$ and $x_{i-1}<r\leq x_i$, we impose that $\alpha_r+\Real[t_i/2]>0$.  Then, 
	\begin{multline}  	\mathbb E^{\boldsymbol{\alpha}, u, v}\left[\prod_{i=1}^k e^{-t_i(\lambda_1^{x_i}-\lambda_1^{x_{i-1}})} \right] =   \frac{1}{\mathcal Z_{\rm LG}^{\boldsymbol{\alpha}, u,v}(N)} \int_{\I\R}\frac{dz_1}{2\I\pi} \dots  \int_{\I\R}\frac{dz_k}{2\I\pi}    
		\prod_{i=1}^{k-1}\frac{\Gamma\left(\pm z_i\pm z_{i+1}+ \frac{t_i-t_{i+1}}{2}\right)}{2\Gamma(\pm 2z_i)\Gamma(t_i-t_{i+1})} \\ 
		\times	\frac{ \Gamma\left(\pm z_1-t_1/2+u\right)\Gamma\left(\pm z_k +t_k/2+v\right) }{2\Gamma(\pm 2z_k)} \prod_{i=1}^k    \prod_{r=x_{i-1}+1}^{x_i} \Gamma\left(\alpha_r+ \frac{t_i}{2}\pm z_i\right).
		\label{eq:LaplacetransformWhittaker}
	\end{multline}
 where each time that $\pm$ appears in the formula, we take the product of the Gamma function with a plus sign with the Gamma function with a minus sign.
	\label{th:formulaWhittaker}
\end{theorem}
The proof is given below. In the simplest case, $k=1$, we obtain that for $2u>t>-2v$, 
\begin{equation}  	\mathbb E^{\boldsymbol{\alpha}, u, v}\left[ e^{-t(\lambda_1^{N}-\lambda_1^{0})} \right] =   \frac{1}{\mathcal Z_{\rm LG}^{\boldsymbol{\alpha}, u,v}(N)} \int_{\I\R}\frac{dz}{2\I\pi} \frac{\Gamma(u-t/2\pm z)\Gamma(v+t/2\pm z)\prod_{i=1}^N\Gamma(\alpha_i+t/2\pm z)}{2\Gamma(2z)\Gamma(-2z)}.
	\label{eq:Laplacetransformspecialcase} 
\end{equation}
As a corollary, the normalization constant  has a simple integral expression.
\begin{corollary}
	When $u,v> 0$, 
	\begin{equation}
	\mathcal Z_{\rm LG}^{\boldsymbol{\alpha}, u,v}(N) = \int_{\I\R}\frac{dz}{2\I\pi}\frac{\Gamma(u\pm z)\Gamma(v\pm z)\prod_{i=1}^N \Gamma(\alpha_i\pm z)}{2\Gamma(2z)\Gamma(-2z)}.
	\end{equation}
	\label{cor:normalizationWhittaker}
\end{corollary}
\begin{proof}
	It suffices to let $t\to 0^+$ in \eqref{eq:Laplacetransformspecialcase}. 
\end{proof}

\subsection{Proof of Theorem \ref{th:formulaWhittaker}}
As in the proof of Proposition \ref{prop:formulaSchur} and Theorem \ref{theo:formulaSchur}, we will rely on an integral representation of skew Whittaker functions (Lemma  \ref{lem:integralWhittaker}), which is itself based on a variant of the Cauchy identity for Whittaker functions (Lemma \ref{lem:skewCauchyWhittaker}). 

The computation of the Laplace transform requires further integral identities. We first prove a Whittaker  analogue of Lemma \ref{lem:specialcomputation} (Lemma \ref{lem:specialcomputationWhittaker} below). Without surprise, it relies on the geometric RSK correspondence instead of the classical RSK correspondence. However, in the Whittaker case, this identity is not directly applicable to the computation of the Laplace transform of the two-layer  Whittaker process. We will need two extra identities, stated below in Lemma \ref{lem:deltaLitllewood} and Lemma \ref{lem:deltaCauchy}, which do not have analogues in Section \ref{sec:Schur}. 

Once all these preliminary  identities are established, we prove Theorem \ref{th:formulaWhittaker}. 

\begin{lemma} Let $k\geq 1$. 
There holds a variant of the Cauchy identity where, for $\boldsymbol{\alpha}=(\alpha_1, \dots, \alpha_k)\in \C^k$ and $\boldsymbol\beta=(\beta_1, \dots, \beta_n)\in \C^n$  satisfying $\Real[\alpha_i+\beta_j]>0$, 
\begin{equation}
	\int_{\mathbb R^n} \Psi^{(n)}_{\boldsymbol\alpha}(x/y)  \Psi_{\boldsymbol\beta}(x)   dx = \Psi_{\boldsymbol\beta}(y) \prod_{i=1}^k\prod_{j=1}^n \Gamma(\alpha_i+\beta_j).
	\label{eq:skewCauchyWhittaker}
\end{equation}
\label{lem:skewCauchyWhittaker}
\end{lemma}
\begin{proof} The statement could be proven by iterating the result of \cite[Corollary 2.2]{gerasimov2008baxter}, which corresponds to the special case $k=1$. We give a direct and self-contained proof:   
	By \eqref{eq:CauchyWhittaker}, we have that for $z=(z_1, \dots, z_n)\in \I\R^n$, and $\boldsymbol{\beta}\in\C^n$ with $\Real[\beta_j]>0$, 
	\begin{equation*}
	\int_{\mathbb R^n} \Psi^*_{\boldsymbol{\alpha}, z_1, \dots, z_n}(x)\Psi_{\boldsymbol{\beta}}(x)dx  = \prod_{i=1}^k\prod_{j=1}^n \Gamma(\alpha_i+\beta_j)\prod_{i,j=1}^n \Gamma(z_i+\beta_j).
	\end{equation*}
 Plugging definition \eqref{eq:defWhittakerfunctiongeneral}, we obtain that
	\begin{equation*}
	\int_{\mathbb R^n} \int_{\R^n}  \;\Psi_{\boldsymbol\alpha}^{(n)}(x/y) \Psi^*_{\mathbf z}(y)\Psi_{\boldsymbol{\beta}}(x)dxdy  = \prod_{i=1}^k\prod_{j=1}^n \Gamma(\alpha_i+\beta_j)\prod_{i=1}^n\prod_{j=1}^n \Gamma(z_i+\beta_j).
\end{equation*}
Now, we take the scalar product of both sides with $\Psi_z(y)$. Using the orthonormality, this yields 
\begin{equation*}
e^{-e^{-y_n}}	\int_{\mathbb R^n}   \Psi_{\boldsymbol\alpha}^{(n)}(x/y) \Psi_{\boldsymbol{\beta}}(x)dx  = \prod_{i=1}^k\prod_{j=1}^n \Gamma(\alpha_i+\beta_j) \left\langle \prod_{i,j=1}^n \Gamma(z_i+\beta_j), \Psi_z(y)\right\rangle_n.
\end{equation*} 
But from the Cauchy identity \eqref{eq:CauchyWhittaker} when $m=n$ and using the orthogonality, we know that the scalar product on the right is simply $\Psi^*_{\boldsymbol{\beta}}(y)$, which yields the claimed formula \eqref{eq:skewCauchyWhittaker}. Finally, by analytic continuation, the formula holds as long as $\Real[\alpha_i+\beta_j]>0$. 
\end{proof}

We arrive at the following integral representation for skew Whittaker functions. 
\begin{lemma} Let $k>n-1$. For $\alpha=(\alpha_1, \dots, \alpha_k)\in \mathbb C^k$ and $x,y\in \R^n$, 
	\begin{equation}
\Psi_{\boldsymbol{\alpha}}^{n}(y/x) =  \int_{\I\mathbb R+\eta} \frac{dz_1}{2\I\pi} \dots  \int_{\I\mathbb R+\eta} \frac{dz_n}{2\I\pi} \Delta(z)  \Psi_{\mathbf z}(x) \prod_{i=1}^k\prod_{j=1}^n \Gamma(\alpha_i+z_j) \overline{\Psi_{\mathbf z}(y)},
	\end{equation}
where the real part of the contours is such that $\Real[\alpha_i+\eta]>0$ for all $i$. 
\label{lem:integralWhittaker}
\end{lemma}
\begin{proof}
	From Lemma \ref{lem:skewCauchyWhittaker}, we have 
	\begin{equation}
		\int_{\mathbb R^n} \Psi^{(n)}_{\boldsymbol\alpha}(y/x)  \Psi_{\mathbf z}(y)   dy = \Psi_{\mathbf z}(x) \prod_{i=1}^k\prod_{j=1}^n \Gamma(\alpha_i+z_j).
	\end{equation}
Orthonormality of Whittaker functions yields 
\begin{equation*}
\Psi^{(n)}_{\boldsymbol\alpha}(y/x) =  \left\langle \Psi_{\mathbf z}(x) \prod_{i=1}^k\prod_{j=1}^n \Gamma(\alpha_i+z_j), \Psi_{\mathbf z}(y)\right\rangle_n
\end{equation*}
from which the desired formula follows directly. Due to the decay of Gamma functions on a vertical axis, the scalar product is well defined when $k>n-1$. 
\end{proof}
We also have the following Whittaker analogue of Lemma \ref{lem:specialcomputation}. 
\begin{lemma} For $u\in \mathbb R, s>0$ and  $(\alpha_1,\alpha_2)$ and $(\beta_1,\beta_2)$ such that  $\Real[\alpha_i+\beta_j+u]>0$ and $\Real[\alpha_1+\alpha_2+\beta_1+\beta_2]>0$, we have 
\begin{multline}
\int_{\mathbb R^2}  \Psi_{\alpha_1, \alpha_2}(x_1,x_2)\Psi_{\beta_1, \beta_2}(x_1,x_2)e^{-s e^{-x_2}} e^{-u(x_1-x_2)} dx_1 dx_2 = \\ s^{-(\alpha_1+\beta_1+\alpha_2+\beta_2)}\frac{\Gamma(\alpha_1+\alpha_2+\beta_1+\beta_2)}{\Gamma(\alpha_1+\alpha_2+\beta_1+\beta_2+2u)} \prod_{1\leq i,j\leq 2} \Gamma(\alpha_i+\beta_j+u).
\label{eq:specWhittaker1}
\end{multline}
\label{lem:specialcomputationWhittaker}
\end{lemma}
\begin{proof}
For $s=1$, the proof is similar with that of Lemma \ref{lem:specialcomputation}. Using the geometric RSK correspondence \cite{corwin2014tropical}, we find that 
\begin{multline}
	 \prod_{1\leq i,j\leq 2} \frac{1}{\Gamma(\alpha_i+\beta_j+u)} \int_{\mathbb R^2}  \Psi_{\alpha_1, \alpha_2}(x_1,x_2)\Psi_{\beta_1, \beta_2}(x_1,x_2)e^{-e^{-x_2}} e^{-u(x_1-x_2)} dx_1 dx_2\\ 
	 =    \prod_{1\leq i,j\leq 2} \frac{1}{\Gamma(\alpha_i+\beta_j+u)} \int_{\mathbb R^2}  \Psi_{\alpha_1+u, \alpha_2+u}(x_1,x_2)\Psi^*_{\beta_1, \beta_2}(x_1,x_2) e^{2u x_2} dx_1 dx_2 \\
	  = \mathbb E\left[e^{2u x_2}\right]
\end{multline}
where in the right-hand side, $e^{x_2}$ is an inverse Gamma random variable with parameter $\alpha_1+\beta_1+\alpha_2+\beta_2+2u$. Hence, \eqref{eq:specWhittaker1} follows from the formula for the characteristic function of the (inverse) Gamma distribution. Finally, the case where $s>0$ is arbitrary follows from the case $s=1$ and the change of variables $x_i= \tilde x_i-\log(s)$. 
\end{proof}

Up to now, the above lemmas are Whittaker analogues of lemmas in Section \ref{sec:Schur} and will play a similar role. Now, we state two further lemmas which are specific to the Whittaker case.

\begin{lemma}
	Let $u>0$ and $\eta$ such that $0\leq \eta<u$. For $z\in (\eta+\I\R)^2$, we have 
	\begin{equation*}
		\int_{\R^2}d\mu  \Psi_{z}(\mu) e^{-u(\mu_1-\mu_2)}  =  \delta_0(z_1+z_2) \Gamma(u+z_1)\Gamma(u+z_2).
	\end{equation*} 
	This should be understood in the sense that for functions $f$ that are holomorphic between the lines $-\eta-\eps+\I\R$ and $\eta+\I\R$, for some $\eps>0$, with exponential decay at infinity (in the sense that $\vert f(z)\vert \leq C e^{-c \vert z\vert}$ when $-\eta-\eps \leq \Real[z]\leq \eta$), we have  
	\begin{equation}
		\int_{\R^2}dx \int_{(\I\R+\eta)^2}\frac{dz}{2\I\pi} f(z) \Psi_{z}(x) e^{-u(x_1-x_2)}  =  \int_{\I\R+\eta} \frac{dz}{2\I\pi} \Gamma(u\pm z)f(z,-z).
		\label{eq:deltaLitllewood}
	\end{equation}
\label{lem:deltaLitllewood}
\end{lemma}
\begin{proof}
	Whittaker functions in two variables can be written as (see e.g. \cite[Section 5]{o2012directed})
	\begin{equation}
		\Psi_{z_1,z_2}(x_1,x_2) = 2 e^{-\frac{1}{2} (z_1+z_2)(x_1+x_2)}K_{z_1-z_2}(2e^{-\frac{x_1-x_2}{2}})
		\label{eq:WhittakerBessel}
	\end{equation}
	where $K$ is a Bessel function. Hence the integral identity  \eqref{eq:deltaLitllewood} could be verified explicitly using properties of Bessel functions. We will, however,  employ another proof method which has the advantage of being more general. We will reuse the argument in Lemma \ref{lem:deltaCauchy}.

	 First assume that $\eta>0$. Using the Littlewood identity \eqref{eq:LittlewoodWhittaker}, we may write, for $s>0$, 
	 \begin{multline}
	 \int_{(\I\R+\eta)^2} \frac{dz}{2\I\pi} f(z) \int_{\R^2} \Psi_{z}(x) e^{-se^{-x_2}}e^{-u(x_1-x_2)} = \\  \int_{(\I\R+\eta)^2} dz f(z) s^{-z_1-z_2} \Gamma(z_1+z_2) \Gamma(z_1+u)\Gamma(z_2+u). 
	 \label{eq:WhittakerdeltaLittlewoods}
	 \end{multline}
	 In the left-hand-side of \eqref{eq:WhittakerdeltaLittlewoods}, we claim that we may exchange the integrations over $z$ and $x$. Indeed, the definition of Whittaker functions \ref{eq:defWhittakerGivental} implies that $\vert \Psi_{z}(x)\vert \leq \Psi_{(\eta,\eta)}(x)$, so that we can integrate $\vert f(z) \Psi_{z}(x) e^{-se^{-x_2}}e^{-u(x_1-x_2)}\vert$ over $x\in \R^2$ using  \eqref{eq:LittlewoodWhittaker}, and the resulting quantity is integrable over $z$  using the decay assumption on the function $f$. Hence, by Fubini theorem, we can exchange integrations.  Now, in the right-hand side of \eqref{eq:WhittakerdeltaLittlewoods}, we may shift to the left the integration contour for $z_2$, so as to pick a pole when $z_2=-z_1$. We obtain that 
	 \begin{multline}
	 \int_{\R^2}dx	\int_{(\I\R+\eta)^2} dz f(z)  \Psi_{z}(x) e^{-se^{-x_2}}e^{-u(x_1-x_2)} =  \\ \int_{\I\R+\eta} \frac{dz_1}{2\I\pi} \int_{\I\R-\eta-\eps} \frac{dz_2}{2\I\pi}   f(z) s^{-z_1-z_2} \Gamma(z_1+z_2) \Gamma(z_1+u)\Gamma(z_2+u)
	  + \int_{\I\R+\eta} \frac{dz_1}{2\I\pi} f(z_1, -z_1)\Gamma(u\pm z_i).
	 \end{multline}
	 At this point, in the right-hand side, the contours are such that $\Re[z_1+z_2]<0$ in the first integral, so that it vanishes as $s$ goes to zero. Thus, we obtain the claimed formula \eqref{eq:deltaLitllewood} letting $s$ going to zero. Once the identity is established, we may slightly deform the contours as long as we do not encounter poles, so that the hypothesis that $\eta>0$ can be relaxed. 
\end{proof}

\begin{lemma}
Fix $u>0$ and real numbers $\eta, \xi$ such that $0\leq \eta+\xi<u$. Then, for a function $f$  holomorphic on $\eta+\I\R$, and a function $g$ holomorphic between the lines $\xi+\I\R$ and $\delta+\I\R$ for some $\delta<-\xi-2\eta$, both having exponential decay at infinity, we have 
\begin{multline}
	\int_{\mathbb R^2} dx \int_{(\eta+\I\R)^2} \frac{d\mathbf z}{(2\I\pi)^2}\int_{(\xi+\I\R)^2} \frac{d\mathbf w}{(2\I\pi)^2}  f(\mathbf z)g(\mathbf w) \Psi_{\mathbf z}(x)\Psi_{\mathbf w}(x)e^{-u(x_1-x_2)} = \\ 
	\int_{(\eta+\I\R)^2} \frac{d\mathbf z}{(2\I\pi)^2}\int_{\xi+\I\R } \frac{d w}{2\I\pi} f(\mathbf z)g(w,-w-z_1-z_2) \frac{\Gamma(\pm (z_1+w)+u)\Gamma(\pm(z_2+w)+u)}{\Gamma(2u)} .
	\label{eq:deltaCauchy}
 \end{multline}
\label{lem:deltaCauchy}
\end{lemma}
\begin{proof}
	First assume that $\xi+\eta>0$. Using Lemma \ref{lem:specialcomputationWhittaker}, we have, for $s>0$, 
	\begin{multline}
	 \int_{(\eta+\I\R)^2} \frac{d\mathbf z}{(2\I\pi)^2}\int_{(\xi+\I\R)^2} \frac{d\mathbf w}{(2\I\pi)^2}  f(\mathbf z)g(\mathbf w) 	\int_{\mathbb R^2} dx \Psi_{\mathbf z}(x)\Psi_{\mathbf w}(x)e^{-se^{-x_2}}e^{-u(x_1-x_2)} = \\ 
		\int_{(\eta+\I\R)^2} \frac{d\mathbf z}{(2\I\pi)^2}\int_{(\xi+\I\R)^2 } \frac{d \mathbf w}{2\I\pi} f(\mathbf z)g(\mathbf w) s^{-(z_1+w_1+z_2+w_2)}\\ \times \frac{\Gamma(z_1+z_2+w_1+w_2)}{\Gamma(z_1+z_2+w_1+w_2+2u)} \prod_{1\leq i,j\leq 2} \Gamma(z_i+w_j+u).  
		\label{eq:identityCauchywithspositive}
	\end{multline}
Using a similar argument as in the proof of Lemma \ref{lem:deltaLitllewood}, we may exchange the integrations over $x$ and over $z$ in the left-hand side. Moreover, in the right-hand side, if we shift to the left the integration contour for $w_2$, so as to pick the pole at $w_2=-z_1-z_2-w_1$, we  obtain 
\begin{multline}
\int_{(\eta+\I\R)^2} \frac{d\mathbf z}{(2\I\pi)^2} \int_{\xi+\I\R } \frac{d w_1}{2\I\pi} \int_{\delta +\I\R } \frac{d w_2}{2\I\pi}  f(\mathbf z)g(\mathbf w) s^{-(z_1+w_1+z_2+w_2)}\\ \times \frac{\Gamma(z_1+z_2+w_1+w_2)}{\Gamma(z_1+z_2+w_1+w_2+2u)} \prod_{1\leq i,j\leq 2} \Gamma(z_i+w_j+u) \\ 
+ \int_{(\eta+\I\R)^2} \frac{d\mathbf z}{(2\I\pi)^2} \int_{\xi+\I\R } \frac{d w_1}{2\I\pi} f(\mathbf z)g(w_1, -w_1-z_1-z_2) \frac{\Gamma(\pm (z_1+w_1)+u) \Gamma(\pm(z_2+w_1)+u)}{\Gamma(2u)},  
\end{multline}
with $\delta<-\xi-2\eta$. In the first integral, $\Re[z_1+w_1+z_2+w_2]<0$, so that the integral vanishes when $s$ goes to zero. Thus, we obtain the desired formula \eqref{eq:deltaCauchy} by letting $s$ go to zero in \eqref{eq:identityCauchywithspositive} and we can finally slightly deform the contours to relax the assumption $\xi+\eta>0$.
\end{proof}

\begin{proof}[Proof of Theorem \ref{th:formulaWhittaker}] 
The proof is different from that of Theorem \ref{theo:formulaSchur}. We will actually not obtain an analogue of Proposition \ref{prop:formulaSchur} but proceed more directly, with the help of slightly different Whittaker functions identities.  

Since the formulas will be rather intricate, we propose to the reader to first compute the normalization constant, in order to get a overall idea of the computation. By definition, 
\begin{equation}
\mathcal Z_{\rm LG}^{\boldsymbol{\alpha}, u,v}(N) = \int_{\mathbb R^2}d\mu^0 \int_{\R^2}d\mu^1 e^{-u(\mu_1^0-\mu_2^0)}e^{-v(\mu_1^1-\mu_2^1)} \Psi^{(2)}_{\boldsymbol{\alpha}}(\mu^1/\mu^0) \delta_0(\mu_2^0). 
\end{equation}
Inserting the formula from Lemma \ref{lem:integralWhittaker}, we find that $\mathcal Z_{\rm LG}^{\boldsymbol{\alpha}, u,v}(N)$ equals 
\begin{equation*}
	  \int_{\mathbb R^2}d\mu^0 \int_{\R^2}d\mu^1 \int_{\I\R}\frac{dz_1}{2\I \pi }\int_{\I\R}\frac{dz_2}{2\I\pi }\Delta(z)  e^{-u(\mu_1^0-\mu_2^0)}e^{-v(\mu_1^1-\mu_2^1)} \Psi_{z}(\mu^0) \Psi_{-z}(\mu^1)  \prod_{j=1}^2\prod_{i=1}^N \Gamma(\alpha_i+z_j) \delta_0(\mu_2^0). 
\end{equation*}
One can exchange the integral over $\mu^0$ and the integral over $z_1,z_2$ (to apply Fubini's theorem, one uses the decay in $z_1,z_2$ coming from the Gamma functions, while the decay in the variables $\mu^0_1,\mu_0^2$ will be clear in the explicit computation below). The integration over $\mu^0$  can be computed explicitly using (below we set $\mu_1^0=\mu_1, \mu_2^0=\mu_2, $  for readability)
\begin{multline}
	\int_{\R^2} d\mu_1 d\mu_2  \delta_0(\mu_2) e^{-u(\mu_1-\mu_2)} \Psi_{z_1, z_2}(\mu_1,\mu_2)  =  \\ 
	\int_{\mathbb R} d\mu_1 \int_{\R}dy  e^{-e^{-(\mu_1-y)}-e^{-y}} e^{-z_1(\mu_1-y)- z_2y} e^{-u\mu_1} =
	\Gamma(z_1+u)\Gamma(z_2+u).
	\label{eq:integralovermu0}
\end{multline}
The integral over $\mu^1$, however,  cannot be computed using the Littlewood type identity  (Lemma \ref{lem:specialcomputationWhittaker}), as we did in the proof of Proposition \ref{prop:formulaSchur}.  Instead, we should use Lemma \ref{lem:deltaLitllewood}, that is, the formula 
$$ \int_{\R^2} d\mu \Psi_{z}(\mu) e^{-v(\mu_1-\mu_2)}  = \delta_0(z_1+z_2) \Gamma(v+z_1)\Gamma(v+z_2).$$
This equation is valid when both sides are integrated against a function of $(z_1,z_2)$ with exponential decay at infinity, and one checks that the Gamma function factors in \eqref{eq:integralovermu0} do have such decay. Thus, this yields (letting $z_1=z$)
\begin{equation}
\mathcal Z_{\rm LG}^{\boldsymbol{\alpha}, u,v}(N) = \frac{1}{2}\int_{\I\R} dz \frac{\Gamma(\pm z+u)\Gamma(\pm z +v)\prod_{i=1}^N\Gamma(\pm z+\alpha_i)}{\Gamma(2z)\Gamma(-2z)},
\end{equation}
where we have used the short-hand notation $\Gamma(\pm z+c):= \Gamma(z+c)\Gamma(-z+c)$. 

\medskip 
Now, we turn to the general case. For the moment, we assume that integers $0= x_0< x_1< \dots< x_k= N$ are such that  $x_{i+1}-x_i\geq 2$ for all $1\leq i\leq k-1$.  We need to compute the average of 	$\prod_{i=1}^k e^{-t_i(\lambda_1^{x_i}-\lambda_1^{x_{i-1}})}$ with respect to the density given in \eqref{eq:marginaldistributionWhittaker}. By definition, 
	\begin{multline} 
		\mathbb E^{\boldsymbol{\alpha}, u, v}\left[\prod_{i=1}^k e^{-t_i(\lambda_1^{x_i}-\lambda_1^{x_{i-1}})} \right] =   \frac{1}{\mathcal Z_{\rm LG}^{\boldsymbol{\alpha}, u,v}(N)}
		 \int_{\mathbb R^{2(k+1)}} d\boldsymbol{\mu} \delta_0(\mu_2^{0})      e^{-u (\mu_1^{0}-\mu_2^{0})}e^{-v(\mu_1^{k}-\mu_2^{k})} \\ \times \prod_{i=1}^k \Psi^{(2)}_{\alpha_{x_{i-1}+1},\dots,\alpha_{x_i}}(\mu^{i}/\mu^{i-1})   e^{-t_i(\mu_1^{i}-\mu_1^{{i-1}})}.
	\label{eq:laplacetransformstartingW}
\end{multline}
Plugging Lemma  \ref{lem:integralWhittaker} (here we use the assumption that $x_{i+1}-x_i\geq 2$), we obtain 
\begin{multline*}  
	\mathbb E^{\boldsymbol{\alpha}, u, v}\left[\prod_{i=1}^k e^{-t_i(\lambda_1^{x_i}-\lambda_1^{x_{i-1}})} \right] =   \frac{1}{\mathcal Z_{\rm LG}^{\boldsymbol{\alpha}, u,v}(N)}
	 \int_{\mathbb R^{2(k+1)}} d\boldsymbol{\mu}  \\  \iint_{\I\R^2}\nu(dz^1) \dots  \iint_{\I\R^2}\nu(dz^k)   \delta_0(\mu_2^{0})   e^{-u (\mu_1^{0}-\mu_2^{0})}e^{-v(\mu_1^{k}-\mu_2^{k})} \\ 
	\prod_{i=1}^k   \left(\Psi_{z^i-t_i/2}(\mu^{i-1}) \Psi_{-z^i+t_i/2}(\mu^{i})  e^{-\frac{t_i}{2}(\mu_1^{i}-\mu_2^i -\mu_1^{{i-1}}+\mu_2^{i-1})} \prod_{j=1}^2 \prod_{r=x_{i-1}+1}^{x_i} \Gamma(\alpha_r+z_j^i) \right)
\end{multline*}
where $z^i=(z^i_1, z^i_2)$  and  $\nu(dz) = \Delta(z)\frac{dz_1}{2\I\pi z_1}\frac{dz_2}{2\I\pi z_2}$. Let us perform the change of variables $(z_1^i-t_i/2, z_2^i-t_i/2)\to (z_1^i, z_2^i)$ so that 
\begin{multline}  
	\mathbb E^{\boldsymbol{\alpha}, u, v}\left[\prod_{i=1}^k e^{-t_i(\lambda_1^{x_i}-\lambda_1^{x_{i-1}})} \right] =   \frac{1}{\mathcal Z_{\rm LG}^{\boldsymbol{\alpha}, u,v}(N)}
	\int_{\mathbb R^{2(k+1)}} d\boldsymbol{\mu}  \\  \iint_{\I\R^2}\nu(dz^1) \dots  \iint_{\I\R^2}\nu(dz^k)   \delta_0(\mu_2^{0})   e^{-u (\mu_1^{0}-\mu_2^{0})}e^{-v(\mu_1^{k}-\mu_2^{k})} \\ 
	\prod_{i=1}^k   \left(\Psi_{z^i}(\mu^{i-1}) \Psi_{-z^i}(\mu^{i})  e^{-\frac{t_i}{2}(\mu_1^{i}-\mu_2^i -\mu_1^{{i-1}}+\mu_2^{i-1})} \prod_{j=1}^2 \prod_{r=x_{i-1}+1}^{x_i} \Gamma(\alpha_r+t_i/2+z_j^i) \right).
	\label{eq:applicationskewformulaW}
\end{multline}
We note that after the change of variables, the integration contour for the variables $z_1^i, z_2^i$ should have become $\I\R-t_i/2$, but we may shoft the contour back to $\I\R$, provided that we do not cross singularities during the contour deformation. This is why we have assumed that $\alpha_r+\Real[t_i/2]>0$ in the statement of the theorem. 

Now, we will integrate over all the $\mu^i$ in \eqref{eq:applicationskewformulaW}. 
The integral over $\mu^0$
can be computed as before. Again, we set $\mu_1^0=\mu_1, \mu_2^0=\mu_2, $ and $t_1=t$ for readability, and since $\Real[t_1]<2u$, the integral is computed as 
\begin{multline}
	\int_{\R^2} d\mu_1 d\mu_2  \delta_0(\mu_2) e^{-(u-t/2)(\mu_1-\mu_2)} \Psi_{z_1^1, z_2^1}(\mu_1,\mu_2)  =  \\ 
	\int_{\mathbb R} d\mu_1 \int_{\R}dy  e^{-e^{-(\mu_1-y)}-e^{-y}} e^{-z_1^1(\mu_1-y)- z_2^1y} e^{-(u-t/2)\mu_1} =
	\Gamma(z_1^1-t/2+u)\Gamma(z_2^1-t/2+u).
\end{multline}

Then, the integral over $\mu^1$ can be computed using Lemma \ref{lem:deltaCauchy}. Indeed, the Gamma factors decay exponentially as $z_1,z_2$ go to infinity along vertical lines, and we recall that we have assumed that $x_{i+1}-x_i\geq 2$, so that we have enough Gamma factors to compensate those in the denominator in the measure $\nu(dz^1)$. Using that $\Real[t_1-t_2]>0$, this yields the rather intricate formula 
\begin{multline*}  
	\mathbb E^{\boldsymbol{\alpha}, u, v}\left[\prod_{i=1}^k e^{-t_i(\lambda_1^{x_i}-\lambda_1^{x_{i-1}})} \right] =   \frac{1}{\mathcal Z_{\rm LG}^{\boldsymbol{\alpha}, u,v}(N)}
	\int_{\mathbb R^{2(k+1)}} d\mu^2\dots d\mu^k  \\  \int_{\I\R}dz^1 \frac{1}{2\Gamma(2z^1)} \iint_{\I\R^2}\nu(dz^2) \dots  \iint_{\I\R^2}\nu(dz^k)   e^{-v(\mu_1^{k}-\mu_2^{k})} \\ 
	\Gamma(z^1-t_1/2+u)\Gamma(z^1-z_1^2-z_2^2-t_1/2+u) \frac{\Gamma\left(\pm(z_1^2+z^1) +\frac{t_1-t_2}{2}\right)\Gamma\left(\pm(z_2^2+z^1) +\frac{t_1-t_2}{2}\right)}{\Gamma(t_1-t_2)}\\
	\prod_{r=x_{0}+1}^{x_1} \Gamma(\alpha_r+t_1/2+z^1) \Gamma(\alpha_r+t_1/2+z^1-z_1^2-z_2^2)    
	\Psi_{-z^2}(\mu^{2})  e^{-\frac{t_2}{2}(\mu_1^{2}-\mu_2^2)} \prod_{j=1}^2 \prod_{r=x_{1}+1}^{x_2} \Gamma(\alpha_r+t_2/2+z_j^2) \\
	\prod_{i=3}^k   \left(\Psi_{z^i}(\mu^{i-1}) \Psi_{-z^i}(\mu^{i})  e^{-\frac{t_i}{2}(\mu_1^{i}-\mu_2^i -\mu_1^{{i-1}}+\mu_2^{i-1})} \prod_{j=1}^2 \prod_{r=x_{i-1}+1}^{x_i} \Gamma(\alpha_r+t_i/2+z_j^i) \right).
\end{multline*}
Eventually, the formula will simplify, but at this point, we must go on, summing over $\mu^2, \mu^3, \dots$ using Lemma \ref{lem:deltaCauchy}. After summing over all the $\mu^i$ up to $\mu^{k-1}$, using that $\Real[t_i-t_{i+1}]>0$,  we arrive at 
\begin{multline*}  
	\mathbb E^{\boldsymbol{\alpha}, u, v}\left[\prod_{i=1}^k e^{-t_i(\lambda_1^{x_i}-\lambda_1^{x_{i-1}})} \right] =   \frac{1}{\mathcal Z_{\rm LG}^{\boldsymbol{\alpha}, u,v}(N)} \int_{\I\R}dz^1 \frac{1}{2\Gamma(2z^1)} \dots \int_{\I\R}dz^{k-1} \frac{1}{2\Gamma(2z^{k-1})} \\ 
	\int_{\mathbb R^{2(k+1)}}  d\mu^k       \iint_{\I\R^2}\nu(dz^k)   e^{-(v+t_k/2)(\mu_1^{k}-\mu_2^{k})}\Psi_{-z^k}(\mu^{k}) \prod_{j=1}^2 \prod_{r=x_{k-1}+1}^{x_k} \Gamma(\alpha_r+t_k/2+z_j^k) \\ 
	\times \Gamma(z^1-t_1/2+u)\Gamma(z^1-z_1^k-z_2^k-t_1/2+u) \frac{\Gamma\left(\pm(z^{k}_1+z^{k-1}) +\frac{t_{k-1}-t_{k}}{2}\right)\Gamma\left(\pm(z_2^{k}+z^{k-1}) +\frac{t_{k-1}-t_{k}}{2}\right)}{\Gamma(t_{k-1}-t_{k})} \\ 
\prod_{i=1}^{k-2}	\frac{\Gamma\left(\pm(z^{i+1}+z^i) +\frac{t_i-t_{i+1}}{2}\right)\Gamma\left(\pm(-z^{i+1}-z_1^k-z_2^k+z^i) +\frac{t_i-t_{i+1}}{2}\right)}{\Gamma(t_i-t_{i+1})}\\
\prod_{i=1}^{k-1} \prod_{r=x_{i-1}+1}^{x_i} \Gamma(\alpha_r+t_i/2+z^i) \Gamma(\alpha_r+t_i/2-z^i-z_1^k-z_2^k).
\end{multline*}
Finally, the integral over $\mu^k$ can be computed using Lemma \ref{lem:deltaLitllewood}. Again, one readily checks that the integrand has enough decay in variables $z_1^k, z_2^k$. After applying Lemma \ref{lem:deltaLitllewood} using that $\Real[t_k]>-2v$, all occurrences of $z_1^k+z_2^k$ are replaced by zero, so that the formula above simplifies and one obtains
\begin{multline*}  
	\mathbb E^{\boldsymbol{\alpha}, u, v}\left[\prod_{i=1}^k e^{-t_i(\lambda_1^{x_i}-\lambda_1^{x_{i-1}})} \right] =   \frac{1}{\mathcal Z_{\rm LG}^{\boldsymbol{\alpha}, u,v}(N)} \int_{\I\R}dz^1 \frac{1}{2\Gamma(2z^1)} \dots \int_{\I\R}dz^{k} \frac{1}{2\Gamma(2z^{k})} \\ 
	\Gamma(v+t_k/2\pm z^k) \Gamma(\pm z^1-t_1/2+u) \prod_{i=1}^{k} \prod_{r=x_{i-1}+1}^{x_i} \Gamma(\alpha_r+t_i/2\pm z^i)  \\ 
	\prod_{i=1}^{k-1}	\frac{\Gamma\left(\pm(z^{i+1}+z^i) +\frac{t_i-t_{i+1}}{2}\right)\Gamma\left(\pm(-z^{i+1}-z_1^k-z_2^k+z^i) +\frac{t_i-t_{i+1}}{2}\right)}{\Gamma(t_i-t_{i+1})}.
\end{multline*}
Finally, denoting  $z_1^i$ by  $z_i$ for simplicity, we arrive at the claimed formula \eqref{eq:LaplacetransformWhittaker}.
\bigskip 

At this point, we have proved the theorem when the $x_i$ are such that $x_i-x_{i+1}\geq 2$. Consider now an arbitrary sequence of integers $0=x_0<x_1<\dots<x_k=N$ and a sequence   $\lambda^0, \dots, \lambda^N$ distributed according to $\mathbb P^{\boldsymbol{\alpha}, u,v}$ defined in \eqref{eq:marginaldistributionWhittaker}, with  $\boldsymbol\alpha=(\alpha_1, \dots, \alpha_N)$. It will be convenient to denote by $\mathcal I_N(\boldsymbol{\alpha}, u,v,\vec x)$ the integral expression in the right-hand side of \eqref{eq:LaplacetransformWhittaker}. 

Let $M=2N$ and $\tilde x_i=2x_i$. Consider the probability measure $\mathbb P^{\boldsymbol{\tilde{\alpha}}, u,v}$ on sequences $\tilde\lambda^0, \dots, \tilde\lambda^M \in \R^{2M}$ with inhomogeneity parameters $\boldsymbol{\tilde{\alpha}}=(e^c, \alpha_1 , e^c, \alpha_2, \dots, e^c, \alpha_N)$ where $c$ is an additional parameter. We have already proved that 
\begin{equation}
	\mathbb E^{\boldsymbol{\tilde\alpha}, u, v}\left[\prod_{i=1}^k e^{-t_i(\tilde\lambda_1^{\tilde x_i}-\tilde\lambda_1^{\tilde x_{i-1}})} \right] = \mathcal I_M(\boldsymbol{\tilde\alpha}, u, v, \vec{\tilde x})
	\label{eq:specialcasetheorem}
\end{equation}
We claim that as $c$ goes to $+\infty$, 
\begin{equation}
	e^{-c\sum_{i=1}^k t_i (x_i-x_{i-1})}\mathbb E^{\boldsymbol{\tilde\alpha}, u, v}\left[\prod_{i=1}^k e^{-t_i(\tilde\lambda_1^{\tilde x_i}-\tilde\lambda_1^{\tilde x_{i-1}})} \right] \xrightarrow{c\to +\infty} 	\mathbb E^{\boldsymbol{\alpha}, u, v}\left[\prod_{i=1}^k e^{-t_i(\lambda_1^{x_i}-\lambda_1^{x_{i-1}})} \right]
\label{eq:toprove1}
\end{equation}
and 
\begin{equation}
	e^{-c\sum_{i=1}^k t_i (x_i-x_{i-1})} \mathcal I_M(\boldsymbol{\tilde\alpha}, u, v, \vec{\tilde x}) \xrightarrow{c\to +\infty} \mathcal I_N(\boldsymbol{\alpha}, u,v,\vec x),	 
\label{eq:toprove2}
\end{equation}
so that by letting $c\to +\infty$ in \eqref{eq:specialcasetheorem}, we obtain \eqref{eq:LaplacetransformWhittaker}. Thus, it  remains to prove \eqref{eq:toprove1}  and \eqref{eq:toprove2}. 

In order to prove \eqref{eq:toprove2}, it is convenient to divide both the $k$-fold integral in $\mathcal I_M(\boldsymbol{\tilde\alpha}, u, v, \vec{\tilde x})$ and the normalization constant by $\Gamma(e^c)^{2M}$. Using 
$$ \frac{\Gamma(e^c\pm z)}{\Gamma(e^c)^2}\xrightarrow{c\to\infty} 1,$$ 
we see that $\mathcal Z_{\rm LG}^{\boldsymbol{\tilde\alpha}, u, v}(M) \xrightarrow{c\to\infty} \mathcal Z_{\rm LG}^{\boldsymbol{\alpha}, u, v}(N) $, and using 
$$ e^{-ct_i}\frac{\Gamma(e^c+t_i/2\pm z_i)}{\Gamma(e^c)^2}\xrightarrow{c\to\infty} 1,$$ 
we see (applying dominated convergence) that \eqref{eq:toprove2} holds. 

In order to prove \eqref{eq:toprove1} we recall the definition of the two-layer Whittaker measure \eqref{eq:deftwolayerWhittaker}. We will use the following asymptotics: for any $\mu\in \R^2$ and a bounded continuous function $f$, we have the limit,
\begin{equation}
	\frac{1}{\Gamma(e^c)} \int_{\mathbb R^2}dx  \Psi^{(2)}_{e^c}(\lambda/\mu) f(x) \xrightarrow[c\to\infty]{} f(0,0),
	\label{eq:limitcinfinity}
\end{equation}
where we let $\mathbf c=(c,c)\in \R^2$ and $\lambda=\mu-\mathbf c+ x $.

In probabilistic terms, this means that under the sub-Markov kernel 
$$\mathsf Q(\mu, \lambda) = \frac{1}{\Gamma(e^c)^2} \Psi^{(2)}_{e^c}(\lambda/\mu),$$ the law of  $\lambda-\mu+\mathbf c$ converges to $0$ as $c$ goes to infinity. Indeed, this kernel can be seen as the density that $\lambda_1-\mu_1$ is the logarithm of a $\mathrm{Gamma}^{-1}(e^c)$ random variable (so that as $c\to\infty$, $\lambda_1-\mu_1+c$ goes to zero), times the density that $\lambda_2-\mu_2$ is the logarithm of another independent  $\mathrm{Gamma}^{-1}(e^c)$ random variable, times a factor $e^{-e^{(\mu_1-\la_2)}}$ which goes to $1$ as $c$ goes to infinity. Using \eqref{eq:limitcinfinity}, one obtains that the law $\mathbb P^{\tilde{\boldsymbol{\alpha}}, u,v}$ of the sequence 
$$(\tilde\lambda^0, \tilde\lambda^2+ \mathbf c, \tilde\lambda^4+2\mathbf c, \dots, \tilde\lambda^M+N\mathbf c)$$ weakly converges to the law $\mathbb P^{\boldsymbol\alpha, u,v} $ of the sequence $(\lambda^0, \dots, \lambda^N)$. 
Since weak convergence implies convergence of the Laplace transforms, we obtain \eqref{eq:toprove1}, which concludes the proof of Theorem \ref{th:formulaWhittaker}.
\end{proof}

\subsection{Proof of Theorem \ref{theo:formulaloggammaintro}}
\label{sec:proofWhittaker}
In view of the identity in distribution \eqref{eq:twodescriptions} in the Whittaker case, the formula  \ref{eq:LaplacetransformWhittakerintro} is a special case of Theorem \ref{th:formulaWhittaker} when $\alpha_1=\dots=\alpha_N=\alpha$ and the $t_i$ are real. Thus, it only remains to justify the analytic continuation. 

Assume that the $t_i$ satisfy the extra assumption that $t_i>-\alpha+ \max\lbrace \vert u\vert , \vert v\vert\rbrace$.
Let us explain  why 
\begin{equation}
	\mathbb E_{\rm LG}^{\alpha, u, v}\left[\prod_{i=1}^N e^{- t_i(\mathbf L_1(i)-\mathbf L_1(i-1))} \right] 
	\label{eq:generalexpectation}
\end{equation}
is analytic in the variable $u\in (-\alpha, +\infty)$ for any fixed $v\in (-\alpha, \infty)$. For the same reason, the same expectation  is analytic in the variable $v$, for any fixed $u$. 
Letting $X_1(j)=L_1(j)-L_1(j-1)$ and $X_2(j)=L_2(j)-L_2(j-1)$, the probability density $P_u(\mathbf X)$ of the vector $\mathbf X= (X_1(1), \dots X_1(N), X_2(1), X_2(N))\in \R^{2N}$ satisfies the following  property (\cite[Prop. 3.18 and 3.16]{barraquand2023stationary}): for any $u_0>-\alpha$, there exist a complex neighborhood $U$ containing  $u_0$ such that  
\begin{itemize}
	\item For any $\mathbf X\in \R^{2N}$, the function $u\mapsto P_u(\mathbf X)$ is holomorphic on $U$. 
	\item There exists constants $C,r>0$ such that, for any $(u, \mathbf X)\in U\times\mathbf R^N$, 
	$$ \left\vert P_u(\mathbf X) \right\vert \leq Ce^{-r \sum_{j=1}^N \vert X_1(j)\vert +\vert X_2(j)\vert}.$$ 
\end{itemize}
Moreover, inspecting the proof of \cite[Prop. 3.18]{barraquand2023stationary}, the constant $r$ above can be chosen as  $r= \min\lbrace \alpha-\vert u\vert, \alpha-\vert v\vert\rbrace = \alpha-\max\lbrace \vert u\vert, \vert v\vert\rbrace$.  
Integrating the density $P_u(\mathbf X)$ against the function $\prod_{i=1}^N e^{-t_i X(j)}$ on $\mathbb R^{2N}$, which is possible as long as $t_i+r>0$, one can deduce using e.g. \cite[Lemma 3.20]{barraquand2023stationary}  that \eqref{eq:generalexpectation} is analytic in the variable $u$ on $(-\alpha, +\infty)$. 

Thus, given $u,v$ such that $\alpha+u, \alpha+v>0$, if the $t_i$ satisfy  $t_i>-\alpha+ \max\lbrace \vert u\vert , \vert v\vert\rbrace$, the value of the expectation \eqref{eq:generalexpectation} is given by the analytic continuation of the integral formula in the right-hand side of \eqref{eq:LaplacetransformWhittakerintro} in the variable $u$, or in the variable $v$, or in both variables one after the other. The latter case arises when $u,v<0$ and the result must not depend on the order in which the analytic continuations are taken. 

\subsection{Markovian description of the two-layer open Whittaker process}
\label{sec:MarkovWhittaker}
In this section we describe the two-layer Whittaker process \eqref{eq:deftwolayerWhittaker} as a Markov process with explicit transition density. The arguments are parallel to the Schur case treated in Section \ref{sec:MarkovianSchur}. 

 Define a probability density $\mathsf P_0(\la^0)$ for $\lambda^0\in \R^2$ by 
\begin{equation}
	\mathsf P_0(\la) = e^{-u(\la_1-\la_2)} H_{0,N}(\la_1-\la_2)\delta_0(\la_2),
	\label{eq:defp0Whittaker}
\end{equation}
where   the functions $H_{x,N}:\R\to \mathbb R_+$ are defined by 
\begin{equation}
	H_{x,N}(\ell)  = \frac{Q_{N-x}(\ell)}{\sum_{\ell\geq 0} e^{-u\ell} Q_N(\ell)}
	\label{eq:defH}
\end{equation} 
with, for $M\leq N$, 
\begin{equation*}
	Q_M(\ell) =\prod_{i=N-M+1}^N\frac{1}{\Gamma(\alpha_i)^{2}} \int_{\R^2}d\lambda  \Psi_{\alpha_{N-M+1}, \dots, \alpha_N} (\la/(\ell,0)) e^{-v(\la_1-\la_2)}. 
\end{equation*} 
Define further transition densities  
\begin{equation*}
	\mathsf P_{x,y}(\la, \mu) = \prod_{i=x+1}^y\frac{1}{\Gamma(\alpha_i)^{2}}\Psi_{\alpha_{x+1}, \dots, \alpha_y}(\mu/\la) \frac{H_{y,N}(\mu_1-\mu_2)}{H_{x,N}(\la_1-\la_2)}. 
	\label{eq:transitionkernelWhittaker}
\end{equation*}
Using the branching rule \eqref{eq:defskewWhittaker} one can verify that  $\mathsf P_{x,y}(\la, \mu)$ defines a probability density on $\mu\in \R^2$ for every given  $\la$ and $x<y$. The branching rule also implies that $\mathsf P_{x,y}$ satisfy a semi-group like property 
\begin{equation}
	\int_{\R^2}d\kappa \mathsf P_{x,y}(\la, \kappa) \mathsf P_{y,z}( \kappa, \mu) =  \mathsf P_{x,z}(\la, \mu).
\end{equation}
\begin{proposition}
	The process $(\la^x)_{x\in \llbracket 1,N\rrbracket}$ (defined by \eqref{eq:deftwolayerWhittaker}) is a time-inhomogeneous Markov chain on $\R^2$ with transition kernel $\mathsf P_{x,y}(\la^x, \la^y)$ and initial density $\mathsf P_0(\la^0)$.
	\label{prop:markoviandescriptionWhittaker}
\end{proposition}
\begin{proof}
Same as the proof of Proposition \ref{prop:markoviandescription}. 
\end{proof}

The Markovian description can  be generalized to the more general two-layer Whittaker processes in Definition \ref{def:generaltwolayerWhittaker}. Let 
$$ \mathsf P_x^{\rightarrow}(\lambda, \mu) = \frac{1}{\Gamma(\alpha_x)^2} \Psi_{\alpha_x}(\mu/\lambda)\frac{H_{x,N}(\mu_1-\mu_2)}{H_{x-1,N}(\lambda_1-\lambda_2)} $$
and 
$$ \mathsf P_x^{\downarrow}(\lambda, \mu) = \frac{1}{\Gamma(\alpha_x)^2} \Psi_{\alpha_x}(\lambda/\mu)\frac{H_{x,N}(\mu_1-\mu_2)}{H_{x-1,N}(\lambda_1-\lambda_2)} . $$
\begin{proposition}
	The process $x\mapsto \lambda^x$ from Definition \ref{def:generaltwolayerWhittaker}, associated with a word $w\in \lbrace \rightarrow, \downarrow\rbrace^N$, is a (time-inhomogeneous) Markov process with kernel $\mathsf P_x^{w_x}(\lambda^{x-1},\lambda^x)$ and initial distribution $\mathsf P_0(\la^0)$.
	\label{prop:MarkoviangeneralWhittaker}
\end{proposition}
\begin{proof}
	The proof is the same as that of Proposition \ref{prop:Markoviangeneral} except that one now uses the Cauchy and Littlewood identities for skew Whittaker functions, i.e. \cite[Lemma 3.5 and 3.6]{barraquand2023stationary}, instead of \eqref{eq:skewCauchysignatures} and \eqref{eq:skewLittlewoodsignatures}.
\end{proof}

\subsubsection{Application}
Let us first  establish an integral formula for the function $Q_M$ arising in the definition of the function $H_{x,N}$. 
\begin{proposition} 
For $u,v>0$, and any $\ell\in \R$,  we have
\begin{equation}
	Q_M(\ell) =\int_{\I\R} \frac{dz}{2\I\pi} \Psi_z(\ell,0)  \frac{\Gamma(v\pm z)}{2\Gamma(\pm 2 z)}\prod_{i=N-M+1}^N \frac{\Gamma(\alpha_i\pm z)}{\Gamma(\alpha_i)^2}.
	\label{eq:expressionQWhittaker}
\end{equation}
When $u<0$ or $v<0$, $Q_M(\ell)$ is given by the analytic continuation in $u$ or $v$ of the formula above. 
\label{prop:formulaQ}
\end{proposition}
\begin{proof}
We omit the proof as the argument is parallel to that of Proposition \ref{prop:formulaqgeneral} and the calculations are very similar as in the proof of Theorem \ref{th:formulaWhittaker}. 
\end{proof}

One can now study limits as $N$ goes to infinity. 
\begin{proposition}
	Assume $\alpha_i\equiv \alpha>0$ and $v>0$. As $M\to \infty$, 
	\begin{equation*}
			Q_M(\ell) \sim_{M\to\infty} \frac{\Gamma(v)^2}{2\sqrt{\pi} \left( \psi_1(\alpha) M\right)^{3/2} } \Psi_{(0,0)}(\ell,0)
		\end{equation*}
	where $\psi_1(z) = \partial_z^2\log\Gamma(z)$ is the trigamma function. 
	Further, as $N\to\infty$,
	\begin{equation}
			H_{x,N}(\ell)\xrightarrow[N\to\infty]{} \frac{\Psi_{(0,0)}(\ell,0)}{\Gamma(u)^2}.
		\end{equation}	
	\label{prop:limitQ}
\end{proposition} 
\begin{proof}
	In \eqref{eq:expressionQWhittaker}, one may perform the change of variables $z=i x/\sqrt{M}$, so that as $M$ going to infinity, 
	$$ Q_M(\ell) \sim_{M\to\infty}  \frac{\Psi_{(0,0)}(\ell,0) \Gamma(v)^2 }{\pi M^{3/2}} \int_{\mathbb R} x^2 e^{-\psi_1(\alpha)x^2}dx,$$
which leads to the expressions given above.
\end{proof}
The above proposition shows that as $N$ goes to infinity, the process $x\mapsto \la^x$ (defined by \eqref{eq:deftwolayerWhittaker}, in the case $\alpha_i\equiv a$, $u,v>0$) weakly converges to the Markov process with transition probability 
\begin{equation}
\overline{	\mathsf P}(\lambda, \mu) = \frac{H(\mu_1-\mu_2)}{H(\la_1-\la_2)} \frac{1}{\Gamma(\alpha)^2}\Psi_{\alpha}(\mu/\la)
	\label{eq:limitDoobtransformLG}
\end{equation}
where 
\begin{equation} 
	H(\ell) = \frac{2K_0(2e^{-\ell/2})}{\Gamma(u)^2}, 
	\label{eq:Htransform}
\end{equation}
and initial distribution 
$$ \overline{\mathsf P}_0(\la^0) = \frac{1}{\Gamma(u)^2}  e^{-u(\la_1^0-\la_2^0)}\Psi_{(0,0)}(\la_1^0-\la_2^0, 0) \mathds{\delta}_0(\lambda_2^0).$$
In the expression \eqref{eq:Htransform}, we have used \eqref{eq:WhittakerBessel} to express $\Psi_{0,0}(\ell,0)$ in terms of the Bessel K function.

 Furthermore, the factor $ \frac{1}{\Gamma(\alpha)^2}\Psi_{\alpha}(\mu/\la)$ in  \eqref{eq:limitDoobtransform} can be interpreted as the transition kernel for a couple of independent random walks $(\la_1^x, \la_2^x) $ with increments distributed as logarithms of $\mathrm{Gamma}^{-1}(\alpha)$ random variables, killed at time $x$ with probability  
 $$ \exp(-e^{-(\la_2^x-\la_1^{x-1})}).$$
  Then, $\overline{\mathsf P}(\la^{x},\la^{x+1})$ can be interpreted as the transition  probability of the same killed random walk conditioned to survive forever.  
\begin{remark}
	By performing analytic continuations, one could study the cases where $c_2>1$ or $c_1>1$ which should lead to a different behaviour following the same phase diagram as in \cite{barraquand2023stationaryloggamma}. This remark also applies to Section \ref{sec:MarkovianSchur}.
\end{remark}

\section{Application to the open KPZ equation} 
\label{sec:KPZ}
The open KPZ equation on the segment $[0,L]$, with boundary parameters $u,v\in \R$ is the stochastic PDE
\begin{equation}
	\begin{cases}
		\partial_t h(t,x) = \tfrac 1 2 \partial_{xx} h(t,x) + \tfrac 1 2 \left( \partial_x h(t,x) \right)^2 +\xi(t,x),\\
		\partial_x h(t,x)\big\vert_{x=0}=u,\\
		\partial_x h(t,x)\big\vert_{x=L}=-v,
	\end{cases}
	\tag{$\mathrm{KPZ}_{u,v}$}
	\label{eq:KPZ}
\end{equation}
where $\xi$ is a space-time white noise. 
As for the KPZ equation on $\R$, solutions can be defined through the Hopf-Cole transform. We say that $h\in C(\mathbb R_{\geq 0}\times [0,L],\R)$ is a solution to \eqref{eq:KPZ} with initial condition $h_0$ if we have $h(t,x)=\log Z(t,x)$ where $Z(t,x)$ solves the multiplicative noise stochastic heat equation 
\begin{equation}
	\begin{cases}
		\partial_t Z(t,x) = \tfrac 1 2 \partial_{xx} Z(t,x) +  Z(t,x)\xi(t,x),\\
		\partial_x Z(t,x)\big\vert_{x=0}=(u-1/2)Z(t,0),\\
		\partial_x Z(t,x)\big\vert_{x=L}=-(v-1/2)Z(t,L),
	\end{cases}
	\tag{$\mathrm{SHE}_{u,v}$}
	\label{eq:SHE}
\end{equation}
where we interpret the product $Z(t,x)\xi(t,x)$ in the Ito sense, with initial condition $Z(0,x)=\log h_0(x)$. 
We refer to  \cite{corwin2016open, parekh2017kpz} for details about what this equation means and for proofs of the existence, uniqueness and positivity of solutions. 

\subsection{High temperature limit} The open KPZ equation is the scaling limit of the weakly asymmetric open ASEP \cite{corwin2016open}, and is expected to be the limit of many other particle systems in a  weak asymmetric scaling. It is well-known that the KPZ equation on $\R$, without boundary,  also arises as the high temperature limit of the free energy of directed polymer models in $\mathbb Z^2$ \cite{alberts2014intermediate}. This is true as well for the KPZ equation on $\mathbb R_+$ and directed polymers confined to a half-quadrant \cite{wu2018intermediate} (see also \cite{parekh2019positive, barraquand2023stationaryloggamma}). Hence, it is natural to expect, though it has not been proved yet,  that the open KPZ equation arises  as the high temperature limit of directed polymers in the strip. 
In particular, a precise convergence statement was proposed   in \cite{barraquand2023stationary} for the free energy of the log-gamma polymer: fix some $L>0$ and scale 
\begin{equation}
	\alpha=\tfrac 1 2 +  \eps^{-1}, \;\; N=\eps^{-1}L, \;\; h^{(\eps)}(t,x) = (\eps^{-2}t+\eps^{-1} x)\log(1/\eps) + H\left(\frac{\eps^{-2}t}{2}+ \eps^{-1}x,\frac{\eps^{-2}t}{2}  \right). 
	\label{eq:scalings}
\end{equation}
Then we expect that for all $t>0$ and parameters $u,v\in \mathbb R$, as $\eps\to 0$, we have the convergence 
\begin{equation}
	h^{(\eps)}(t,x) \Rightarrow h(t,x)
	\label{eq:conjecture}
\end{equation} 
in the space $C([0,L],\R)$. We refer to \cite[Conjecture 4.2]{barraquand2023stationary} for more details about this conjecture and the specific form of the scalings \eqref{eq:scalings}.  

\subsection{Stationary measures for the open KPZ equation}
Under the same scalings as \eqref{eq:scalings}, it was shown in \cite{barraquand2023stationary} that the probability distribution $\PP^{\alpha,u,v}_{\rm LG}$ from Definition \ref{def:stationaryloggamma} on couples of random walks converges to a probability distribution on continuous functions $B_1,B_2$ on $[0,L]$ defined by 
\begin{equation} \PP^{u,v}_{\rm KPZ}(B_1,B_2) = \frac{1}{\mathcal K_{u,v}(L)} e^{(u+v)B_1 \otimes B_2(L)} \,\PP_{\rm Brown}^{-v}(B_1)\,\PP_{\rm Brown}^{-u}(B_2) 
	\label{eq:KPZstat1}
\end{equation}
where the reference measure $\PP_{\rm Brown}^{d}$ is the law of a Brownian motion with drift $d$ starting from $0$, and the composition $B_1\otimes B_2$ is the geometric Pitman transform defined as 
$$ B_1\otimes B_2 \,(t)  =   - \log \int_0^t e^{-(B_1(s)+B_2(t)-B_2(s))}ds,$$
The normalization function satisfies 
$$ \mathcal K_{u,v}(L)  = \mathbb E\left[ \left(\int_0^L e^{-(B_1(s)+B_2(L)-B_2(s))}ds \right)^{-u-v}\right],$$
where the expectation is taken over the law of independent standard Brownian motions $B_1,B_2$. Modulo \eqref{eq:conjecture}, this proves the conjecture from \cite{barraquand2021steady} that the law of $B_1$ under $ \PP^{u,v}_{\rm KPZ}$ is the stationary measure for the open KPZ equation.  

When $u+v\geq 0$, the result was already proven through a combination of results from \cite{corwin2021stationary,  bryc2021markov}. Indeed, in the  special case $u+v>0$, it was shown in \cite{barraquand2021steady} that the law of $(B_1(x), B_2(x))$ is the same as the law of the processes  $(\Lambda_1(x)-\Lambda_1(0), \Lambda_2(x)-\Lambda_2(0))$
where $\Lambda_1, \Lambda_2$ are continuous stochastic processes on $[0,L]$ distributed as 
\begin{equation}
	\mathbb P_{\rm KPZ}^{u,v} (\Lambda_1, \Lambda_2) =    \frac{1}{\mathcal Z_{u,v}(L)}  F_{u,v}(\Lambda_1, \Lambda_2) \delta_0(\Lambda_2(0))\,\, \mathbb P^{\rm free}_{\rm Brown}(\Lambda_1)\mathbb P^{\rm free}_{\rm Brown}(\Lambda_2) 
	\label{eq:KPZtwolayer}
\end{equation}
where the reference measure $\mathbb P^{\rm free}_{\rm Brown}$ is the Brownian measure with free endpoints (though the presence of the  factor $\delta_0(\Lambda_2(0))$ will fix $\Lambda_2(0)=0$), and the functional $F_{u,v}$ is 
$$F_{u,v}(\Lambda_1, \Lambda_2) = \exp\left(  -\int_0^L  e^{-(\Lambda_1(s)-\Lambda_2(s))}d s \right) e^{-u(\Lambda_1(0)-\Lambda_2(0))}e^{-v  (\Lambda_1(L)-\Lambda_2(L))}.$$
The normalization constants in \eqref{eq:KPZtwolayer} and in \eqref{eq:KPZstat1} are related by 
\begin{equation}
	\mathcal Z_{u,v}(L) = \Gamma(u+v)e^{\frac{-L}{2}(u^2+v^2)}  \mathcal K_{u,v}(L).
	\label{eq:normalizationKPZ}
\end{equation}
Indeed, to see that $(\Lambda_1(\cdot)-\Lambda_1(0), \Lambda_2(\cdot)-\Lambda_2(0))$ under \eqref{eq:KPZstat1} has the same law as $(B_1,B_2)$ under \eqref{eq:KPZtwolayer}, we first integrate over $\Lambda_1(0)-\Lambda_2(0)$, and then we arrange the formula using Girsanov theorem to change the drifts of the Brownian motions. The first averaging yields the factor $\Gamma(u+v)$, and the application of Girsanov's theorem yields the exponential factors in \eqref{eq:normalizationKPZ}. 

\medskip 

We insist that  \eqref{eq:KPZtwolayer} defines a probability measure only for $u+v>0$. It turns out that this probability measure can be related \cite{bryc2021markov, barraquand2021steady} to formulas characterizing the open KPZ stationary measures \cite{corwin2021stationary}, so that the law of $B_1$ (or $\Lambda_1(x)-\Lambda_1(0)$), is the stationary measure of the open KPZ equation. It is known to be unique \cite{knizel2022strong, parekh2022ergodicity}.

\medskip 
After noticing that the marginal distribution of $\Lambda_1(x)-\Lambda_1(0)$ in  \eqref{eq:KPZtwolayer} can be described by $B_1(x)$ in \eqref{eq:KPZstat1}, \cite{barraquand2021steady} conjectured that the law of $B_1$ is actually the stationary measure for all $u,v$. This conjecture was motivated by an analytic continuation argument which seems difficult to formalize in the continuous setting. The framework of \cite{barraquand2023stationary} allows to justify rigorously this analytic continuation procedure by going through a discretized variant of the KPZ equation, namely the log-gamma polymer (though this approach requires the convergence of one model to the other, which has not been established yet for the log-gamma polymer in a strip a strip). 
The idea consists in building discretized versions of the measure \eqref{eq:KPZtwolayer} (the two-layer Schur and Whittaker processes from Definition \ref{def:generaltwolayerSchur} and Definition \ref{def:generaltwolayerWhittaker}) and prove that: 
\begin{enumerate}
	\item These measures are stationary for certain Markov dynamics on sequences of signatures.
	\item Under these Markov dynamics,  the top layer, i.e. the process $x\mapsto \lambda_1^x$, evolves according to the stochastic recursion satisfied by the LPP time $G(n,m)$ or log-gamma polymer free energies $H(n,m)$.
	\item Analytic continuations in boundary parameters can be performed at the discrete level. 
\end{enumerate}

\subsection{Computation of the growth rate of the stationary open KPZ equation}
We now explain how the results from the present paper can be applied to the open KPZ equation \eqref{eq:KPZ}. 

Let, for $u,v>0$, 
\begin{equation}
	c_{u,v}(L) = 
	\frac{-1}{24} + \frac{1}{2} \partial_L \log \mathcal Z_{u,v}(L) 
	\label{eq:defcuv}
\end{equation}
where $\mathcal Z_{u,v}(L)$ -- the same renormalization constant as above in \eqref{eq:normalizationKPZ} -- can be computed explicitly  as 
\begin{equation}
	\mathcal Z_{u,v}(L) = \int_{\I\R}\frac{dz}{2\I\pi} \left\vert \frac{\Gamma(u+ z)\Gamma(v+ z) }{\Gamma(2z)}\right\vert^2 \frac{e^{z^2 L}}{2}.
	\label{eq:normalizationexplicit}
\end{equation} 
For other values of $u,v$, we define $c_{u,v}(L)$ by analytic continuation in the variables $u$ or $v$ from the formula in the  case $u,v>0$ (as we will explain, such analytic continuations are easy to perform taking residues in the integral formula \eqref{eq:normalizationexplicit}). The fact that the normalization constant in \eqref{eq:KPZtwolayer} admits the explicit expression \eqref{eq:normalizationexplicit} was already known \cite{bryc2021markov}, and the same expression already arises in \cite{corwin2021stationary}.

Let $h(t,x)$ be the solution to \eqref{eq:KPZ} starting from stationary initial data at $t=0$. Then, assuming that  the unproven convergence \eqref{eq:conjecture} holds in expectation,  the next Theorem shows that for any time $t\geq 0$, and for the solution $h(t,x)$ to  \eqref{eq:KPZ} started from its stationary measure, we have 
\begin{equation}
	\mathbb E\left[h(t,0) \right] = t\; c_{u,v}(L). 
\end{equation}
\begin{theorem}
	Assume that  $n$ is a multiple of $N$, and $u,v>0$. Consider the free energy of the log-gamma polymer (Definition \ref{def:loggamma}) starting from stationary initial data (Definition \ref{def:stationaryloggamma}). Then, 
	\begin{equation*}  
		\mathbb E\left[ H(n,n) \right] =  n\,\times \frac{-1 }{\mathcal Z_{\rm LG}^{\alpha,u,v}(N)} \int_{\I\R}\frac{dz}{2\I\pi} \frac{\Gamma(u\pm z)\Gamma(v\pm z)\left(\Gamma(\alpha \pm z)\right)^N}{2\Gamma(2z)\Gamma(-2z)} (\psi(\alpha+z)+\psi(\alpha-z)),
	\end{equation*}
	where $\psi(z)=\partial_z\log\Gamma(z)$ is the digamma function. Moreover, under the scalings \eqref{eq:scalings}, we have that for any $u,v\in \R$, 
	$$ \lim_{\eps\to 0} \mathbb E[h^{(\eps)}(t,0) ] = t \; c_{u,v}(L).$$
	\label{th:constantcuv}
\end{theorem}
Theorem \ref{th:constantcuv} is proved below in in Section \ref{sec:proofcuv}. 
\begin{remark}
	On the line $u+v=0$, the open KPZ equation stationary measure is a Brownian motion with drift $u$  and, after appropriate analytic continuation, one finds that 
	$$ c_{u,v}(L) = \frac{-1}{24} + \frac{u^2}{2}.$$
	This expression for the growth rate of the open KPZ equation could also be deduced from the convergence of open ASEP to the KPZ equation \cite{corwin2016open, parekh2017kpz} along with the well-known  fact that for a special choice of boundary conditions,  stationary measures of  open ASEP are Bernoulli i.i.d. 
\end{remark}

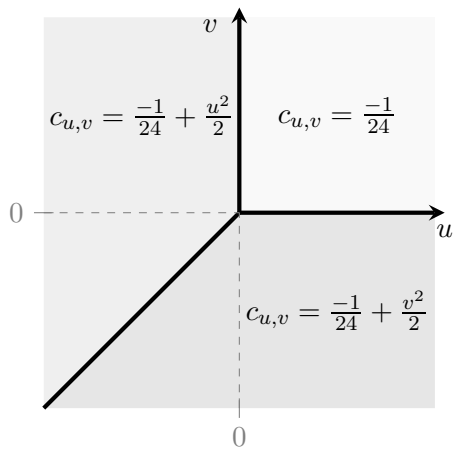
\begin{figure}
	\begin{center}
		\begin{tikzpicture}[scale=1.3, every text node part/.style={align=center}]
			\fill[gray!05] (2,2) -- (4,2) -- (4,4) -- (2,4) -- cycle;
			\fill[gray!20] (2,2) -- (4,2) -- (4,0) -- (0,0) -- cycle;
			\fill[ gray!13] (2,2) -- (2,4) -- (0,4) -- (0,0) -- cycle;
			\draw[ultra thick] (0,0) -- (2,2);
			\draw[ultra thick, -stealth] (2,2) -- (2,4.1) node[anchor=north east] {$v$ };
			\draw[ultra thick, -stealth] (2,2) -- (4.1,2) node[anchor=north] {$u$};
			\draw[dashed, gray] (2,0) -- (2,2);
			\draw[dashed, gray] (0,2) -- (2,2);
			\draw[gray] (2,0.1) -- (2,-0.1) node[anchor=north] {$0$};
			\draw[gray] (0.1,2) -- (-0.1,2) node[anchor=east] {$0$};
			\draw (3,3) node{ $c_{u,v} = \frac{-1}{24} $};
			\draw (1,3) node{ $c_{u,v} = \frac{-1}{24}+\frac{u^2}{2} $};
			\draw (3,1) node{ $c_{u,v} = \frac{-1}{24}+\frac{v^2}{2} $};
		\end{tikzpicture}
	\end{center}
	\caption{The phase diagram for $c_{u,v} = \lim_{L\to\infty} c_{u,v}(L)$}
	\label{fig:phasetransition}
\end{figure}
As $L$ goes to infinity the value of $c_{u,v}(L)$ should converge to the growth rate of the KPZ equation in full-space, that is $-1/24$. This is correct when the boundaries do not play too much role. We show that the asymptotics of $c_{u,v}(L)$ are governed by the phase diagram shown in Fig. \ref{fig:phasetransition}. 
\begin{corollary}
	We have the convergence
	$$ \lim_{L\to\infty} c_{u,v}(L) = \begin{cases}
		\frac{-1}{24} &\mbox{ when } u,v\geq 0,\\ 
		\frac{-1}{24} + \frac{u^2}{2}&\mbox{ when } u\leq 0 , v\geq u,\\ 
		\frac{-1}{24} + \frac{v^2}{2}&\mbox{ when } v\leq 0 , u\geq v.\\ 
	\end{cases}$$
	\label{cor:phasetransition}
\end{corollary}
Corollary \ref{cor:phasetransition} is proved in Section \ref{sec:proofphasetransition}.

\subsection{Proof of Theorem \ref{th:constantcuv}}
\label{sec:proofcuv}

	Consider the log-gamma directed polymer model from Definition \ref{def:loggamma}, starting from the stationary measure with from Definition \ref{def:stationaryloggamma}, with  $H(0,0)=0$. Let  $n=kN$  for some integer $k$. We can decompose the computation of $\mathbb E \left[ H(n,n) \right]$ as a sum of horizontal and vertical increments as in  Fig. \ref{fig:computationexpectation}.
	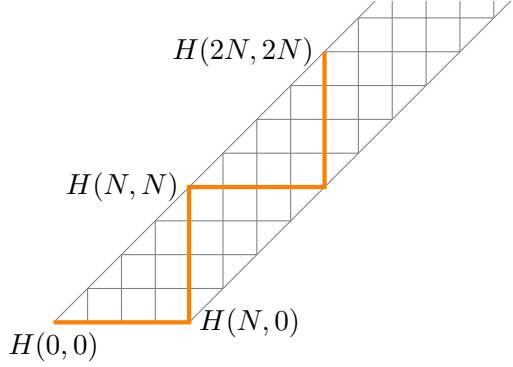
\begin{figure}
\begin{center}
	\begin{tikzpicture}[scale=0.45]
		\draw[gray] (0,0)--(9.5,9.5);
		\draw[gray] (4,0)--(13.5,9.5);
		\foreach \x in {0,1,...,9}
		\draw[gray] (\x,\x) -- (4+\x,\x); 
		\foreach \x in {0,1,...,5}
		\draw[gray] (4+\x, \x) -- (4+\x,4+\x);
		\draw[gray] (1,0) -- (1,1);
		\draw[gray] (2,0) -- (2,2);
		\draw[gray] (3,0) -- (3,3);
		\draw[gray] (10,6) -- (10,9.5);
		\draw[gray] (11,7) -- (11,9.5);
		\draw[gray] (12,8) -- (12,9.5);
		\draw[gray] (13,9) -- (13,9.5);
		\draw[ultra thick, orange] (0,0) -- (4,0) -- (4,4) -- (8,4) -- (8,8); 
		\draw (0,0) node[below]{$H(0,0)$};
		\draw (4,0) node[right]{$H(N,0)$};
		\draw (4,4) node[left]{$H(N,N)$};
		\draw (8,8) node[left]{$H(2N,2N)$};
	\end{tikzpicture}
\end{center}
\caption{Decomposition of $H(2N,2N)$ as a sum of horizontal and vertical increments.}
\label{fig:computationexpectation}
\end{figure}
This yields
\begin{align}	\mathbb E \left[ H(kN,kN) \right] &= k \, \times\,  \mathbb E \left[ H(N,N)-H(0,0) \right]\\
	&=  k\Big( \mathbb E \left[ H(N,N)- H(N,0)\right]+\mathbb E\left[H(N,0)-H(0,0) \right]\Big),
	\label{eq:decomposition}
\end{align}
where we have use stationarity in the first equality. Both expectations $\mathbb E \left[ H(N,N)- H(N,0)\right]$ and $\mathbb E\left[H(N,0)-H(0,0) \right]$ can be extracted from the description of the stationary measure. More precisely, we will use the fact that by symmetry, $H(N,N)- H(N,0)$ has the same distribution as $H(N,0)-H(0,0)$ provided one exchanges the roles of $u$ and $v$. 
Then, for $u,v>0$, $\mathbb E\left[H(N,0)-H(0,0) \right]$ can be computed as 
$$\mathbb E\left[H(N,0)-H(0,0) \right] = -\partial_t \,\mathbb E_{\rm LG}^{\alpha, u,v}\left[ e^{-tL_1(N)}\right]\bigg\vert_{t=0}.$$ 
Recall the formula  \eqref{eq:Laplacetransformspecialcaseintro}. We find  
\begin{multline*}
	-\partial_t \left( \Gamma(u-t\pm z)\Gamma(v+t\pm z)\left(\Gamma(\alpha+t\pm z) \right)^N\right)\Big\vert_{t=0}  =   -\Gamma(u\pm z)\Gamma(v\pm z)\left(\Gamma(\alpha \pm z) \right)^N \\  \times \left( -\psi(u+z)-\psi(u-z)+\psi(v+z)+\psi(v-z) + n \left(\psi(\alpha+z)+\psi(\alpha-z) \right)\right) , 
\end{multline*} 
where $\psi(z)=\partial_z \log \Gamma(z)$, so that 
\begin{multline}  	\mathbb E_{\rm LG}^{\alpha, u,v}\left[ L_1(N) \right] =   \frac{-1}{\mathcal Z_{\rm LG}^{\alpha,u,v}(N)} \int_{\I\R}\frac{dz}{2\I\pi} \frac{\Gamma(u\pm z)\Gamma(v\pm z)(\Gamma(\alpha\pm z))^N}{4\Gamma(2z)\Gamma(-2z)}\\ \times \left(-\psi(u+z)-\psi(u-z)+\psi(v+z)+\psi(v-z)+n(\psi(\alpha+z)+\psi(\alpha-z)) \right).
\end{multline}
Hence, using \eqref{eq:decomposition}, 
\begin{multline}  
\mathbb E\left[H(n,n) \right] =     \frac{-n }{\mathcal Z_{\rm LG}^{\alpha,u,v}(N)} \int_{\I\R}\frac{dz}{2\I\pi} \frac{\Gamma(u\pm z)\Gamma(v\pm z)\prod_{i=1}^N\Gamma(\alpha_i\pm z)}{2\Gamma(2z)\Gamma(-2z)} (\psi(\alpha+z)+\psi(\alpha-z)),
\label{eq:formulaexpectationH}
\end{multline}
which is the claimed formula. 

Now we use the scalings \eqref{eq:scalings} in \eqref{eq:formulaexpectationH}. For $\alpha=\frac 1 2 +\eps^{-1}$, using Stirling approximation, we have 
$$ \left( \Gamma(\alpha+z)\Gamma(\alpha-z)\right)^{\eps^{-1}L}  \sim f(\eps) e^{z^2 L}$$
for some function  $f(\eps)$  and 
$$\frac{\eps^{-2}t}{2}\left(\psi(\alpha+z)+\psi(\alpha-z)\right) - t\eps^{-2} \log(\eps^{-1}) \xrightarrow[\eps\to 0]{} \frac{t}{24}(1-12z^2).$$ 
The expression of $f(\eps)$ can be determined exactly, but it does not matter, since $f(\eps)$ will eventually cancel, being  present in the integral in  \eqref{eq:formulaexpectationH} and in the normalization constant. Combining the two limits above, one obtains 
that 
\begin{equation*}
	\lim_{\eps\to 0}\frac{1}{f(\eps)} \mathcal Z_{\rm LG}^{\alpha,u,v}(\eps^{-1}L) = \mathcal Z_{u,v}(L) := \int_{\I\R}\frac{dz}{2\I\pi} \frac{\Gamma(u\pm z)\Gamma(v\pm z)e^{z^2 L}}{2\Gamma(2z)\Gamma(-2z)}
\end{equation*}
and 
\begin{equation}  
\lim_{\eps\to 0}\mathbb E\left[ h_{\eps}(t,0) \right] =   \frac{-t }{\mathcal Z_{u,v}} \int_{\I\R}\frac{dz}{2\I\pi} \frac{\Gamma(u\pm z)\Gamma(v\pm z)e^{z^2 L}}{2\Gamma(2z)\Gamma(-2z)} \frac{1-12 z^2}{24}.
\label{eq:expectationlimit}
\end{equation}
Strictly speaking, in order to justify that the integral in \eqref{eq:formulaexpectationH} indeed converges to the integral of the pointwise limit of the integrand, one needs to bound the Gamma and digamma factors up to a sufficient order and apply the dominated convergence theorem. Appropriate bounds for Gamma and digamma functions are available in the literature, and their apprlication poses no difficulty, so we omit the details. 
Finally, we recognize in \eqref{eq:expectationlimit} the quantity $t \,c_{u,v}(L)$, which concludes the proof when $u,v>0$.

It remains to explain why the formula still holds for any $u,v$, where the value of $c_{u,v}(L)$ is defined as the analytic continuation of its expression when $u,v>0$. 

By Theorem \ref{theo:formulaloggammaintro}, when $u<t<-v $, 
$$\mathbb E_{\rm LG}^{\alpha, u,v}\left[ e^{-2 tL_1(N)}\right] = \frac{\mathcal I(\alpha, u,v,N,t)}{\mathcal I(\alpha, u,v,N,0)}$$
where 
$$\mathcal I(\alpha, u,v,N,t)=\int_{\I\R}\frac{dz}{2\I\pi} F_{\alpha, u,v,N,t}(z)$$ and 
$$ F_{\alpha, u,v,N,t}(z) = \frac{\Gamma(u-t\pm z)\Gamma(v+t\pm z)\left(\Gamma(\alpha+t\pm z) \right)^N}{2\Gamma(2z)\Gamma(-2z)}.$$ 
 If $u-t<0$ or $v+t<0$, Theorem \ref{theo:formulaloggammaintro} says that the value of $\mathbb E_{\rm LG}^{\alpha, u,v}\left[ e^{-2 tL_1(N)}\right]$ is the analytic continuation of its expression above, which is  is given by the following. 
\begin{multline} 
	\mathcal I(\alpha, u,v,N,t)=\int_{\I\R+\epsilon}\frac{dz}{2\I\pi} F_{\alpha, u,v,N,t}(z) 
	\\ + \sum_{i\in \mathbb Z_{\geq 0}; u-t+i\leq 0} \mathrm{Res}_{z=t-u-i} \lbrace F_{\alpha, u,v,N,t}(z) \rbrace - \mathrm{Res}_{z=-t+u+i} \lbrace F_{\alpha, u,v,N,t}(z) \rbrace  \\ 
	+ \sum_{i\in \mathbb Z_{\geq 0}; v+t+i\leq 0} \mathrm{Res}_{z=-t-v-i} \lbrace F_{\alpha, u,v,N,t}(z) \rbrace - \mathrm{Res}_{z=t+v+i} \lbrace F_{\alpha, u,v,N,t}(z) \rbrace.
	\label{eq:analyticcontinuation}
\end{multline}
In the first integral, we chose the contour to be $\I\R+\epsilon$ for some small $\epsilon >0$. This is necessary only to cover the cases where $u-t$ or $v+t$ is a negative integer, otherwise the contour may be $\I\R$. 
 
At this point, one can reproduce all the steps and limits performed above. Eventually, one finds that \eqref{eq:expectationlimit} still holds, provided one replaces the integral operator $f\mapsto \int_{\I\R}\frac{dz}{2\I\pi}f(z)$ by 
\begin{multline} 
	 f\mapsto \int_{\I\R+\epsilon}\frac{dz}{2\I\pi} f(z) + \sum_{i\in \mathbb Z_{\geq 0}; u-t+i\leq 0} \mathrm{Res}_{z=t-u-i} \lbrace f(z) \rbrace - \mathrm{Res}_{z=-t+u+i} \lbrace f(z) \rbrace\\ 
+ \sum_{i\in \mathbb Z_{\geq 0}; v+t+i\leq 0} \mathrm{Res}_{z=t+v+i} \lbrace f(z)  \rbrace - \mathrm{Res}_{z=-t-v-i} \lbrace f(z)  \rbrace.
\label{eq:deformedintegral}
\end{multline}
This shows that the limit $\lim_{\eps\to 0}\mathbb E\left[ h_{\eps}(t,0) \right]$ is given by the analytic continuation of the expression for $t\,c_{u,v}(L)$ when $u,v>0$, and it concludes the proof of Theorem \ref{th:constantcuv}.

\begin{remark}
	The formula from Theorem \ref{th:constantcuv} can be rewritten as 
	$$ \mathbb E\left[H(n,n)\right] = -n \,\partial_{\alpha_1} \log\left( \mathcal Z_{\rm LG}^{\boldsymbol\alpha,u,v}(N) \right)\Big\vert_{\alpha_1=\dots=\alpha_N=\alpha}. $$
	Given the definition of the normalization constant  in  \eqref{eq:deftwolayerWhittaker}, this implies that $ \mathbb E\left[H(n,n)\right] = n\mathbb E\left[\mathbf L_1(1)+\mathbf L_2(1)\right]$. This identity could have been guessed, since it can be argued that $\mathbf L_2(1)$ has the same distribution as the vertical increment of free energy  $H(k,k)-H(k,k-1)$ (in the model started from stationary initial condition $H_0$). Actually, this provides an alternative route for proving Theorem \ref{th:constantcuv}.
\end{remark}

\subsection{Proof of Corollary \ref{cor:phasetransition}}
\label{sec:proofphasetransition}
First assume that $u,v>0$. Using the change of variables $z=\tilde z/\sqrt{L}$, it is easy to see that as $L\to\infty$, 
$$ \mathcal Z_{u,v}(L) \sim  \frac{-\Gamma(u)^2\Gamma(v)^2}{L^{3/2}} \int_{\I\R}\frac{dz}{2\I\pi} 2z^2e^{z^2} = \frac{\Gamma(u)^2\Gamma(v)^2}{2\sqrt{\pi}L^{3/2}},$$
and 
$$ \partial_L \mathcal Z_{u,v}(L) \sim  \frac{-\Gamma(u)^2\Gamma(v)^2}{L^{5/2}} \int_{\I\R}\frac{dz}{2\I\pi} 2z^4e^{z^2}=\frac{-3\Gamma(u)^2\Gamma(v)^2}{4\sqrt{\pi}L^{5/2}}.$$
Hence, we find that 
$$c_{u,v}(L)=\frac{-1}{24} - \frac{3}{4L}+o\left(\frac{1}{L}\right),$$ 
which converges to $-1/24$ as claimed.

For other values of $u,v$, we need to compute the analytic continuation of $c_{u,v}(L)$. We note that the partition function $\mathcal Z_{u,v}(L)$ is analytic only when $u+v>0$ (but the singularity at $u+v=0$ cancels in the ratio $\partial_L\mathcal Z_{u,v}(L)/\mathcal Z_{u,v}(L)$ which is analytic for all $u,v$). In view of \eqref{eq:normalizationKPZ}, it is convenient to set $\tilde{\mathcal Z}_{u,v}(L)= \mathcal Z_{u,v}(L)/\Gamma(u+v)$ so that we still have 
$$ c_{u,v}(L)=\frac{-1}{24}+ \frac{1}{2} \log \tilde{\mathcal Z}_{u,v}(L) $$
and we can work with the analytic continuation of $\log \tilde{\mathcal Z}_{u,v}(L)$, which is given by 
\begin{multline}
	  \tilde{\mathcal Z}_{u,v}(L) =  \frac{1}{\Gamma(u+v)}\int_{\I\R}\frac{dz}{2\I\pi} F(z) 
	 + \frac{1}{\Gamma(u+v)}\sum_{i\in \Z_{\geq 0}; u+i<0} \mathrm{Res}_{z=-u-i}\lbrace F(z)\rbrace - \mathrm{Res}_{z=u+i}\lbrace F(z)\rbrace \\
	 + \frac{1}{\Gamma(u+v)}\sum_{i\in \Z_{\geq 0}; v+i<0} \mathrm{Res}_{z=-v-i}\lbrace F(z)\rbrace - \mathrm{Res}_{z=v+i}\lbrace F(z)\rbrace
\end{multline}
with 
$$ F(z) =  \frac{\Gamma(u\pm z)\Gamma(v\pm z)e^{z^2 L}}{2\Gamma(2z)\Gamma(-2z)}.$$
Computing explicitly the residues, we find that when $u\leq 0$ and $u<v$, the asymptotic behaviour of $\tilde{\mathcal Z}_{u,v}(L)$ as $L\to\infty$ is dominated by the residues at $z=\pm u$. 
We find that 
\begin{equation}
	 \frac{1}{\Gamma(u+v)} \mathrm{Res}_{z=-u}\lbrace F(z)\rbrace \sim_{L\to\infty}  \frac{\Gamma(v-u)}{2\Gamma(-2u)}e^{u^2L},
	 \label{eq:residuecomputed}
\end{equation}
so that 
\begin{equation}
	\frac{\partial_L \tilde{\mathcal Z}_{u,v}(L)}{\tilde{\mathcal Z}_{u,v}(L)} \xrightarrow[L\to\infty]{} u^2,
	\label{eq:limitratio}
\end{equation} 
and 
$c_{u,v(L)}$ converges to $-1/24+u^2/2$ as claimed. From the expression of the residue in \eqref{eq:residuecomputed}, it seems that there may be a singularity when $u=v<0$. This is actually not the case if one takes into account both residues at $\pm u$ and $\pm v$. Although each residue has a singularity at $u+v$, their sum can be analytically continued when $u=v$ via l'Hôpital's rule. One finds that \eqref{eq:limitratio} still holds. 
Finally, the case $v\leq 0, v<u$ is treated similarly, exchanging the roles of $u$ and $v$.

\renewcommand{\emph}[1]{\textit{#1}}
\bibliography{mainbiblio.bib}
\bibliographystyle{goodbibtexstyle} 
\end{document}